\documentclass[11pt]{article}
\usepackage{amsfonts,latexsym,amsmath,amscd,geometry}
\usepackage{graphicx}
\usepackage{amssymb}
\usepackage{latexsym}
\usepackage{xcolor}
\usepackage{indentfirst}

\usepackage{amssymb,mathrsfs,mathtools,graphicx,enumerate,color,colortbl}
\usepackage{hyperref}
\usepackage{comment}

\newcommand{\Dcal}{\mathcal{D}}
\newcommand{\Ecal}{\mathcal{E}}

\newcommand{\Hcal}{\mathcal{H}}

\newcommand{\Jcal}{\mathcal{J}}

\newcommand{\Scal}{\mathcal{S}}

\newcommand{\Xcal}{\mathcal{X}}

\renewcommand{\hbar}{\bar{h}}

\newcommand{\ubar}{\bar{u}}

\newcommand{\lbar}{\bar{\lambda}}

\newcommand{\util}{\tilde{u}}

\renewcommand{\a}{\alpha}
\renewcommand{\b}{\beta}

\renewcommand{\d}{\delta}

\newcommand{\et}{\eta}
\renewcommand{\th}{\theta}

\renewcommand{\l}{\lambda}

\newcommand{\m}{\mu}
\newcommand{\n}{\nu}
\newcommand{\x}{\xi}

\renewcommand{\r}{\rho}
\newcommand{\s}{\sigma}
\renewcommand{\t}{\tau}

\newcommand{\e}{\varepsilon}

\newcommand{\RR}{{\mathbb R}}
\newcommand{\NN}{{\mathbb N}}

\newcommand{\abs}[1]{\left|#1\right|}

\newcommand{\norm}[1]{\left\|#1\right\|}

\newcommand{\one}[1]{\mathbf{1}_{\{#1\}}}

\newcommand{\bci}[1]{\noindent\textbf{Control of \(#1\):}} 
\newcommand{\step}[1]{\vskip0.2cm \noindent{\it Step #1:} }
\newcommand{\case}[1]{\noindent{\it Case #1:} }

\newcommand{\uun}{\underline{u}^\nu}

\DeclareMathOperator{\supp}{supp}

\definecolor{black}{rgb}{0.0, 0.0, 0.0}
\definecolor{red}{rgb}{1.0, 0.5, 0.5}

\newcommand{\rd}{\partial}

\newcommand{\uubar}{\underline{u}}

\newcommand \nc{\newcommand}
\newtheorem{theorem}{Theorem}[section]
\newtheorem{lemma}[theorem]{Lemma}
\newtheorem{proposition}[theorem]{Proposition}

\newtheorem{remark}[theorem]{Remark}

\nc{\ba}{\begin{array}}\nc{\ea}{\end{array}}
\nc{\be}{\begin{eqnarray}}\nc{\ee}{\end{eqnarray}}
\nc{\beq}{\begin{equation}}\nc{\eeq}{\end{equation}}
\nc{\bex}{\begin{eqnarray*}}\nc{\eex}{\end{eqnarray*}}
\nc{\btm}{\begin{theorem}} \nc{\etm}{\end{theorem}}
\nc{\blm}{\begin{lemma}} \nc{\elm}{\end{lemma}}
\nc{\R}{\mathbb{R}}  

\newcommand \qed {\hfill $\Box$}

\newcommand{\subjclass}[1]{\par\medskip\noindent 2020 \textit{Mathematics Subject Classification}. #1.}
\newcommand{\keywords}[1]{\par\medskip\noindent \textit{Key words and phrases}. #1.}

\begin{document}

\title{Uniform Stability of Oscillatory Shocks for KdV-Burgers Equation}

\author{Geng Chen
\footnote{Department of Mathematics, University of Kansas, Lawrence, KS 66045, U.S.A. Email: gengchen@ku.edu},
\quad
Namhyun Eun
\footnote{School of Mathematics, Korea Institute for Advanced Study (KIAS), Seoul 02455, Republic of Korea. Email: namhyuneun@kias.re.kr},
\quad
Moon-Jin Kang
\footnote{Department of Mathematical Sciences, Korea Advanced Institute of Science and Technology, Daejeon 34141, Republic of Korea. Email: moonjinkang@kaist.ac.kr},
\quad
Yannan Shen
\footnote{Department of Mathematics, University of Kansas, Lawrence, KS 66045, U.S.A. Email: yshen@ku.edu}
}



\maketitle

\begin{abstract}
We study viscous-dispersive shock waves with infinite oscillations of the Korteweg--de Vries--Burgers (KdVB) equation.
First, we establish detail structures of the shock waves, including the rates at which the local extrema converge to the left end state towards the left far field.
Then, by exploiting the structural properties of the shocks, we show the \(L^2\)-contraction property of the shock profiles under arbitrarily large perturbations, up to time-dependent shifts.
This property implies both time-asymptotic stability and uniform stability with respect to the viscosity and dispersion coefficients.
This uniformity yields the existence of zero viscosity-dispersion limits, on which Riemann shocks are orbitally stable.


\end{abstract}

\section{Introduction}
\setcounter{equation}{0}
We consider the following form of Korteweg--de Vries--Burgers (KdVB) equation: 
\begin{equation} \label{burgers}
u_t + \Big(\frac{u^2}{2}\Big)_x = \e u_{xx}-\d u_{xxx}, \qquad u(0,x)=u_0(x)
\end{equation}
where \(u=u(t,x)\) is the unknown defined on \(t>0\) and \(x\in\RR\), and two constants \(\e,\d>0\) represent the viscosity and dispersion coefficients, respectively.
This model constitutes one of the most fundamental equations incorporating the combined effects of nonlinearity, viscosity (dissipation), and dispersion, and thus plays a central role in the study of numerous physical systems.
In particular, the KdVB equation serves as a fundamental model of waves on shallow water surfaces, which was first introduced (without dissipation) by Boussinesq \cite{B1872,B1877} and Korteweg--de Vries \cite{KdV1895}.
The form in \eqref{burgers} was later derived by Su--Gardner \cite{SG-JMP} (see also \cite{Johnson1,Johnson2,OS70}).
For modelling shallow water waves, \(u\) is the height displacement of the water surface from its equilibrium height.
Moreover, the KdVB equation can be regarded as a reduced prototype of the Navier--Stokes--Korteweg system \cite{Dunn}, capturing both viscosity and capillarity effects.

\vspace{2mm}
The equation \eqref{burgers} is known for its traveling wave solutions describing viscous-dispersive shocks with infinite oscillations \cite{bona}.
The main objective of this paper is to establish the stability of the oscillatory viscous-dispersive shocks, under arbitrarily large \(H^1\) perturbations.
The shock for \eqref{burgers} can be described as a traveling wave solution \(\util(\x)=\util(x-\s t)\), which satisfies
\begin{equation} \label{visS}
\begin{cases}
-\s \util' + (\frac{1}{2}\util^2)' =\e\util'' -\d\util''',\\
\lim_{\x\to\pm\infty} \util(\x) = u_\pm,\\
\lim_{\x\to\pm\infty} \util'(\x) = 0,
\end{cases}
\end{equation}
where \(u_->u_+\) and the traveling wave speed \(\s\) is given by 
\begin{equation} \label{sigma}
\s = \frac{\frac{1}{2}u_+^2-\frac{1}{2}u_-^2}{u_+-u_-} = \frac{u_-+u_+}{2}
\end{equation}
which aligns with the Rankine--Hugoniot condition for shock speed of Burgers equation.

\vspace{2mm}
The existence and stability of traveling wave solutions were first investigated by numerical methods at a very early stage \cite{CG-Numeric}.
Several years later, the existence of traveling waves was established by Bona--Schonbek \cite{bona} in 1985, where it was further shown that the behaviors of the viscous-dispersive shocks \(\util\) described above highly depend on the critical parameter \(2\d(u_--u_+)/\e^2\).
In fact, when \(2\d(u_--u_+)/\e^2\leq 1\), i.e., when the viscosity dominates the dispersion effect, the linearized equation of \eqref{visS} around the equilibrium state \(\util=u_-\) has real eigenvalue(s).
So the traveling wave solution \(\util\) is monotonically decreasing.
In 1985, Pego \cite{Pego} showed the time-asymptotic stability of monotonic shock profiles.
Since monotonic viscous-dispersive shock profiles share qualitative structures with viscous shock profiles, the anti-derivative method by Goodman \cite{Goodman}, Kawashima--Matsumura \cite{KM-1985} and Matsumura--Nishihara \cite{MN-1985} for viscous shocks was adopted in \cite{Pego}.
The approach developed in the present paper can also be applied to improve \cite{Pego} by allowing arbitrarily large perturbations.
However, to avoid distraction, we will address this simpler case in a separate paper.
More broadly, the stability of traveling waves has long been a central topic in various contexts (see, e.g., \cite{Mat-HB,Sand-HB}).
We refer to classical stability results \cite{Zumbrun2, Sattinger, Zumbrun1}.
We also mention an interesting work of Carter--de Rijk--Sandstede \cite{CdRS} on the stability of traveling pulses with oscillatory tails.

\vspace{2mm}
The much more interesting case is when the dispersion dominates, i.e., when \(2\d(u_--u_+)/\e^2> 1\).
In this case, the non-monotone traveling dispersive shock wave solutions \(\util\) oscillate around \(\util=u_-\) infinitely many times as \(\x \to -\infty\).

Many different methods have been used to deal with the fast oscillations in general.
In \cite{Wh}, Whitham derived a hyperbolic system with Riemann invariants for the short-wavelength nonlinear oscillations and Gurevich--Pitaevskii \cite{GP} patched it with Riemann solutions of zero dispersion limits by two moving boundaries out of the transition area.
Another well-known framework for the zero-dispersion limit was established by Lax and Levermore, in which they used the method of inverse scattering and spectral approximation \cite{LL-1,LL-2,LL-3}.
The study was also extended to more general flux, for example the rigorous justification of the Whitham modulation equations for the generalized KdV--Burgers equation was obtained in \cite{JZ}.
The generalized KdV equation with flux \(u^p\) and $\epsilon=0$, is linearly unstable when \(p>5\),  \(H^1\) stable under small perturbation when \(2<p<5\), as shown in \cite{PegoWeinstein}. 
See also the references in \cite{PegoWeinstein} for other results for the generalized KdV--Burgers equation.

For the KdVB equation, the stability of oscillatory viscous-dispersive shocks is a major open problem.
Khodja \cite{Khodja} and Naumkin--Shishmarev \cite{Naumkin} considered the stability of oscillatory shocks using a perturbative approach.
Their analysis builds upon the anti-derivative method for monotone shocks \cite{Goodman,Pego}, where a small anti-derivative of the perturbation is required and the \(L^\infty\) norm of the perturbation was proved to decay to zero.
However, as the argument is perturbative around monotonic shock profiles, the results appear to be restricted to profiles that are close enough to the monotone regime, and thus do not seem to capture the genuinely nonlinear oscillatory setting.

In a recent work \cite{Hur}, Barker--Bronski--Hur--Yang studied the contraction of large \(L^2\)-perturbation for oscillatory shocks of the KdVB equation, under some spectral assumption.
More precisely, their starting point is to observe the time-evolution of $L^2$-perturbation with time-dependent shifts (or temporal modulations) as in Lemma \ref{lem:RHS}, and show that if the Schr\"odinger operator associated with the last two terms of \eqref{key1} has precisely one bounded state then the shock is orbitally stable.
To find a sufficient condition for the spectral assumption, a computer-assisted analysis was used.
Their sufficient condition boils down to the assumption on \(2\d(u_--u_+)/\e^2\), that is, they assume $\delta\in(0.2533,3.9)$ with the normalization $u_\pm = \mp 1$ and $\e=1$ for the oscillatory regime.

\vspace{2mm}
In the present paper, we rigorously prove the \(L^2\) stability of the oscillatory shocks, in particular, \(1<2\d(u_--u_+)/\e^2<2\).
Here, the upper bound \(2\) for the critical parameter \(2\d(u_--u_+)/\e^2\) is not the optimal bound of our framework, but is enough to show the strength of our method to go beyond the threshold to the regime of non-monotone dispersive shocks.
We do not attempt to obtain the optimal upper bound, in order to keep the proofs accessible to the readers. 

Another strength of our stability result is that it yields a uniform estimate with respect to the viscosity and dispersion coefficients.
This uniformity enables us to justify the zero viscosity-dispersion limit of \eqref{burgers} to the (inviscid) Burgers equation.
Our result also fills the remaining gap in \cite{Hur} by covering the range \(\d\in(0.25,0.2533)\).

\vspace{2mm}
To establish the desired stability, we provide several detailed structural properties of the shock profiles \(\util\).
These structures are highly useful for future studies on viscous-dispersive shocks.
The rich dynamics of dispersive shocks under many types of flux (convex or nonconvex) were explored and experimented, see \cite{Heofer} and references therein.
The dispersive shocks also show up in various physical contexts, such as plasma physics \cite{Biskamp}, traffic flow \cite{Lighthill}, optical fibers \cite{Xu}, nematic liquid crystals \cite{Smyth}, quantum fluids  \cite{HoeferAblowitz} and nonlinear dynamical lattices \cite{Biondini}.
We expect that the methodology developed here can be extended to many of these contexts.

\subsection{Main Results}
We now introduce the main results of the present paper.
To state the first theorem, we introduce the following function class to describe perturbations:
\[
\Xcal_T \coloneqq
\{u\colon\RR^+\times\RR\to\RR \mid u-\uubar \in C([0,T];H^1(\RR)) \cap L^2(0,T;H^2(\RR))\}
\]
where \(\uubar\) is a smooth monotone function which satisfies
\begin{equation} \label{def:ubar}
\uubar(x) \coloneqq
\begin{cases}
u_-, &x\le-1,\\
u_+, &x\ge+1.
\end{cases}
\end{equation}

We present the first theorem, which is on the \(L^2\)-contraction properties of oscillatory viscous-dispersive shock waves without any size restrictions on perturbations.
This implies not only time-asymptotic stability but also uniform stability.

\begin{theorem} \label{thm:main}
Let \(\e,\d>0\) be constants and let \(u_\pm\) be given states.
Assume that \(u_->u_+\) and 
\begin{equation} \label{cond1}
\frac{1}{4} < \frac{\d(u_--u_+)}{2\e^2} < \frac{1}{2}.
\end{equation}
Let \(\util\) be the associated non-monotonic viscous-dispersive shock of \eqref{burgers} which connects \(u_-\) and \(u_+\).
Let \(u_0\) be given initial data satisfying \(\norm{u_0-\util}_{H^1(\RR)}<+\infty\) and let \(u\in\Xcal_T\) denote the solution of \eqref{burgers} with the initial data \(u_0\).
Then, the following \(L^2\)-contraction holds:
\begin{equation} \label{L2-stable}
\begin{aligned}
&\int_\RR \big(u(t,x)-\util^{-X}(x-\s t)\big)^2 dx
+\frac{2(u_--u_+)}{M}\int_0^T |\dot{X}(t)|^2 dt\\
&\qquad +\frac{\e}{5}\int_0^T \int_\RR \big((u(t,x)-\util^{-X}(x-\s t))_x\big)^2 dx dt
\le \int_\RR \big(u_0(x)-\util(x)\big)^2 dx.
\end{aligned}
\end{equation}
for the Lipschitz shift function \(X(t)\) satisfying \(X(0)=0\),
\begin{equation} \label{X-def}
\dot{X}(t)= \frac{2M}{(u_--u_+)} \int_\RR \big(u(t,x+X(t))-\util(x-\s t)\big) \, \util'(x-\s t)dx,
\end{equation}
where \(M\ge\frac{4}{3}\) is any fixed constant.
Moreover, the following time-asymptotic stability holds:
\begin{equation} \label{timeasy}
\norm{u(t,\cdot)-\util(\cdot-X(t))}_{L^p(\RR)} \to 0 \quad \text{as } t \to \infty,
\quad \text{for all} \quad 2<p\le\infty,
\end{equation}
and
\begin{equation} \label{shiftasy}
|\dot{X}(t)| \to 0 \quad \text{as } t \to \infty.
\end{equation}
\end{theorem}

\begin{remark}
(1) The choice of the shift \eqref{X-def} is inspired by earlier works \cite{Hur,KangVasseur}.
This type of shift function was first introduced in the study of the stability of viscous shocks for the viscous Burgers equation \cite{KangVasseur}, and was later employed in the analysis of viscous-dispersive shocks of the KdVB equation, where their stability was shown via a computer-assisted proof \cite{Hur}.
As explained in detail in \cite{Hur}, the choice of such a shift could be understood from a perspective of a gradient flow designed to minimize the \(L^2\)-distance between the solution and the shock profiles.\\
(2) The definition of the shift depends on the regularity of underlying solutions.
If \(u\in\Xcal_T\), then the Cauchy--Lipschitz theorem guarantees that the shift is well-defined and, moreover, Lipschitz continuous (see \cite[Remark 1.1]{KangVasseur} and \cite[Appendix C]{KV-JDE22}).
However, to the best of our knowledge, the existence of solutions in \(\Xcal_T\) has not been established in the literature, particularly in the case of distinct asymptotic states.
For the sake of completeness, we provide its proof in Appendix \ref{sec:exist}.
The work \cite{Hur} assumes the global existence of solutions evolving from \(H^2\) data.
We also refer to a series of works \cite{MR1, MR2, MV} concerning the existence of low-regularity solutions. These works build on the ideas introduced by Bourgain in \cite{Bourgain}.
See also \cite{KPV-KDV,ST-KDV} for classical results on the KdV equation.\\
(3) The contraction estimate \eqref{L2-stable} remains valid with the numerator \(M\) in place of \(2M\) in \eqref{X-def}.
We introduce the factor \(2\) so as to obtain \(L^2\)-dissipation for the shift derivative.
This plays a crucial role in the proof of Theorem \ref{thm:limit}, which concerns the zero viscosity-dispersion limits.
In particular, it is used to establish suitable compactness for the family of shift functions.\\
(4) The shift function \(X(t)\) is shown to satisfy the time-asymptotic behavior described in \eqref{shiftasy}, which implies 
\[
\frac{X(t)}{t} \to 0 \quad \text{as } t\to\infty.
\]
In particular, the shift grows at most sublinearly in time, and the shifted shock wave asymptotically retains the original traveling wave profile.
\end{remark}

\begin{remark} \label{rmk:wlog}
Throughout this paper, we always assume that \(\d>0\).
If \(\d<0\), one can perform the change of variables \(x\to -x,\ u\to -u,\ \d\to -\d\) to obtain the same result.
Owing to the scaling \(u(\e t,\e x)\) and \(\util(\e(x-\s t))\), we may also assume \(\e=1\).
The scaling invariance is discussed in detail in Section \ref{sec:scaling}.
Lastly, from the Galilean invariant transformation \(u(t,x-\s t)+\s \to u(t,x)\), we may assume \(\s=0\).
\end{remark}

As an important application of the above result, we now state the following result on the zero viscosity-dispersion limits.
To this end, we introduce the scaled equation which is given as follows:
\begin{equation} \label{scaled-eq}
(u^\n)_t + \Big(\frac{(u^\n)^2}{2}\Big)_x = \n \e (u^\n)_{xx}- \n^2 \d (u^\n)_{xxx}, \qquad
u^\n(0,x)=u_0\Big(\frac{x}{\n}\Big).
\end{equation}
Notice that if \(u^\n(t,x)\) is a solution to \eqref{scaled-eq}, the function \(u(t,x)=u^\n(\n t,\n x)\) is a solution to \eqref{burgers}.
Moreover, the viscous-dispersive shock of the scaled equation \eqref{scaled-eq} is given by \(\util^\n(x)\coloneqq \util(x/\n)\), where \(\util(x)\) is the shock of the original equation \eqref{burgers} corresponding to the pair of parameters \((\e,\d)\).

We also need to introduce an entropy shock, also known as a Riemann shock, to the (inviscid) Burgers equation associated with the same prescribed states \(u_-\) and \(u_+\).
When \(u_->u_+\), the associated entropy shock is given by 
\begin{equation} \label{RS}
\ubar(x)=
\begin{cases}
u_- &\text{if } x-\s t<0,\\
u_+ &\text{if } x-\s t>0,
\end{cases}
\end{equation}
where \(\s\) is determined by the Rankine--Hugoniot condition \eqref{sigma}.

\vspace{2mm}
We are ready to state the next theorem. 
The first part serves an optimal convergence rate (see \cite{KangVasseur} for viscous Burgers equation). The second part ensures the existence of zero viscosity-dispersion limits under well-prepared initial data, on which the entropy shock is unique and orbitally stable.

\begin{theorem} \label{thm:limit}
Let \(\e,\d>0\) be fixed constants and let \(u_\pm\) be given states.
Assume that \(u_->u_+\) and \eqref{cond1}.
Let \(\util\) denote the associated non-monotone shock of \eqref{burgers} which connects \(u_-\) and \(u_+\) with a fixed location \(\util(0)=(u_-+u_+)/2\).\\
(1) Let \(u_0\) be any initial datum with \(u_0-\util\in H^1(\RR)\) and let \(u^\n\) be the solution to \eqref{scaled-eq} subject to the scaled initial data \(u_0(x/\n)\).
Then, the solution \(u^\n\) satisfies that for any \(t>0\),
\begin{equation} \label{pstu}
\norm{u^\n(t,\cdot)-\ubar(\cdot-\s t-Y_\n(t))}_{L^2(\RR)}
\le \norm{u_0-\ubar}_{L^2(\RR)} + C\sqrt{\n}
\end{equation}
where \(Y_\n\) is given by \(Y_\n(t)=\n X(t/\n)\) from the shift function \(X\) given in \eqref{X-def} with \(\util(x)\) and \(u_0(x)\).\\
(2) Let \(u^0\) be any initial datum with \(\Ecal_0 \coloneqq \int_\RR (u^0-\ubar)^2 dx < \infty\).
Then the following holds:\\
(i) (Well-prepared initial data) There exists a sequence of smooth functions \(\{u_0^\n\}_{\n>0}\) on \(\RR\) such that 
\begin{equation} \label{wpid}
\lim_{\n\to0} u_0^\n = u^0 \quad \text{a.e.}, \qquad
\lim_{\n\to0} \int_\RR (u_0^\n-\util^\n)^2 dx = \Ecal_0,
\end{equation}
where \(\util^\n(x)=\util(x/\n)\).\\
(ii) For any given \(T>0\) and any \(\n>0\), let \(u^\n\) be the solution in \(\Xcal_T\) to \eqref{scaled-eq} subject to the initial data \(u_0^\n\).
Then, there exists a function \(u_\infty\in L^\infty(0,T;L_{loc}^2(\RR))\) such that up to a subsequence,
\begin{equation} \label{zvdl}
u^\n \rightharpoonup u_\infty \quad \text{in } L^\infty(0,T;L_{loc}^2(\RR)), \quad \text{as } \n \to 0.
\end{equation}
(iii) There exists a function \(X_\infty\in BV(0,T)\) and a constant \(C>0\) such that for a.e. \(t\in(0,T)\),
\begin{equation} \label{zvds}
\frac{1}{2} \int_\RR \big(u_\infty(t,x)-\ubar(x-\s t-X_\infty(t))\big)^2 dx
\le C\Ecal_0.
\end{equation}
(iv) Moreover, the function \(X_\infty\) satisfies 
\begin{equation} \label{X-con}
|X_\infty(t)| \le \frac{C(T)}{u_--u_+} (\Ecal_0+\sqrt{\Ecal_0}).
\end{equation}
Therefore, entropy shocks \eqref{RS} are stable and unique in the class of zero viscosity-dispersion limits of solutions to the KdVB equation \eqref{scaled-eq}.
\end{theorem}

A further main result regarding structural properties of the viscous-dispersive shock profiles will be presented in the next section.

\vspace{2mm}
The paper is organized as follows.
In Section \ref{sec:idea}, we introduce the main idea of the present work and state an additional main result concerning the properties of viscous-dispersive shock profiles.
This result is proved in Section \ref{sec:shock}, where we analyze the dynamics along the right-most monotonic interval as well as the other monotonic intervals.
In Section \ref{sec:main}, we obtain \(L^2\) type estimates for the shock profile and establish the \(L^2\)-contraction property under a time-dependent shift \(X(t)\).
Finally, in Section \ref{sec:limit}, we prove the zero viscosity-dispersion limits.

\section{Properties of Shock Profile and Main Ideas of Proofs} \label{sec:idea}
\setcounter{equation}{0}
In this section, we discuss the structure properties of the viscous-dispersive shock profile and present the main ideas underlying the proofs.
Note that the study of shock \(\util\) plays a crucial role in proving Theorem \ref{thm:main}.

\subsection{Structural Properties of Viscous-Dispersive Shocks}
We integrate the traveling wave equation \eqref{visS} over \((\x,\pm\infty)\) to find that
\begin{equation} \label{t00}
-\s (\util-u_\pm) + \frac{1}{2}\util^2-\frac{1}{2}u_\pm^2 =\e\util' -\d\util''.
\end{equation}
The qualitative behavior of the shock profile can be understood from the associated phase portrait, shown in Figure~\ref{pp}.

\begin{figure}[htp] \centering
    \includegraphics[width=.45\textwidth]{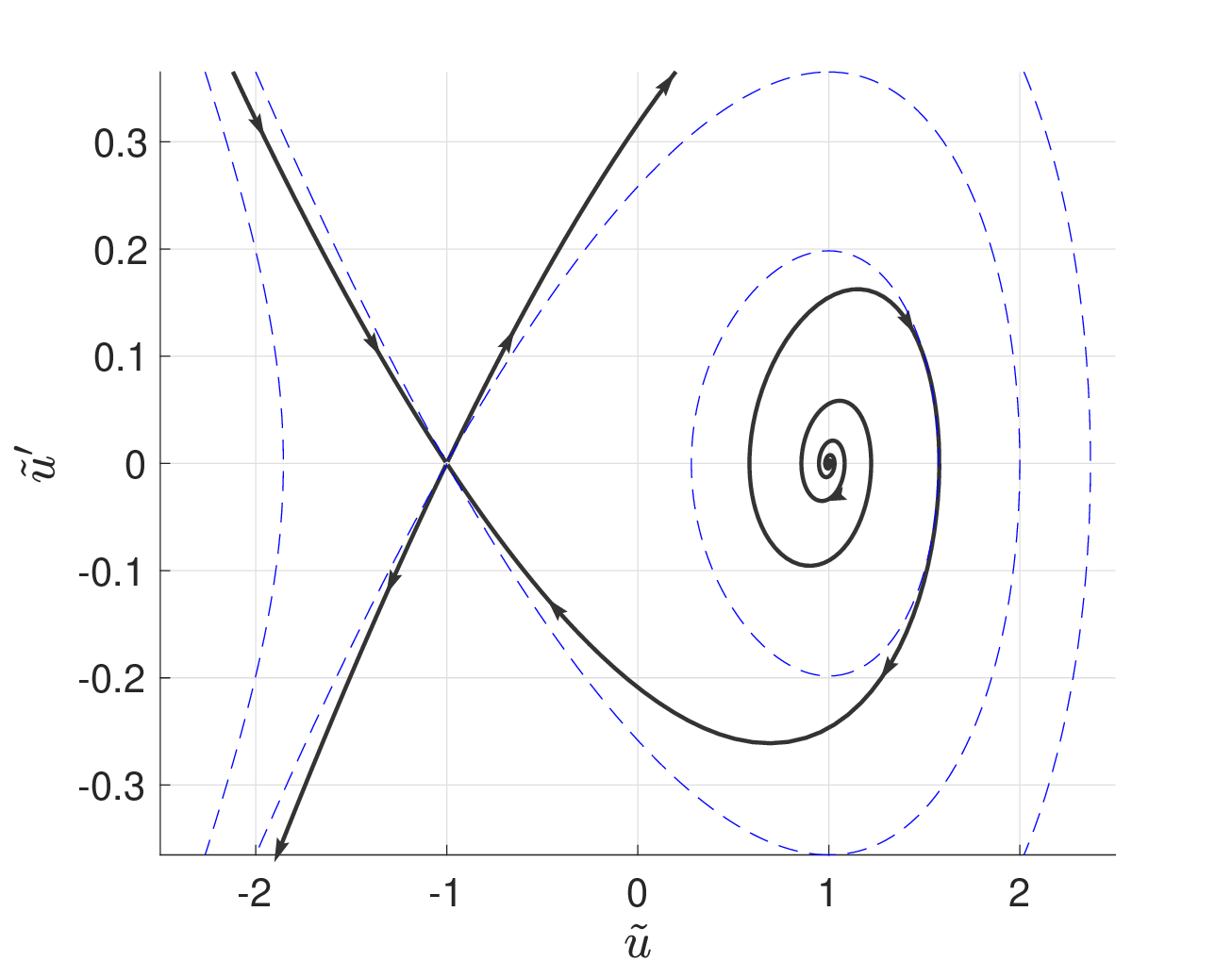}
    \includegraphics[width=.45\textwidth]{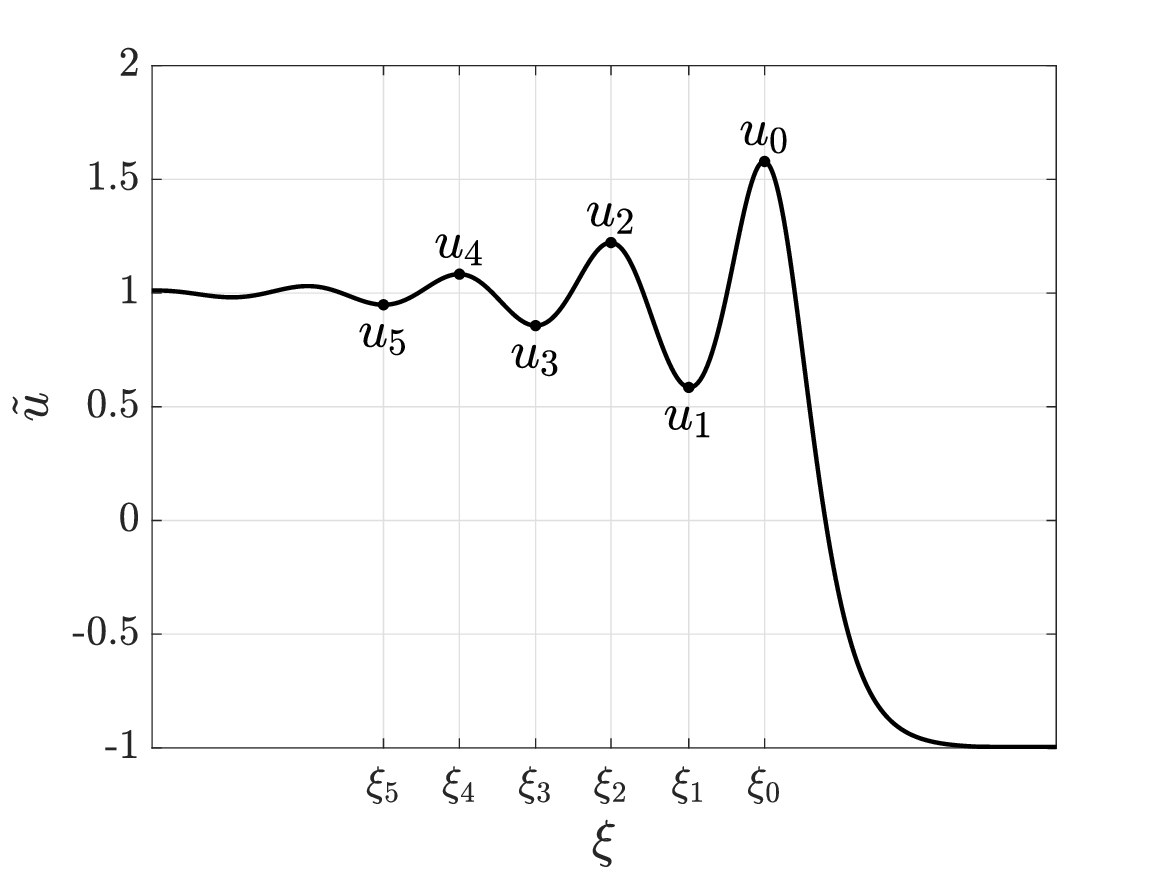}
    \caption{The left panel displays the phase portrait of solutions to \eqref{t00} with \(\d=10\),  \(\e=1\) and \(u_{\pm}=\mp 1\) in the \((\util,\util')\)-plane.
    The blue dotted curves correspond to solitary wave solutions of KdVB equation with \(\e=0\), while the black solid curve represents the heteroclinic orbit for \(\e=1\) and \(u_--u_+=2\).
    This orbit spirals from \(\util=u_-\) to \(\util=u_+\) and corresponds to the shock profile \(\util(\x)\) shown in the right panel.\label{pp}}
\end{figure}

As shown in the left panel of Figure~\ref{pp}, in the \((\util,\util')\)-plane, the heteroclinic orbit connects two steady states \((\util=u_\pm, \util'=0)\) and corresponds to the dispersive shock solution in the right panel.
This oscillating shock profile is the object of interest in the present work.

Let's first consider the linearized equation of \eqref{t00} around the equilibrium $\tilde u=u_\pm$.
At the saddle point \((\util=u_+, \util'=0)\), the eigenvalues  are given by \((\e \pm \sqrt{\e^2+2\d(u_--u_+)})/(2\d)\).
Meanwhile, at \((\util=u_-, \util'=0)\), the eigenvalues are given by 
\begin{equation} \label{L-slope}
\l_{u_-} = \frac{\e \pm \sqrt{\e^2-2\d(u_--u_+)}}{2\d}.
\end{equation}
We focus on the scenario when \(\e>0\) and \(\frac{\d(u_--u_+)}{2\e^2} >\frac{1}{4}\), i.e., the dispersion dominates the dissipation, and the dispersive shock is non-monotonic as in the right panel of Figure~\ref{pp}.
Notice that as \(\x\to -\infty\), the amplitudes of the oscillations around \(\tilde u=u_-\) for the linearized equation decay at the rate of 
\begin{equation} \label{lin-decay}
\exp \Big(-\frac{\pi}{2\delta \sqrt{\frac{\d(u_--u_+)}{2\e^2} - \frac{1}{4}}}\Big).
\end{equation}

Now we consider the traveling wave solution $\tilde u$ for the nonlinear equation \eqref{t00}.
To begin with, we introduce the following notation: we index the local extrema of \(\util(\x)\) from right to left.
The rightmost extremum is denoted by \(u_0\), and the remaining extrema are indexed sequentially as \(u_1, \ u_2, \cdots\) in the order they appear when moving leftward.
We write \(\x_i\) for the spatial location at which the value \(u_i\) is attained, so that \(\util(\x_i)=u_i\), as shown in Figure \ref{pp}.


We then present the final main result, which elucidates the structural properties of the viscous-dispersive shock with infinite oscillations.

\begin{theorem} \label{thm:shock}
Let \(\e,\d>0\) be constants and let \(u_\pm\) be given states.
Assume that \(u_->u_+\) and 
\begin{equation} \label{cond2}
\frac{1}{4} < \frac{\d(u_--u_+)}{2\e^2} < A \le 1.
\end{equation}
Then, the following holds: 
\begin{equation} \label{u0-upper}
u_0-u_- \le \frac{(u_--u_+)}{2} \frac{1}{100A} \Big(\sqrt{9+25(\sqrt{1+4A}-1)^2}-3\Big)^2
\end{equation}
which provides an upper bound for the rightmost local maximum of the dispersive shock profile \(\util\).
Moreover, the amplitudes of the oscillations \(\abs{u_i-u_-}\) decays exponentially in the sense that for any odd positive integer \(i>0\) and for each interval of $\xi$ with increasing $\util$, it holds that
\begin{equation} \label{inc-decay}
\frac{u_{i-1}-u_-}{u_--u_i} \ge \r_*>1,
\end{equation}
and for each interval of $\xi$ with decreasing $\util$, it holds that
\begin{equation} \label{dec-decay}
\frac{u_--u_i}{u_{i+1}-u_-} \ge \r^*>1,
\end{equation}
where the decay rates for selected values of \(A\) are given by

\begin{table}[h]
\centering
\begin{tabular}{c c c c c}
\hline
$A$ & $\r_*$   & $\r^*$\\
\hline
$1/3$ & 6.05 & 6.14\\
$1/2$ & 4.64 & 4.77\\
$2/3$ & 3.89 & 4.06\\
$3/4$ & 3.63 & 3.81\\
$1$ & 3.10 & 3.30\\
\hline
\end{tabular}
\caption{Decay Rates}
\label{tab:decay}
\end{table}
\end{theorem}

In view of Remark \ref{rmk:wlog}, we may assume that \(\e=1\) and \(\s=0\).
Indeed, the structural properties stated in Theorem \ref{thm:shock} concern the relative heights of local extrema measured with respect to the left end state, and are therefore independent of both scaling and Galilean invariant transformation.

\subsection{Idea for the Proof of Theorem \ref{thm:shock}}
In what follows, we restrict our attention to the case \(\e=1\) and \(\s=0\), without loss of generality.
Accordingly, the shock is stationary, and the asymptotic states are given by
\begin{equation} \label{stationary}
u_\pm = \lim_{\x \to \pm \infty} \util(\x)  \eqqcolon \mp s.
\end{equation}
Then, the condition \eqref{cond2} can be rewritten as \(\frac{1}{4}<\d s<A\le1\), and the dispersive shock satisfies 
\begin{equation} \label{t0}
\util' -\d\util'' = \frac{1}{2}(\util^2-s^2).
\end{equation}
and it is known by \cite[Theorem 4]{bona} that
\begin{equation} \label{t0_2}
-s<\util<2s.
\end{equation}

We now introduce an effective energy as follows:
\begin{equation} \label{t1}
E(\x)\coloneqq
\frac{1}{6}(\util(\x)+s)^2(\util(\x)-2s)
+\frac{\d}{2}(\util'(\x))^2
\in \Big(-\frac{2}{3}s^3,0\Big).
\end{equation}
Then, using \eqref{t0}, we find that
\begin{equation} \label{t2}
E'(\x)=\util'(\x) \Big(\frac{1}{2}(\util^2(\x)-s^2)+\d\util''(\x)\Big)=(\util'(\x))^2 \ge 0,
\end{equation}
which implies that \(E(\x)\) is monotonically increasing from \(E(-\infty)=-\frac{2}{3}s^3\) to \(E(+\infty)=0\).

\vspace{2mm}
The key idea is to approximate the derivative of the shock profile on \((\x_0,\infty)\).
More precisely, we establish the following form of inequality (see Proposition \ref{prop:RM}):
\begin{equation} \label{utpf}
-\util'(\x) \ge f(\util(\x))
\end{equation}
for some nonnegative algebraic function \(f\ge0\).
This type of inequality is crucial for us to obtain a sharp estimate on the decay of energy between two adjacent extreme points $u_{i+1}$ and  $u_{i}$. 
More precisely, for example, in the rightmost monotonic interval of $\xi$, \eqref{utpf} together with \eqref{t2} implies
\begin{align*}
\frac{1}{6}(u_0+s)^2(2s-u_0)
&=E(+\infty)-E(\x_0)\\
&=\int_{\x_0}^\infty (\util'(\x))^2 d\x
\ge \int_{\x_0}^\infty f(\util(\x))(-\util'(\x)) d\x
=\int_{-s}^{u_0} f(\util) d\util.
\end{align*}
Since \(f\) is an algebraic function, the integration on the right-hand side can in principle be carried out explicitly.
Consequently, the inequality above reduces to an algebraic inequality in \(u_0\), and the constraint imposed by this inequality enables us to derive an upper bound on \(u_0\) of the form \eqref{u0-upper}.
To obtain a sharper estimate for \(u_0\), we also employ a bootstrap argument.

\vspace{2mm}
The proofs of \eqref{inc-decay} and \eqref{dec-decay} are based on the same idea.
We estimate the derivative of the shock on each interval \((\x_i,\x_{i-1})\) as follows (see Proposition \ref{prop:inc} and \ref{prop:dec}): 
\begin{equation} \label{utpg}
\abs{\util'(\x)} \ge f(\util(\x))
\end{equation}
for some nonnegative algebraic function \(f\ge0\).
This together with \eqref{t2} yields that
\begin{align*}
&\frac{1}{6}(u_{i-1}+s)^2(u_{i-1}-2s)-\frac{1}{6}(u_i+s)^2(u_i-2s)
=E(\x_{i-1})-E(\x_i)\\
&\qquad\qquad\qquad
=\int_{\x_i}^{\x_{i-1}} (\util'(\x))^2 d\x
\ge \abs{\int_{\x_i}^{\x_{i-1}} f(\util(\x)) \abs{\util'(\x)} d\x}
=\abs{\int_{u_i}^{u_{i-1}} f(\util) d\util}.
\end{align*}
This can be reduced to an algebraic inequality in \(u_i\) and \(u_{i-1}\), from which the decay rates follow.

\begin{remark} \label{rmk:accuracy}
A key component of the strategy described above is the derivation of precise estimates for the derivative of the shock---namely, the sharper the estimate, the more accurately it captures the actual structure of the profile.
In particular, by examining \eqref{u0-upper} near the threshold of the non-monotonicity, i.e., \(A=\frac{1}{4}\), the accuracy of our approximation \(f\) in \eqref{utpf} can be indirectly evaluated as follows.
It is intuitively natural that \(u_0\to s\) as \(A\to\frac{1}{4}\).
On the other hand, \eqref{u0-upper} shows that 
\[
u_0 \le
\lim_{A\to\frac{1}{4}} \bigg(1+\frac{1}{100A} \Big(\sqrt{9+25(\sqrt{1+4A}-1)^2}-3\Big)^2\bigg)s
\approx 1.0166s
\]
which indicates that we have identified a highly accurate approximation \(f\) in Proposition \ref{prop:RM}.
\end{remark}

To establish the above estimates \eqref{utpf} and \eqref{utpg} for the derivative of the shock, we introduce a contradiction argument, which can be described as follows.
Since each inequality is considered on an interval where the shock profile is monotone---such as \((\x_i,\x_{i-1})\) or \((\x_0,+\infty)\)---we may parametrize both sides of the inequality in terms of the shock profile itself.
Restating the inequalities under this parametrization, we may rewrite them in the following form: 
\[
h(\util) \le p(\util) \quad \forall \, \util\in [u_L,u_R].
\]
Then, the following observations will be required:
\begin{align*}
&h(u_L)=p(u_L),
&&h'(u_L)<p'(u_L)<0,\\
&h(u_R)=p(u_R),
&&h'(u_R)>p'(u_R)>0.
\end{align*}
Assuming the contrary, the above observations imply that there exist two points \(b\) and \(c\) such that \(u_L<b<c<u_R\) and 
\begin{align*}
&h(b)=p(b), 
&&h'(b)-p'(b) \ge 0,\\
&h(c)=p(c),
&&h'(c)-p'(c) \le 0.
\end{align*}
This together with the structures of \(h\) and \(p\) allows us to define a function \(g\colon[u_L,u_R]\to\RR\) satisfying 
\[
g(u_L)<0, \qquad
g(b)\ge0, \qquad
g(c)\le0, \qquad
g(u_R)>0,
\]
which can be shown to be convex (or concave).
This leads to a contraction.

In Section \ref{sec:shock}, we repeatedly employ the above contradiction argument to characterize the shock properties.
Although the technical details vary depending on the interval under consideration, the underlying idea is robust and can be applied on any interval where the shock profile is monotone.

\subsection{Idea for the Proof of Theorem \ref{thm:main}} \label{sec:main3}
The study of the \(L^2\)-contraction starts from the following lemma on the time derivative of the \(L^2\)-distance between the shifted solution \(u^X\) and the dispersive shock \(\util\), which can be written as
\[
\frac{d}{dt}\int_\RR (u^X-\tilde u)^2 d\x
=\frac{d}{dt}\int_\RR (u-\tilde u^{-X})^2 d\x
\]
where \(u^X(t,x) \coloneqq u(t,x+X(t))\) and \(\util^{-X}(t,x) \coloneqq \util(x-X(t))\).

\begin{lemma} \label{lem:RHS}
Let \(\util\) be the viscous-dispersive shock given by \eqref{visS} with \eqref{stationary}.
Then, for any solution \(u\in\Xcal_T\) to \eqref{burgers} and any Lipschitz function \(X\colon[0,T]\to\RR\), the following holds.
\begin{equation} \label{key1}
\frac{d}{dt}\frac{1}{2}\int_\RR (u^X-\tilde u)^2 d\x
=\dot{X}(t) \int_\RR (u^X-\util)\util'd\x
-\frac{1}{2}\int_\RR (u^X-\util)^2 \util' d\x
- \int_\RR ((u^X-\util)_\x)^2 d\x.
\end{equation}
\end{lemma}

This lemma can be proved in the same way as \cite{KangVasseur}; see also \cite{Hur}.
To make the paper self-contained, we include the proof of Lemma \ref{lem:RHS} in Appendix \ref{section_proof}.

\vspace{2mm}
The need for a time-dependent shift \(X(t)\) can be justified as follows.
Consider a perturbation \(w\coloneqq u-\util\) that takes a large constant value one a large but compact set---containing, in particular, the transition zone from the rightmost local extremum to the right-end state---and decays slowly to zero outside of this set.
Then the dissipation term \(\int_\RR (w_\x)^2 d\x\) can be negligible, whereas the perturbation energy term \(\int_\RR w^2 \util' d\x\) is large; this leads to a failure of the contraction property.

\vspace{2mm}
To illustrate the strategy of the proof, we begin by outlining the argument for the simpler cases, namely, the viscous Burgers equation case and the KdVB equation but monotone shock case.
Then, we compare them with the oscillatory shock case in order to highlight the additional difficulties that arise in our problem of interest.

\paragraph{Case 1: \(\d=0\).}
This corresponds to the viscous Burgers equation, which admits smooth monotone viscous shock profiles.
An \(L^2\)-contraction property, analogous to Theorem \ref{thm:main}, can be established by combining Lemma \ref{lem:RHS} with the following Poincar\'e-type inequality.
\begin{lemma}\label{KV_lemma} \cite[Lemma 2.9]{acon_i_2}
For any function \(f\colon[a,b]\to\RR\) satisfying \(\int_a^b (y-a)(b-y)\abs{f'}^2 dy<\infty\),
\begin{equation} \label{KV}
\int_a^b f^2dy
\le \frac{1}{2}\int_a^b(y-a)(b-y)\abs{f'}^2 dy
+\frac{1}{b-a} \Big(\int_a^b f dy\Big)^2.
\end{equation}
\end{lemma}
By exploiting the localization effect of the derivative of the viscous shock, one may introduce a change of variables that maps \(\x\) to a bounded domain.
We emphasize that inequality \eqref{KV} is optimal, and the constant $\frac{1}{2}$ is independent of the length of the interval $[a,b]$ (see the proof of \cite[Lemma 2.9]{acon_i_2}). 
The first term on the right-hand side of \eqref{KV} is precisely the quadratic form associated with the Legendre operator  $-\partial_y((y-a)(b-y)\partial_y)$ whose eigenvalues are explicitly given by $n(n+2)$, $n\ge 0$. In particular, the only non-positive eigenvalue is zero.
Thus, the only non-dissipative direction corresponds to the kernel (the constant mode). This mode is exactly captured by the second term on the right-hand side of \eqref{KV}, namely the projection onto the kernel. In our framework, this component is controlled through the time-dependent shift (or temporal modulation), which plays the same role as the first term on the right-hand side of \eqref{key1}.

\vspace{2mm}
The \(L^2\)-contraction property follows from the above lemma.
This approach was carried out in \cite{KangVasseur}, and its extension to the multi-D case was established in \cite{Oh}.
This idea has been successfully extended to systems of viscous conservation laws.
In particular, contraction properties for viscous shocks under any large perturbations have been obtained for a variety of physical systems, including the isentropic Navier--Stokes system \cite{acon_i_2,KV}, the isothermal Navier--Stokes system \cite{EEKO-ISO}, and the Brenner--Navier--Stokes--Fourier system \cite{EEK-BNSF}.
Moreover, this approach has proved powerful enough to address challenging problems concerning the inviscid limits from the Navier--Stokes equations to the Euler equations; see, for instance, \cite{CKV,KV}.
On the other hand, the above lemma also provides a fundamental tool for time-asymptotic stability of composite wave patterns beyond (single) viscous shock waves.
More specifically, stability results for composite waves consisting of viscous shocks and rarefaction waves have been obtained for the barotropic Navier--Stokes system \cite{KVW-ADV}, the compressible Navier--Stokes--Fourier system \cite{KVW-ARMA}, and nonconvex viscous conservation laws \cite{HWZ-MA}.
It is worth noting that the composite waves in \cite{KVW-ARMA} also contain contact discontinuities, while the shocks considered in \cite{HWZ-MA} are degenerate Oleinik shocks.

\paragraph{Case 2: \(0<\d s\le \frac{1}{4}\).}
In this case, the dispersive shock profile \(\util(\x)\) of KdVB equation \eqref{burgers} is globally decreasing, which resembles the shape of the viscous shock to the viscous Burgers equation: 
\[
u_t + f(u)_x = u_{xx}.
\]
Thus, the same \(L^2\)-contraction property, as in Theorem \ref{thm:main}, can be established similarly.
To focus on the non-monotonic case, we leave it to another paper.
We also refer to \cite{EKK-NSK} for contraction estimates of monotone viscous-dispersive shocks to the Naiver--Stokes--Korteweg system.

\vspace{2mm}
Before discussing the main new idea for the non-monotonic shock profile case, we briefly explain how the \(L^2\)-contraction is obtained in Case 1 and 2.
When \(\d=0\), choosing the shift \(X(t)\) as 
\[
\dot{X}(t) = - K\int_\RR (u^X-\util)\util'd\x
\]
for some constant \(K>0\), and using the change of variables 
\begin{equation} \label{ztu}
z(\x) \coloneqq \util(\x), \qquad
dz = \util' d\x,
\end{equation}
with \eqref{stationary} and \eqref{t0}, we rewrite the right-hand side of \eqref{key1} as 
\begin{equation} \label{L2dis}
-K \Big(\int_{-s}^s w dz\Big)^2
+\frac{1}{2} \int_{-s}^s w^2dz
-\frac{1}{2} \int_{-s}^s (s-z)(z+s)(w_z)^2 dz
\end{equation}
with \(w\coloneqq u^X-\util\).
Then, choosing \(K>\frac{1}{4s}\), we obtain the desired \(L^2\)-contraction via Lemma \ref{KV_lemma}.

\vspace{2mm}
When \(0<\d s\le\frac{1}{4}\), since the shock profile \(\util\) is monotonically decreasing, \(dz=\util' d\x\) is globally well-defined as well.
However, unfortunately, the time evolution of the \(L^2\)-distance cannot be written in the form of \eqref{L2dis}.
To still apply Lemma \ref{KV_lemma}, we prove the following inequality:
\begin{equation} \label{ineq_00}
-\util'(\x) \ge \l(s-\util(\x))(\util(\x)+s), \quad \forall\, \x\in\RR,
\end{equation}
for some \(\l\ge\frac{1}{4}\).
This inequality is nontrivial, but it can still be proved using our new contradiction argument explained above.
To be specific, this can be achieved by examining a heteroclinic orbit on the \((\util,\util')\)-plane using 
\begin{equation} \label{orbit}
\d \frac{d\util'}{d\util} = 1 - \frac{1}{2\util'}(\util-s)(\util+s).
\end{equation}
This can be obtain directly from \eqref{t0}.

\paragraph{Case 3: $\delta s>\frac{1}{4}$.}
In this much more interesting case, the dispersive shock includes infinitely many oscillations.
Note that the second term on the right-hand side of \eqref{key1} is now negative on intervals where the dispersive shock is increasing, and positive on intervals where it is decreasing.
Moreover, the transformation \eqref{ztu} is not globally well defined on \(\RR\).
If one applies the transformation \eqref{ztu} on each interval of $\xi$ with monotonic $\util$ (sometime just call it monotonic interval), as we do here, the first difficulty is to establish an inequality analogous to \eqref{ineq_00} in this interval.
On the rightmost monotonic interval, where \(\x \in (\x_0,\infty)\), we show
\begin{equation} \label{ineq_01}
-\util'(\x) \ge \l (u_0-\util(\x))(\util(\x)+s), \quad \forall\, \x\in(\x_0,\infty),
\end{equation}
for some \(\l>\frac{1}{4}\) (see Proposition \ref{prop:key}), whose proof relies on the upper bound for \(u_0\) established in Theorem \ref{thm:shock}.
This is a key inequality for the future step of applying the Poincar\'e-type inequality, i.e., Lemma \ref{KV_lemma}.
One remark is that, although the right-hand side of \eqref{ineq_01} is an algebraic function (and thus could serve as the function \(f\) in \eqref{utpf}), it is not very useful, since we instead employ a much sharper approximation, given by 
\begin{equation} \label{ineq_02}
-\util'(\x) \ge \l \sqrt{s} (u_0-\util(\x))^\frac{1}{2}(\util(\x)+s)
+\m(u_0-\util(\x))(\util(\x)+s), \quad \forall\, \x\in(\x_0,\infty),
\end{equation}
for some positive constants \(\l\) and \(\m\) depending on \(A\), which is far more accurate near \(u_0\).
Here, the right-hand side of \eqref{ineq_02} is carefully chosen to fit the heteroclinic orbit, see Figure \ref{het}.

\begin{figure}[htp] \centering
\includegraphics[width=.4\textwidth]{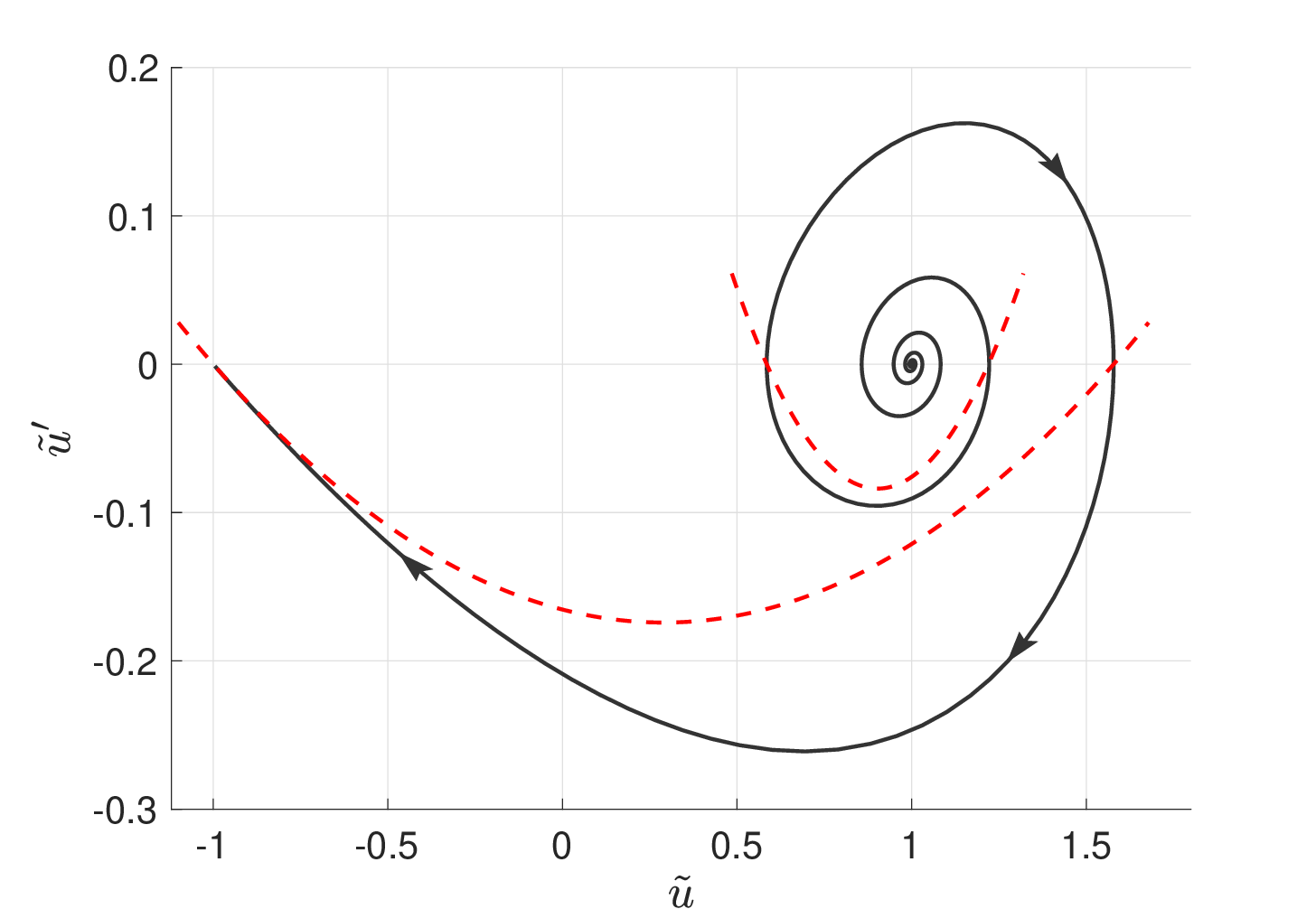}
    \qquad
\includegraphics[width=.4\textwidth]
    {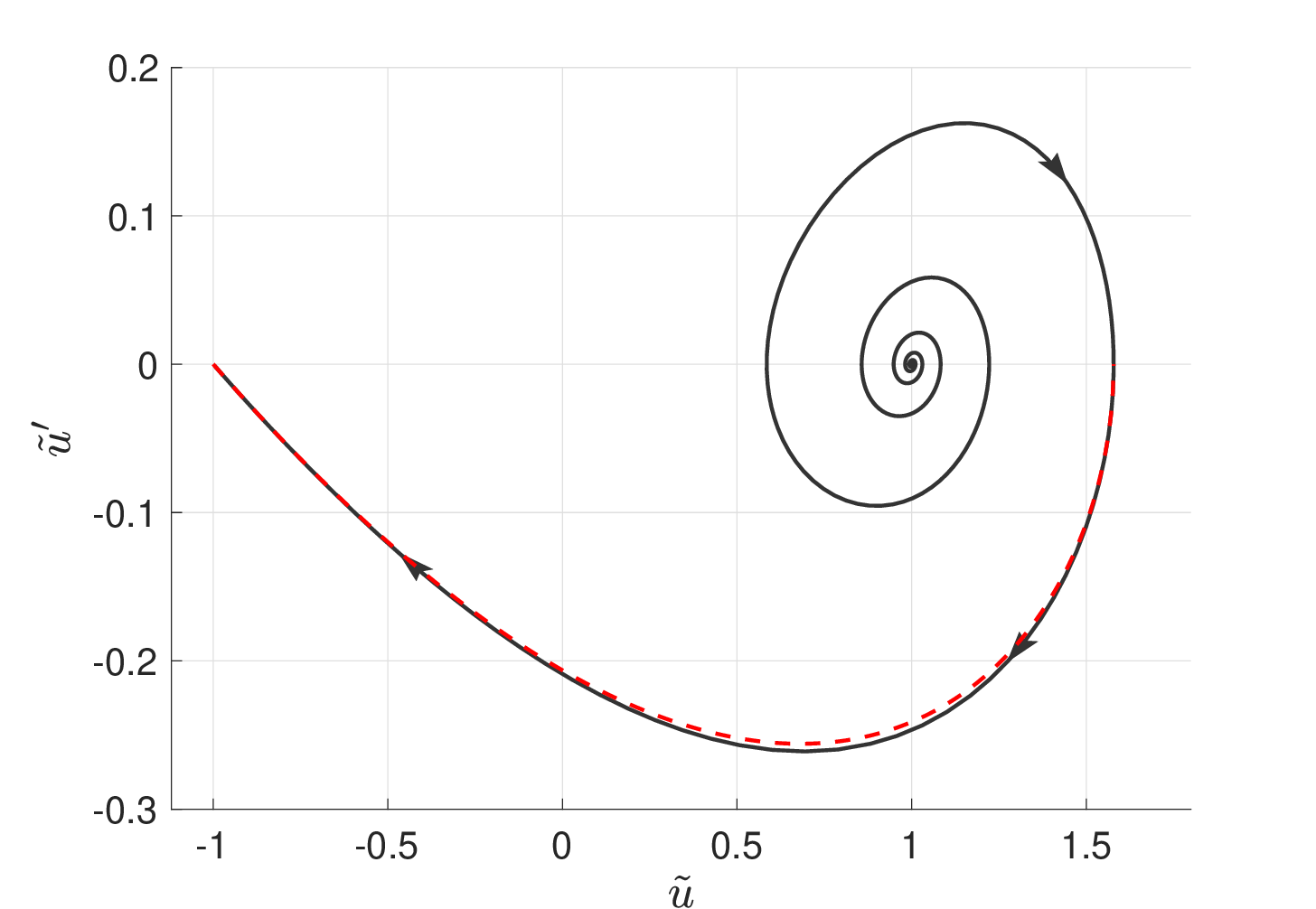}
\caption{Illustration of the heteroclinic orbit in comparison with auxiliary curves on each interval of $\xi$ with monotonic $\util$.
(Left) The dotted curves denote parabolas, such as \(y=\l(u_0-\util)(\util+s)\) in \eqref{ineq_00} for the rightmost monotonic interval, which lie above \(\abs{\util'}\).
(Right) To obtain a sharper estimate for \(u_0\), we show \eqref{ineq_02}.
The slope of \((u_0-\util)^\frac{1}{2}(\util+s)\) is infinite at \(u_0\), yielding a more accurate approximation than \((u_0-\util)(\util+s)\).
In the other monotonic intervals, we use functions of the form \(\l_i(\util-u_i)^\frac{1}{2}(u_{i-1}-\util)^\frac{1}{2}\) to fit the heteroclinic orbit.\label{het}}
\end{figure}

\vspace{2mm}
To apply Lemma \ref{KV_lemma} for the other intervals of $\xi$ with monotonic $\util$, it also requires to establish key inequalities similar to \eqref{ineq_01} (see Proposition \ref{prop:key}) as follows: for odd \(i\), (i.e., \(\util\) is increasing on \((\x_i,\x_{i-1})\))
\begin{equation} \label{ineq_03}
\util'(\x) \ge \l_i(u_{i-1}-\util(\x))(\util(\x)-u_i), \quad \forall\, \x\in(\x_i,\x_{i-1}),
\end{equation}
and for even \(i\), (i.e., \(\util\) is decreasing on \((\x_i,\x_{i-1})\))
\begin{equation} \label{ineq_04}
-\util'(\x) \ge \l_i(u_i-\util(\x))(\util(\x)-u_{i-1}), \quad \forall\, \x\in(\x_i,\x_{i-1}).
\end{equation}
The decay rates \eqref{inc-decay} and \eqref{dec-decay} imply that \(\l_i\) is increasing in \(i\) and diverges to infinity as \(i\to\infty\), which will play a crucial role later.
We also remark that while the right-hand sides of \eqref{ineq_03} and \eqref{ineq_04} are algebraic and thus admissible as candidates for \(f\) in \eqref{utpg}, it provides little analytical leverage.
Instead, we rely on sharper approximations, especially near \(u_i\) and \(u_{i-1}\), namely
\begin{align*}
\util'(\x) &\ge \l_i (u_{i-1}-\util(\x))^\frac{1}{2}(\util(\x)-u_i)^\frac{1}{2},
\quad \forall\, \x\in(\x_i,\x_{i-1}) \\
-\util'(\x) &\ge \l_i (u_i-\util(\x))^\frac{1}{2}(\util(\x)-u_{i-1})^\frac{1}{2},
\quad \forall\, \x\in(\x_i,\x_{i-1})
\end{align*}
for odd and even \(i\), respectively (see Proposition \ref{prop:inc} and \ref{prop:dec}).

\vspace{2mm} 
Thanks to \eqref{ineq_01}, \eqref{ineq_02} and \eqref{ineq_03} with \eqref{ztu}, the dissipation term \(\int_\RR (w_\x)^2 d\x\) can be estimated as the first term on the right-hand side of \eqref{KV} (or the last term in \eqref{L2dis}). So, to use Lemma \ref{KV_lemma} for each interval of $\xi$ with monotonic $\util$, we need the square of the mean as the last term of \eqref{KV}. 
Unfortunately, the following square of global mean obtained from \eqref{X-def} and \eqref{key1} is not localized for each monotonic interval:
\begin{equation}\label{glsm}
-\frac{M}{2s} \Big(\int_\RR w \util' d\x\Big)^2.
\end{equation}
The most challenging part of the proof is to localize the above quantity sufficiently for each monotonic interval.
We perform the localization process in an inductive argument by connecting adjacent monotonic intervals step by step with careful quantitative estimates.

The starting point of the argument is to extract the squared mean from a portion of the diffusion on each interval $J_i:=(\x_i,\x^i)$ with \(\x^i\in (\x_{i-1},\x_{i-2})\) satisfying \(\util(\x^i)=u_i\) (see Figure \ref{fig_mj}) as follows:
\begin{equation} \label{transfer}
\Big(\int_{J_i} w\util' d\x\Big)^2
=\Big(\int_{J_i} w_\x (\util-u_i) d\x\Big)^2
\le \Big(\int_{J_i} (\util-u_i)^2 d\x\Big)
\Big(\int_{J_i} (w_\x)^2 d\x\Big).
\end{equation}
\begin{figure}[htp] \centering
	\includegraphics[width=.6\textwidth]{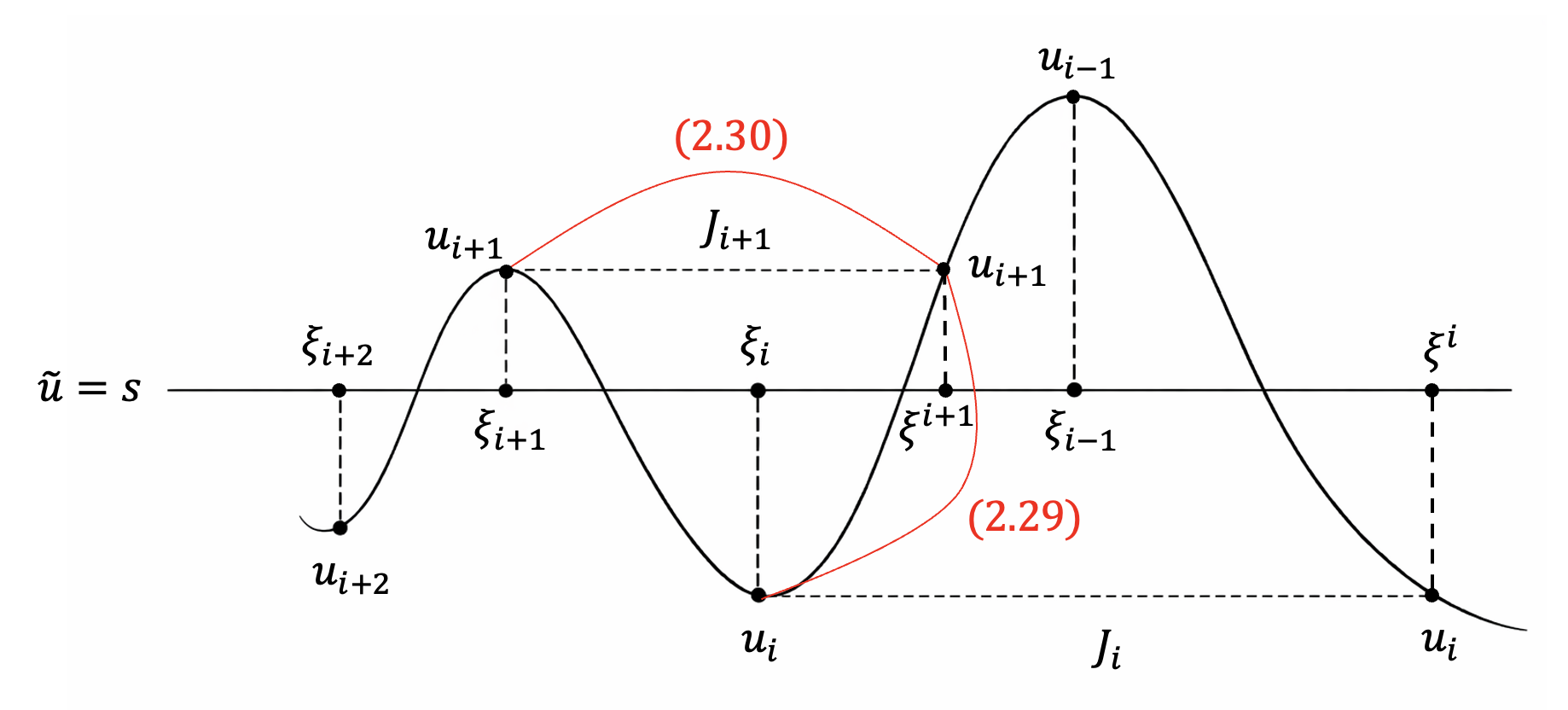}
	\caption{Illustration of \(J_i\) for odd \(i\) and the induction argument.} \label{fig_mj}
\end{figure}
To get the squared mean on a decreasing interval (i.e., on $(\x_{i+1},\x_i)$ for odd $i$), we may use \eqref{transfer} with the following good terms on an increasing interval $(\xi_i, \x_{i-1})$ from the second term on r.h.s. of \eqref{key1}: by $\util'>0$  whenever $i$ is odd,
\begin{equation}\label{goodinc}
-\frac{1}{2} \int_{\x_i}^{\x_{i-1}} w^2 \util' d\x <0.
\end{equation}
This quantity, however, is not sufficient to obtain the desired amount of the squared mean.
In fact, we need slightly more than that.
Thus, as an induction hypothesis, we assume that for odd \(i\), the following good term on \((\x_i,\x_{i-1})\) is available: for some positive constant \(a_i\),
\begin{equation} \label{26}
-\Big(\frac{1}{2}+a_i\Big) \int_{\x_i}^{\x_{i-1}} w^2 \util' d\x.
\end{equation}
The initial induction hypothesis, i.e., the base case, will be justified in \textit{Step 0} and \textit{Step 1} of Section \ref{sec:pf-main}, where the global squared mean \eqref{glsm} plays a significant role. 
Then, we note from \eqref{26} that 
\begin{equation} \label{ind-hypo}
-\Big(\frac{1}{2}+a_i\Big) \int_{\x_i}^{\x_{i-1}} w^2 \util' d\x
\le -\Big(\frac{1}{2}+a_i\Big) \int_{\x_i}^{\x^{i+1}} w^2 \util' d\x 
\le -\frac{\big(\frac{1}{2}+a_i\big)}{u_{i+1}-u_i} \Big(\int_{\x_i}^{\x^{i+1}}w\util'd\x\Big)^2.
\end{equation}
To get the squared mean on $(\xi_{i+1},\xi_i)$  from \eqref{ind-hypo}, we use the squared mean on a wider set $J_{i+1}=(\xi_{i+1}, \xi^{i+1})$ obtained from the diffusion as in \eqref{transfer} : 
\begin{equation}\label{ciplus}
-\frac{C_{i+1}}{u_{i+1}-u_i}\Big(\int_{J_{i+1}} (\util-u_{i+1})^2 d\x\Big)^2
\Big(\int_{J_{i+1}} (w_\x)^2 d\x\Big)^2
\le -\frac{C_{i+1}}{u_{i+1}-u_i}\Big(\int_{J_{i+1}} w\util' d\x\Big)^2.
\end{equation}
Combining \eqref{ind-hypo} and \eqref{ciplus} with choosing large enough \(C_{i+1}\), we obtain a sufficient amount of the squared mean  on $(\xi_{i+1},\xi_i)$ as
\[
-\frac{\big(\frac{1}{2}+a_{i+1}\big)}{u_{i+1}-u_i} \Big(\int_{\x_{i+1}}^{\x_i}w\util'd\x\Big)^2,
\]
for some constant \(a_{i+1}>0\) which is strictly less than \(a_i\) (we will take \(a_{i+1}=\frac{a_i}{2}\)).
This quantity together with a bit of diffusion (thanks to \eqref{ineq_04} with $\lambda_i$ far larger than \(\frac{1}{4}\)) controls the bad term on $(\xi_{i+1},\xi_i)$ via Lemma \ref{KV_lemma}, and so we are left with  
\[
a_{i+1} \int_{\x_{i+1}}^{\x_i}w^2\util'd\x,
\]
from which we have
\[
a_{i+1}\int_{\x_{i+1}}^{\x_i}w^2\util'd\x
\le a_{i+1}\int_{\x_{i+1}}^{\x^{i+2}}w^2\util'd\x
\le -\frac{a_{i+1}}{u_{i+1}-u_{i+2}} \Big(\int_{\x_{i+1}}^{\x^{i+2}} w\util'd\x\Big)^2.
\]
We again use \eqref{transfer} to derive the squared mean term on \(J_{i+2}\), which, in turn, leads us to obtain 
\[
-\frac{a_{i+2}}{u_{i+1}-u_{i+2}} \Big(\int_{\x_{i+2}}^{\x_{i+1}} w\util'd\x\Big)^2
\]
for  \(a_{i+2}=\frac{a_{i+1}}{2}\).
Then, applying Lemma \ref{KV_lemma} with a bit of the diffusion on \((\x_{i+2},\x_{i+1})\), we obtain 
\[
-a_{i+2} \int_{\x_{i+2}}^{\x_{i+1}} w^2 \util' d\x
\]
which with the good term on \((\x_{i+2},\x_{i+1})\) as in \eqref{goodinc} recovers the induction hypothesis \eqref{ind-hypo}.

\vspace{2mm}
In the above procedure, we should carefully estimate the quantity $\frac{1}{\abs{u_{i+1}-u_{i}}}\int_{J_{i+1}} (\util-u_{i+1}) ^2 d\x$ to be small enough, which ensures that there is a sufficient amount of the diffusion term in \eqref{key1} to be used in the inequality \eqref{ciplus}. 
More precisely, from \eqref{ciplus}, the choice of \(\{C_i\}\) depends on the following upper bound:
\[
\frac{1}{\abs{u_{i+1}-u_{i}}}\int_{J_{i+1}} (\util-u_{i+1}) ^2 d\x \lesssim  (\r_*)^{-i},
\]
which will be shown in Section \ref{section_l2}. This quantitative estimate can be performed by choosing $A$ suitably from Theorem \ref{thm:shock}, although the choice of $A$ is not optimal.

\vspace{2mm}
We close this section with a remark on the robustness of our inductive argument.
As noted in Remark \ref{rmk:accuracy}, the estimate on \(u_0\) is sharp.
By contrast, the decay rates \(\r_*\) and \(\r^*\) obtained here do not seem to be accurate when compared with the decay rate at the linearized level \eqref{lin-decay}.
A sharper decay rate would likely extend values of \(A\) beyond the present restriction \(A=\frac{1}{2}\).

\section{Proof of Theorem \ref{thm:shock}: Shock Properties} \label{sec:shock}
\setcounter{equation}{0}
This section is devoted to the proof of Theorem \ref{thm:shock}.
As discussed in Section \ref{sec:idea}, the key step of the proof consists in constructing an appropriate approximation of \(\util'\).
These approximations will be established via the contradiction argument introduced earlier.
The proof proceeds through the successive verification of \eqref{u0-upper}, \eqref{inc-decay} and \eqref{dec-decay}.

We recall that through the proof we work under the assumption \(\s=0\), namely, \(u_-=-u_+=s\), and \(\e=1\).
We also recall that we are in the oscillatory regime where the dispersion dominates dissipation, that is, \(\frac{1}{4}< \d s\).
Moreover, for any constant \(\frac{1}{4}<A\le1\), we carry out the analysis under the condition 
\[
\frac{1}{4} < \d s < A.
\]

\subsection{Upper Bound for the Rightmost Local Maximum}
In this subsection, we derive an upper bound for the rightmost local extremum, i.e., \(u_0=\util(\x_0)\).
The proof relies on a bootstrapping argument: once an initial approximation of the shock derivative is obtained, an upper bound for \(u_0\) follows from the effective energy.
This bound, in turn, yields an improved approximation, which, when combined again with the effective energy, leads to a sharper upper bound for \(u_0\).

The only upper bound initially available is \(u_0<2s\) (see \eqref{t0_2}).
Our first step is to improve this bound by means of a preliminary bootstrapping argument: using the lemma below, we show that \(u_0<\frac{5}{4}s\).
This improvement is essential for the subsequent construction of a sufficiently accurate approximation \eqref{ineq_02} (or Proposition \ref{prop:RM}), which ultimately enables us to establish \eqref{u0-upper}.

\begin{lemma} \label{lem:RM}
Let \(A\) be a constant which satisfies \(\frac{1}{4}<A\le1\).
Assume \(\frac{1}{4}<\d s<A\).
\begin{itemize}
\item[(i)] Let \(k\le1\) be a positive constant such that 
\begin{equation} \label{RM-ass}
k < \frac{2\sqrt{A}}{1+\sqrt{1+4A}} \sqrt{\frac{s}{u_0-s}}.
\end{equation}
Then, for \(\l=k \sqrt{\frac{1}{A}} \sqrt{\frac{u_0-s}{u_0+s}}\), the following holds:
\begin{equation} \label{RM-ineq} 
-\util'(\x) \ge \l \sqrt{s} (u_0-\util(\x))^\frac{1}{2}(\util(\x)+s), \quad \forall\, \x\in(\x_0,\infty).
\end{equation}
Moreover, \(u_0\) satisfies the following upper bound: for \(k_* = \big(1+\frac{32k^2}{25A}\big)\),
\begin{equation} \label{u0}
u_0 \le \big(1+k_*-\sqrt{k_*^2-1}\big)s.
\end{equation}
\item[(ii)] The following holds:
\begin{equation} \label{RM-5/4}
u_0 \le \frac{5}{4}s.
\end{equation} 
\end{itemize}
\end{lemma}

\begin{pf}
The proof is based on the contradiction argument outlined above.
First of all, using the monotonicity of the shock profile on \((\x_0,\infty)\), we introduce a parameterization with respect to \(a=\util\) and define two functions as follows: 
\[
h(a) \coloneqq \util'=\util_\x, \qquad
p(a) \coloneqq -\l\sqrt{s} (u_0-a)^\frac{1}{2} (a+s).
\]
In the proof, we take \(a=\util\) as the variable instead of \(\x\), and thus we use the notation \(\util_\x\) in place of \(\util'\) for clarity.
The desired inequality \eqref{RM-ineq} now boils down to
\begin{equation} \label{RM-ineq2}
h(a) \le p(a), \quad \forall a\in(-s,u_0).
\end{equation}
Then, we examine the two functions at the endpoints as follows: we first have
\begin{equation} \label{h-p}
h(u_0)=p(u_0)=0, \qquad
h(-s)=p(-s)=0.
\end{equation}
To compare the derivatives of the functions \(h\) and \(p\) at the endpoints, we recall \eqref{t0} and divide it by \(\util_\x\), which yields
\[
\d \frac{d\util_\x}{d\util} = 1 - \frac{1}{2\util_\x}(\util-s)(\util+s).
\]
This implies that
\begin{equation} \label{hp}
h'(a) = \frac{d\util_\x}{d\util} = \frac{1}{\d} \Big[1 - \frac{1}{2h(a)}(a-s)(a+s)\Big].
\end{equation}
Note that both \(h'(a)\) and \(p'(a)\) diverge when \(a\) approaches to \(u_0\).
To evaluate the divergence rate of \(h'(a)\) around \(a=u_0\), we apply L'Hôpital's rule to find that 
\[
\lim_{a\to u_0-} \frac{h(a)}{(u_0-a)^\frac{1}{2}}
=\lim_{\x\to \x_0+} \frac{\util_\x}{(u_0-\util)^\frac{1}{2}}
=\lim_{\x\to \x_0+}  -\frac{2(u_0-\util)^\frac{1}{2} \util_{\x\x}}{\util_\x}.
\]
Then, \eqref{t0} implies \(\util_{\x\x}(\x_0)= \frac{s^2-u_0^2}{2\d}\), and thus we have 
\[
\lim_{a\to u_0-} \frac{h(a)}{(u_0-a)^\frac{1}{2}}
=-\sqrt{\frac{u_0^2-s^2}{\d}}.
\]
This together with \eqref{hp} yields that
\begin{equation} \label{hpu0}
\lim_{a\to u_0-} h'(a)(u_0-a)^\frac{1}{2}
=\frac{1}{2}\sqrt{\frac{u_0^2-s^2}{\d}}.
\end{equation}
To observe \(h'(a)\) at \(a=-s\), we need to estimate the following quantity:
\[
\lim_{a\to(-s)+} \frac{h(a)}{a+s}
=\lim_{\x\to\infty}\frac{\util_\x}{\util+s} \eqqcolon L.
\]
We again use \eqref{t0} and L'Hôpital's rule to find that
\begin{align*}
\lim_{\x\to\infty}\Big[\d\frac{\util_\x}{\util+s}-1\Big]
=\lim_{\x\to\infty}\Big[\d\frac{\util_{\x\x}}{\util_\x}-1\Big]
=\lim_{\x\to\infty}\frac{s^2-\util^2}{2\util_\x}
=\lim_{\x\to\infty}\Big[\frac{s-\util}{2} \cdot \frac{\util+s}{\util_\x}\Big].
\end{align*}
This leads \(L\) to satisfy \(\d L^2 - L -s =0\).
Then, since \(\util_\x<0\) and \(\util+s>0\) on \((\x_0,\infty)\), it holds that 
\[
L
= -\frac{2s}{1+\sqrt{1+4\d s}}
\]
and thus, we obtain
\begin{equation} \label{hpns}
h'(-s) = \lim_{a\to(-s)+} \frac{1}{\d} \Big[1 - \frac{1}{2h(a)}(a-s)(a+s)\Big]
=\frac{1}{\d}\Big[1+\frac{s}{L}\Big] = L = -\frac{2s}{1+\sqrt{1+4\d s}}.
\end{equation}
On the other hand, we have 
\[
p'(a) = \frac{\l}{2}\sqrt{s} (u_0-a)^{-\frac{1}{2}}(a+s)
-\l\sqrt{s} (u_0-a)^\frac{1}{2}.
\]
Then, it follows that 
\[
\lim_{a\to u_0-} p'(a)(u_0-a)^\frac{1}{2} = \frac{\l}{2}\sqrt{s} (u_0+s), \qquad
p'(-s)=-\l\sqrt{s}(u_0+s)^\frac{1}{2}.
\]
Now we compare \(h'(a)\) and \(p'(a)\) at the endpoints.
Since \(k\le1\), or equivalently, \(\l \le \sqrt{\frac{1}{A}}\sqrt{\frac{u_0-s}{u_0+s}}\),
\begin{equation} \label{hppp-u0}
\lim_{a\to u_0-} h'(a)(u_0-a)^\frac{1}{2}
=\frac{1}{2}\sqrt{\frac{u_0^2-s^2}{\d}}
>\frac{\l}{2}\sqrt{s} (u_0+s)
=\lim_{a\to u_0-} p'(a)(u_0-a)^\frac{1}{2}.
\end{equation}
Moreover, under the assumption \(\d s<A\), the condition \eqref{RM-ass} directly implies that
\begin{equation} \label{hppp-ns}
-h'(-s)=\frac{2s}{1+\sqrt{1+4\d s}}
>\l\sqrt{s}(u_0+s)^\frac{1}{2}=-p'(-s).
\end{equation}

We now proceed to prove \eqref{RM-ineq2} by contradiction, i.e., we assume that there exists \(a\in(-s,u_0)\) such that \(h(a)>p(a)\).
Then, the earlier observations \eqref{h-p} and \eqref{hppp-u0}-\eqref{hppp-ns} imply that there exist two points \(b\) and \(c\) with \(-s<b<c<u_0\) such that
\begin{align*}
&h(b)=p(b), &&h'(b)-p'(b) \ge 0\\
&h(c)=p(c), &&h'(c)-p'(c) \le 0.
\end{align*}
We consider a function \(g:[-s,u_0]\to\RR\) which is defined as 
\[
g(a) \coloneqq
2\l\sqrt{s}(u_0-a)^\frac{1}{2}
+ (a-s)
-\l^2 \d s(a+s)
+2\l^2 \d s (u_0-a).
\]
Then, when \(a=b\) or \(a=c\), we have
\begin{align*}
h'(a)-p'(a)
&=\frac{1}{\d}\Big[1-\frac{1}{2p(a)}(a-s)(a+s)\Big]
-\frac{\l}{2}\sqrt{s}(u_0-a)^{-\frac{1}{2}} (a+s)
+\l\sqrt{s}(u_0-a)^\frac{1}{2}\\
&=\frac{1}{\d}\Big[1+\frac{a-s}{2\l\sqrt{s}(u_0-a)^\frac{1}{2}}\Big]
-\frac{\l}{2}\sqrt{s}(u_0-a)^{-\frac{1}{2}} (a+s)
+\l\sqrt{s}(u_0-a)^\frac{1}{2},
\end{align*}
it follows that 
\[
g(b) = 2\l \d \sqrt{s} (u_0-b)^\frac{1}{2} (h'(b)-p'(b)) \ge0, \qquad
g(c) = 2\l \d \sqrt{s} (u_0-c)^\frac{1}{2} (h'(c)-p'(c)) \le0.
\]
Moreover, since \(k\le1\), or equivalently, \(\l \le \sqrt{\frac{1}{A}}\sqrt{\frac{u_0-s}{u_0+s}}\), we find that
\begin{equation} \label{gu0}
g(u_0) = (u_0-s) -\l^2 \d s (u_0+s)
\ge (u_0-s) - \frac{\d s}{A} (u_0-s) >0,
\end{equation}
and we claim that
\begin{equation} \label{gns}
g(-s)
=2\l\sqrt{s}(u_0+s)^\frac{1}{2} +2\l^2\d s(u_0+s) -2s
<0.
\end{equation}
This can be verified as follows: \eqref{gns} is equivalent to 
\[
\frac{k}{\sqrt{A}}\sqrt{\frac{u_0-s}{s}}
+k^2\frac{\d s}{A}\frac{u_0-s}{s} <1,
\]
which can be immediately established by the assumption \(\d s<A\) and the condition \eqref{RM-ass}:
\[
\frac{k}{\sqrt{A}}\sqrt{\frac{u_0-s}{s}}
+k^2\frac{\d s}{A}\frac{u_0-s}{s}
< \frac{2}{1+\sqrt{1+4A}} + \frac{4A}{(1+\sqrt{1+4A})^2}
=1.
\]
On the other hand, the function \(g\) is strictly concave:
\[
g''(a)
=-\frac{1}{2}\l\sqrt{s}(u_0-a)^{-\frac{3}{2}}<0.
\]
In summary, the concave function \(g\) satisfies
\[
g(-s) < 0, \qquad
g(b) \ge 0, \qquad
g(c) \le 0, \qquad
g(u_0) > 0, \qquad
(-s<b<c<u_0)
\]
which gives a contradiction.
This completes the proof of \eqref{RM-ineq2} and \eqref{RM-ineq}.

\vspace{2mm}
We now prove \eqref{u0}.
We first note by the definition of the effective energy \(E\) that
\[
E(+\infty)-E(\x_0)
=\int_{\x_0}^\infty (\util')^2 d\x
=\frac{1}{6}(2s-u_0)(u_0+s)^2.
\]
Then, using \eqref{RM-ineq}, we obtain
\begin{align*}
\frac{1}{6}(2s-u_0)(u_0+s)^2
=\int_{\x_0}^\infty (\util')^2 d\x
&\ge \l \sqrt{s} \int_{\x_0}^\infty (u_0-\util)^\frac{1}{2} (\util+s) (-\util')d\x\\
&= \l \sqrt{s} \int_{-s}^{u_0} (u_0-\util)^\frac{1}{2} (\util+s) d\util
=\frac{4}{15}\l\sqrt{s}(u_0+s)^\frac{5}{2}.
\end{align*}
This yields that 
\[
2s-u_0
\ge \frac{8}{5}\l \sqrt{s} \sqrt{u_0+s}
=\frac{8}{5} \sqrt{\frac{1}{A}} k \sqrt{s} \sqrt{u_0-s}.
\]
Then, letting \(u_0=(1+\a)s\), we obtain
\[
1-\a \ge \frac{8}{5} \sqrt{\frac{1}{A}} k \sqrt{\a},
\]
from which it follows that for \(k_*=\big(1+\frac{32k^2}{25A}\big)\),
\[
\a^2-2k_*\a+1 \ge 0.
\]
This yields that 
\[
\a \le k_* -\sqrt{k_*^2-1},
\]
and thus the desired upper bound \eqref{u0} holds.
This completes the proof of \eqref{u0}.

\vspace{2mm}
Finally, we show \eqref{RM-5/4}.
In the proof, we may assume that \(A=1\) and we apply a bootstrapping argument as follows: since only \(u_0<2s\) is initially available, we choose \(k=\frac{1}{2}\) so that \eqref{RM-ass} holds.
Then, \eqref{u0} implies that 
\[
u_0 \le \frac{3}{2}s.
\]
This allows the choice of a larger \(k\).
We now choose \(k=\frac{4}{5}\).
This together with \eqref{u0} implies that
\[
u_0 \le \frac{13}{10}s.
\]
Now \(k=1\) is available.
Then, using \eqref{u0}, we obtain \eqref{RM-5/4}, which completes the proof.
\end{pf}

\vspace{2mm}
The following proposition provides a much more accurate approximation of the shock derivative.
This will be used to establish \eqref{u0-upper}.

\begin{proposition} \label{prop:RM}
Let \(\l\) and \(\m\) be constants given as
\begin{equation} \label{mu-lda}
\l=\frac{1}{\sqrt{A}}\sqrt{\frac{u_0-s}{u_0+s}}, \qquad
\m=\frac{\frac{2s}{1+\sqrt{1+4A}}-\frac{1}{\sqrt{A}}\sqrt{s(u_0-s)}}{u_0+s}.
\end{equation}
Then, the following holds:
\begin{equation} \label{RM-ineq_2} 
-\util'(\x) \ge
\l \sqrt{s} (u_0-\util(\x))^\frac{1}{2}(\util(\x)+s)
+\m (u_0-\util(\x))(\util(\x)+s), \quad \forall\, \x\in(\x_0,\infty).
\end{equation}
\end{proposition}

\begin{pf}
First of all, we note by \eqref{RM-5/4} that \(\m>0\).
As in Lemma \ref{lem:RM}, we introduce two functions of \(a=\util\) defined as follows: 
\[
h(a) \coloneqq \util' = \util_\x, \qquad
p(a) \coloneqq
-\l\sqrt{s} (u_0-a)^\frac{1}{2} (a+s)
-\m(u_0-a)(a+s).
\]
Then, the desired inequality \eqref{RM-ineq} boils down to
\begin{equation} \label{RM-ineq2_2}
h(a) \le p(a), \quad \forall a\in(-s,u_0).
\end{equation}
In the same way, we examine the two functions at the endpoints as follows:
\begin{equation} \label{h-p_2}
h(u_0)=p(u_0)=0, \qquad
h(-s)=p(-s)=0.
\end{equation}
We also recall \eqref{hpu0} and \eqref{hpns}: 
\[
\lim_{a\to u_0-} h'(a)(u_0-a)^\frac{1}{2}
=\frac{1}{2}\sqrt{\frac{u_0^2-s^2}{\d}}, \qquad
h'(-s) = -\frac{2s}{1+\sqrt{1+4\d s}}.
\]
On the other hand, since we have
\[
p'(a) = \frac{\l}{2}\sqrt{s} (u_0-a)^{-\frac{1}{2}}(a+s)
-\l\sqrt{s} (u_0-a)^\frac{1}{2}
+\m(2a+s-u_0),
\]
it follows that
\[
\lim_{a\to u_0-} p'(a)(u_0-a)^\frac{1}{2} = \frac{\l}{2}\sqrt{s} (u_0+s), \qquad
p'(-s) = -\l\sqrt{s}(u_0+s)^\frac{1}{2} -\m(u_0+s).
\]
Then, using \eqref{mu-lda} with \(\d s<A\), we obtain
\begin{equation} \label{hppp-u0_2}
\lim_{a\to u_0-} h'(a)(u_0-a)^\frac{1}{2}
=\frac{1}{2}\sqrt{\frac{u_0^2-s^2}{\d}}
>\frac{\l}{2}\sqrt{s} (u_0+s)
=\lim_{a\to u_0-} p'(a)(u_0-a)^\frac{1}{2}.
\end{equation}
Moreover, using \eqref{mu-lda} with \(\d s<A\) once again, we find that 
\[
\l\sqrt{s}(u_0+s)^\frac{1}{2} +\m(u_0+s)
=\frac{1}{\sqrt{A}}\sqrt{s(u_0-s)}
+\frac{2s}{1+\sqrt{1+4A}}-\frac{1}{\sqrt{A}}\sqrt{s(u_0-s)}
=\frac{2s}{1+\sqrt{1+4A}}
\]
and thus, we obtain
\begin{equation} \label{hppp-ns_2}
-h'(-s)=\frac{2s}{1+\sqrt{1+4\d s}}
>\l\sqrt{s}(u_0+s)^\frac{1}{2}
+\m(u_0+s)=-p'(-s).
\end{equation}

We now assume the contrary, i.e., we assume that there exists \(a\in(-s,u_0)\) such that \(h(a)>p(a)\).
Thanks to \eqref{h-p_2}, \eqref{hppp-u0_2} and \eqref{hppp-ns_2}, there exist two points \(b\) and \(c\) with \(-s<b<c<u_0\) such that
\begin{align*}
&h(b)=p(b), &&h'(b)-p'(b) \ge 0\\
&h(c)=p(c), &&h'(c)-p'(c) \le 0.
\end{align*}
Then, we define a function \(g:[-s,u_0]\to\RR\) by
\begin{align*}
g(a) &\coloneqq
2\l\sqrt{s}(u_0-a)^\frac{1}{2}
+2\m(u_0-a)
+(a-s)
-\l^2 \d s (a+s)
-3\l\m\d\sqrt{s}(u_0-a)^\frac{1}{2}(a+s) \\
&\qquad
+2\l^2 \d s (u_0-a)
+4\l\m\d\sqrt{s}(u_0-a)^\frac{3}{2}
-2\m^2\d(u_0-a)(2a+s-u_0).
\end{align*}
Since \(h(b)=p(b)\) and \(h(c)=p(c)\), it follows that when \(a=b\) or \(a=c\), 
\begin{align*}
h'(a)-p'(a)
&=\frac{1}{\d}\Big[1+\frac{a-s}{2\l\sqrt{s}(u_0-a)^\frac{1}{2}+2\m(u_0-a)}\Big]\\
&\qquad-\frac{\l}{2}\sqrt{s}(u_0-a)^{-\frac{1}{2}} (a+s)
+\l\sqrt{s}(u_0-a)^\frac{1}{2}
-\m(2a+s-u_0).
\end{align*}
Thus, we have
\begin{align*}
g(b) = \Big(2\l \d \sqrt{s} (u_0-b)^\frac{1}{2} + 2\d\m(u_0-a)\Big) (h'(b)-p'(b)) \ge0,\\
g(c) = \Big(2\l \d \sqrt{s} (u_0-c)^\frac{1}{2} + 2\d\m(u_0-a)\Big) (h'(c)-p'(c)) \le0.
\end{align*}
We also observe that
\begin{equation} \label{gu0_2}
g(u_0) = (u_0-s) -\l^2 \d s (u_0+s)
= (u_0-s) - \frac{\d s}{A} (u_0-s) >0.
\end{equation}
Moreover, since 
\[
\l\sqrt{s}(u_0+s)^\frac{1}{2}
+\m(u_0+s)
=\frac{2s}{1+\sqrt{1+4A}}
\]
and 
\begin{align*}
g(-s)
&=2\Big(\l\sqrt{s}(u_0+s)^\frac{1}{2}
+\m(u_0+s)\Big)
+2\l\d\sqrt{s} (u_0+s)^\frac{1}{2} \Big(\l\sqrt{s}(u_0+s)^\frac{1}{2}
+\m(u_0+s)\Big)\\
&\qquad
+2\m\d (u_0+s)\Big(\l\sqrt{s}(u_0+s)^\frac{1}{2}
+\m(u_0+s)\Big)
-2s,
\end{align*}
we find that 
\begin{equation}
\begin{aligned} \label{gns_2}
g(-s)
&=\frac{2s}{1+\sqrt{1+4A}}
\Big(2 + \frac{4\d s}{1+\sqrt{1+4A}}\Big)
-2s\\
&<\frac{2s}{1+\sqrt{1+4A}}
\Big(2 + \frac{4A}{1+\sqrt{1+4A}}\Big)
-2s =0.
\end{aligned}
\end{equation}
Now, we introduce \(y\coloneqq (u_0-a)^\frac{1}{2}\).
Then, \(g(a)\) can be seen as a function of \(y\) as follows:
\begin{align*}
G(y)&\coloneqq
2\l\sqrt{s}y
+2\m y^2
+(u_0-s-y^2)
-\l^2 \d s (u_0+s-y^2)
-\l\m\d\sqrt{s}y(u_0+s-y^2)\\
&\qquad
+2\l^2\d s y^2
+2\l\m\d\sqrt{s}y^3
-2\l\m\d\sqrt{s}y(s+u_0-2y^2)
-2\m^2\d y^2(s+u_0-2y^2)
(=g(a)).
\end{align*}
This is a quadratic polynomial function whose coefficients of the fourth-degree and third-degree terms are positive---namely, 
\[
G(y) = 4\m^2 \d y^4
+7\l\m\d\sqrt{s} y^3
+\text{lower order terms}\cdots.
\]
Thus, the third derivative \(G'''(y)\) is nonnegative, implying that the second derivative \(G''(y)\) is increasing.
Hence, the convexity can change at most once, from concave to convex.
This contradicts to the fact that \(-s<b<c<u_0\) and \(0<(u_0-c)^\frac{1}{2}<(u_0-b)^\frac{1}{2}<(u_0+s)^\frac{1}{2}\) with 
\[
G(0)>0, \qquad
G((u_0-c)^\frac{1}{2}) \le 0, \qquad
G((u_0-b)^\frac{1}{2}) \ge 0, \qquad
G((u_0+s)^\frac{1}{2}) <0.
\]
This completes the proof of \eqref{RM-ineq2_2} and \eqref{RM-ineq_2}.
\end{pf}

\vspace{2mm}
We are now ready to prove \eqref{u0-upper} in Theorem \ref{thm:shock}, whose proof relies on the approximation of the derivative of the shock profile given in Proposition \ref{prop:RM}.

\paragraph{Proof of \eqref{u0-upper} in Theorem \ref{thm:shock}.}
As in the proof of \eqref{u0}, using \eqref{RM-ineq_2} with the effective energy, we obtain
\begin{align*}
\frac{1}{6}(2s-u_0)(u_0+s)^2
&=E(+\infty)-E(\x_0)
=\int_{\x_0}^\infty (\util_\x)^2 d\x\\
&\ge \l \sqrt{s} \int_{\x_0}^\infty (u_0-\util)^\frac{1}{2} (\util+s) (-\util_\x)d\x
+\m \int_{\x_0}^\infty (u_0-\util) (\util+s) (-\util_\x)d\x\\
&= \l \sqrt{s} \int_{-s}^{u_0} (u_0-\util)^\frac{1}{2} (\util+s) d\util
+\m \int_{-s}^{u_0} (u_0-\util) (\util+s) d\util\\
&=\frac{4}{15}\l\sqrt{s}(u_0+s)^\frac{5}{2}
+\frac{\m}{6}(u_0+s)^3.
\end{align*}
Then, using \eqref{mu-lda}, the above inequality can be rewritten as 
\[
2s-u_0
\ge \frac{3\sqrt{s}}{5\sqrt{A}} (u_0-s)^\frac{1}{2}
+\frac{2s}{1+\sqrt{1+4A}}.
\]
Let \(u_0=(1+\a)s\).
The above inequality is equivalent to
\[
1-\a \ge \frac{3\sqrt{\a}}{5\sqrt{A}} + \frac{2}{1+\sqrt{1+4A}}
\]
and it follows that 
\[
\sqrt{\a} \le -\frac{3}{10\sqrt{A}}
+\frac{1}{2}\sqrt{\frac{9}{25A}-4\Big(\frac{2}{1+\sqrt{1+4A}}-1\Big)}.
\]
This immediately implies the desired estimate \eqref{u0-upper}. \qed

\subsection{Decay Rates of Oscillations} \label{sec:other_cycles}
We now turn our attention to the other local extrema.
In particular, we quantify the rate at which the local extrema \(u_i\) (see Figure \ref{pp} for the definition of \(u_i\)) approach the left end state \(u_-=s\) as \(i\) increases.
These decay rates are crucial when deriving the key inequalities and the \(L^2\) estimate in the following section.

\vspace{2mm}
The key ingredient of the proof is the construction of an approximation curve of the shock derivative in the \((\util,\util')\)-plane.
We divide the analysis into two cases: the increasing and decreasing intervals of the shock profiles.

\vspace{2mm}
We first consider the increasing intervals of the shock profile, namely the intervals \(I_i\coloneqq(\x_i,\x_{i-1})\) for odd \(i\).
Note that for odd \(i\), \(\util'>0\) on \(I_i\) and 
\[
\util(\x_i)=u_i <s< \util(\x_{i-1})=u_{i-1}, \qquad
u_{i-1}-s > s-u_i, \qquad
\util'(\x_i)=\util'(\x_{i-1}) =0.
\]
In fact, \(u_{i-1}-s > s-u_i\) needs to be justified.
At this stage, the only available estimates are \eqref{u0-upper}, i.e., \(u_{i-1}\le u_0 \le 1.15s\) for \(A\le1\), and \(u_i>-s\).
These bounds are sufficient to establish the key inequality \eqref{ineq_03} with \(\l_i=\frac{3}{10}\) for all odd \(i\).
In other words, the proof of \textit{Case 2} in Proposition \ref{prop:key} works with these two bounds and \(\lbar_i=\frac{3}{10}\).
Combined with the argument in the proof of \eqref{inc-decay}, based on the effective energy, \(u_{i-1}-s > s-u_i\) can be obtained.
We omit the further details.

\vspace{2mm}
Then, we present the following lemma, which provides an approximation of the shock derivative.

\begin{proposition} \label{prop:inc}
Let \(A\) be a constant which satisfies \(\frac{1}{4}<A\le1\).\\
Then, for each \(i\in\NN\) with \(i\) odd and for \(\l_i=\sqrt{\frac{1}{A}} \sqrt{\frac{s(s^2-u_i^2)}{u_{i-1}-u_i}}\), the following holds:
\begin{equation} \label{inc-ineq}
\util'(\x) \ge \l_i(u_{i-1}-\util(\x))^\frac{1}{2} (\util(\x)-u_i)^\frac{1}{2},
\quad \forall \x \in I_i.
\end{equation}
\end{proposition}

\begin{pf}
The proof relies on the contradiction argument introduced above.
To this end, we first consider two functions in \(a=\util\) as follows: 
\[
h(a) \coloneqq \util' = \util_\x, \qquad
p(a) \coloneqq \l_i (u_{i-1}-a)^\frac{1}{2} (a-u_i)^\frac{1}{2}.
\]
The desired inequality \eqref{inc-ineq} is equivalent to
\begin{equation} \label{inc-ineq2}
h(a) \ge p(a), \quad \forall a \in (u_i, u_{i-1}).
\end{equation}
We proceed to investigate the two functions at the endpoints.
First of all, it is trivial that 
\[
h(u_i)=p(u_i)=0, \qquad
h(u_{i-1})=p(u_{i-1})=0,
\]
as \(\util(\x)\) has its local extrema at \(\x_i\) and \(\x_{i-1}\).
Then, we recall \eqref{hp}:
\[
h'(a) = \frac{d\util_\x}{d\util} = \frac{1}{\d} \Big[1 - \frac{1}{2h(a)}(a-s)(a+s)\Big].
\]
and we observe 
\[
p'(a)=
-\frac{\l_i}{2} (u_{i-1}-a)^{-\frac{1}{2}} (a-u_i)^\frac{1}{2}
+\frac{\l_i}{2} (u_{i-1}-a)^\frac{1}{2} (a-u_i)^{-\frac{1}{2}}.
\]
Both \(h'(a)\) and \(p'(a)\) diverge when \(a \to u_i\) and \(a \to u_{i-1}\).
The divergence rate of \(p'(a)\) is as follows: 
\[
\lim_{a\to u_i+} p'(a)(a-u_i)^\frac{1}{2} = \frac{\l_i}{2} (u_{i-1}-u_i)^\frac{1}{2}, \qquad
\lim_{a\to u_{i-1}-} p'(a)(u_{i-1}-a)^\frac{1}{2} = -\frac{\l_i}{2} (u_{i-1}-u_i)^\frac{1}{2}.
\]
To find the divergence rate of \(h'(a)\), we apply L'Hôpital's rule and observe
\[
\lim_{a\to u_i+} \frac{h(a)}{(a-u_i)^\frac{1}{2}}
=\lim_{\x\to \x_i+} \frac{\util_\x}{(\util-u_i)^\frac{1}{2}}
=\lim_{\x\to \x_i+} \frac{2(\util-u_i)^\frac{1}{2} \util_{\x\x}}{\util_\x}.
\]
Since it holds by \eqref{t0} that \(\util_{\x\x}(\x_i)= \frac{s^2-u_i^2}{2\d}\), we obtain
\[
\lim_{a\to u_i+} \frac{h(a)}{(a-u_i)^\frac{1}{2}}
=\sqrt{\frac{s^2-u_i^2}{\d}}.
\]
It follows that
\begin{equation} \label{hp-ui}
\lim_{a\to u_i+} h'(a)(a-u_i)^\frac{1}{2}
=\frac{1}{2}\sqrt{\frac{s^2-u_i^2}{\d}}.
\end{equation}
Likewise, we obtain
\begin{equation} \label{hp-ui-1}
\lim_{a\to u_{i-1}-} h'(a)(u_{i-1}-a)^\frac{1}{2}
=-\frac{1}{2}\sqrt{\frac{u_{i-1}^2-s^2}{\d}}.
\end{equation}
We now compare the derivatives \(h'\) and \(p'\).
Since we have
\[
\frac{s^2-u_i^2}{\d}
>\frac{s(s^2-u_i^2)}{A}
=\frac{1}{A} \frac{s(s^2-u_i^2)}{(u_{i-1}-u_i)} (u_{i-1}-u_i)
=\l_i^2 (u_{i-1}-u_i),
\]
it follows that
\[
\lim_{a\to u_i+} h'(a)(a-u_i)^\frac{1}{2}
=\frac{1}{2}\sqrt{\frac{s^2-u_i^2}{\d}}
> \frac{\l_i}{2} (u_{i-1}-u_i)^\frac{1}{2}
=\lim_{a\to u_i+} p'(a)(a-u_i)^\frac{1}{2}.
\]
Moreover, using
\[
\frac{u_{i-1}^2-s^2}{\d}
>\frac{s(u_{i-1}^2-s^2)}{A}
>\frac{s(s^2-u_i^2)}{A}
=\frac{1}{A} \frac{s(s^2-u_i^2)}{(u_{i-1}-u_i)} (u_{i-1}-u_i)
=\l_i^2 (u_{i-1}-u_i),
\]
we obtain
\[
-\lim_{a\to u_{i-1}-} h'(a)(u_{i-1}-a)^\frac{1}{2}
=\frac{1}{2}\sqrt{\frac{u_{i-1}^2-s^2}{\d}}
>
\frac{\l_i}{2} (u_{i-1}-u_i)^\frac{1}{2}
=-\lim_{a\to u_{i-1}-} p'(a)(u_{i-1}-a)^\frac{1}{2}.
\]

We suppose that there exists \(a \in (u_i,u_{i-1})\) such that \(h(a)<p(a)\) and show that this leads to a contradiction.
From the observations above, we choose two points \(b\) and \(c\) with \(u_i<b<c<u_{i-1}\) which satisfy
\begin{align*}
&h(b)=p(b), &&h'(b)-p'(b) \le 0\\
&h(c)=p(c), &&h'(c)-p'(c) \ge 0.
\end{align*}
We then consider a function \(g:[u_i,u_{i-1}]\to\RR\) which is given by
\[
g(a) \coloneqq
\l_i(u_{i-1}-a)^\frac{1}{2}(a-u_i)^\frac{1}{2}
-\frac{1}{2}(a-s)(a+s)
+\frac{1}{2}\l_i^2\d(2a-u_i-u_{i-1}).
\]
Note that when \(a=b\) or \(a=c\), we have 
\[
h'(a)-p'(a)
=\frac{1}{\d}\Big[1-\frac{1}{2p(a)}(a-s)(a+s)\Big]
+\frac{\l_i}{2}(u_{i-1}-a)^{-\frac{1}{2}} (a-u_i)^\frac{1}{2}
-\frac{\l_i}{2}(u_{i-1}-a)^\frac{1}{2} (a-u_i)^{-\frac{1}{2}}.
\]
This implies that 
\begin{align*}
g(b) &= \l_i \d (h'(b)-p'(b)) (u_{i-1}-b)^\frac{1}{2} (b-u_i)^\frac{1}{2} \le0,\\
g(c) &= \l_i \d (h'(c)-p'(c)) (u_{i-1}-c)^\frac{1}{2} (c-u_i)^\frac{1}{2} \ge0.
\end{align*}
At the endpoints, we also have 
\[
g(u_i)
= \frac{1}{2} (s^2-u_i^2) + \frac{\l_i^2\d}{2} (u_i-u_{i-1})
= \frac{s^2-u_i^2}{2} \Big(1-\frac{\d s}{A}\Big) >0,
\]
and 
\[
g(u_{i-1})
= \frac{1}{2} (s^2-u_{i-1}^2) + \frac{\l_i^2\d}{2} (u_{i-1}-u_i)
= -\frac{1}{2} \Big( (u_{i-1}^2-s^2) -\frac{\d s}{A} (s^2 - u_i^2) \Big) <0.
\]
On the other hand, we observe
\[
g''(a)=-1 -\frac{\l_i}{4}
(u_{i-1}-a)^{-\frac{3}{2}}(a-u_i)^{-\frac{3}{2}} (u_{i-1}-u_i)^2<0.
\]
In conclusion, the concave function \(g\) satisfies 
\[
g(u_i) > 0, \qquad
g(b) \le 0, \qquad
g(c) \ge 0, \qquad
g(u_{i-1}) < 0, \qquad
(u_i<b<c<u_{i-1})
\]
which gives a contradiction.
This completes the proof of \eqref{inc-ineq2} and \eqref{inc-ineq}.
\end{pf}

\vspace{2mm}
Using the approximation in Proposition \ref{prop:inc}, we now prove \eqref{inc-decay} in Theorem \ref{thm:shock}.

\paragraph{Proof of \eqref{inc-decay} in Theorem \ref{thm:shock}.}
We recall \eqref{t1}-\eqref{t2} and use Proposition \ref{prop:inc} to find that
\begin{align*}
E(\x_{i-1})-E(\x_i)
&=\frac{1}{6}(u_{i-1}-u_i)(u_{i-1}^2+u_{i-1}u_i+u_i^2-3s^2)
=\int_{\x_i}^{\x_{i-1}} (\util')^2 d\x\\
&\ge \l_i \int_{\x_i}^{\x_{i-1}} (u_{i-1}-\util)^\frac{1}{2} (\util-u_i)^\frac{1}{2} \util' d\x\\
&= \l_i \int_{u_i}^{u_{i-1}} (u_{i-1}-\util)^\frac{1}{2} (\util-u_i)^\frac{1}{2} d\util
= \frac{\pi}{8} \l_i (u_{i-1}-u_i)^2.
\end{align*}
This is equivalent to 
\[
(u_{i-1}^2+u_{i-1}u_i+u_i^2-3s^2)
\ge \frac{3}{4}\pi \sqrt{\frac{1}{A}} \sqrt{s(s^2-u_i^2)} (u_{i-1}-u_i)^\frac{1}{2}.
\]
Letting \(u_{i-1}=(1+\a)s\) and \(u_i=(1-\b)s\), we rewrite the above inequality as follows:
\[
3(\a-\b)+\a^2+\b^2-\a\b \ge \frac{3}{4}\pi \sqrt{\frac{1}{A}} \sqrt{2\b-\b^2} \sqrt{\a+\b}.
\]
Then, we set \(\r\coloneqq\frac{\a}{\b}=\frac{u_{i-1}-s}{s-u_i}>1\) and find that
\begin{equation} \label{inc-ui0}
3(\r-1)+(\r^2-\r+1) \frac{\a}{\r}
\ge \frac{3}{4}\pi \sqrt{\frac{1}{A}}\sqrt{2-\frac{\a}{\r}}\sqrt{\r+1}.
\end{equation}
We now exploit the upper bound on \(u_0\) in \eqref{u0-upper} to provide an upper bound on \(\a\).
We choose \(\a_0>0\) such that \(u_0 \le (1+\a_0)s\).
Then the upper bound on \(u_0\) in \eqref{u0-upper} determines an admissible choice of \(\a_0\), and it follows that \(\a \le \a_0\).
However, since the \(A\)-dependence of the bound \eqref{u0-upper} is nontrivial and highly nonlinear, we restrict our analysis to several selected values of \(A\), namely \(A=1/3,1/2,2/3,3/4,1\), and work with the corresponding bounds.
In each case, the explicit bounds are given as follows:
\begin{equation} \label{u0-exp}
u_0 \le
\begin{cases}
1.03s, & A = 1/3,\\
1.0601s, & A = 1/2,\\
1.092s, & A = 2/3,\\
1.11s, & A = 3/4,\\
1.15s, & A = 1,
\end{cases} \qquad \qquad
\a_0 =
\begin{cases}
0.03, & A = 1/3,\\
0.0601, & A = 1/2,\\
0.092, & A = 2/3,\\
0.11, & A = 3/4,\\
0.15, & A = 1.
\end{cases}
\end{equation}
Since \(\r^2-\r+1 \ge 0\) and \(0<\a\le\a_0\), \eqref{inc-ui0} yields that 
\begin{equation} \label{inc-ui}
3(\r-1)+(\r^2-\r+1) \frac{\a_0}{\r}
\ge \frac{3}{4}\pi \sqrt{\frac{1}{A}}\sqrt{2-\frac{\a_0}{\r}}\sqrt{\r+1},
\end{equation}
which fails when \(\r\) is close to \(1\).
To obtain the desired decay rate \(\r_*\), we examine this inequality more closely.
We define a function \(l(\r)\) as the difference between the left- and right- hand sides: 
\[
l(\r)\coloneqq 3(\r-1)+(\r^2-\r+1) \frac{\a_0}{\r}
- \frac{3}{4}\pi \sqrt{\frac{1}{A}}\sqrt{2-\frac{\a_0}{\r}}\sqrt{\r+1}.
\]
For each of the selected \(A\), elementary calculations show that \(l(\r)\) is increasing for \(\r>1\); therefore, the inequality \eqref{inc-ui} holds if \(\r\) is greater than the (unique) solution of \(l(\r)=0\), and fails otherwise.
The inequality \eqref{inc-ui} indeed fails for each value of \(A\) when we substitute the corresponding \(\a\) from \eqref{u0-exp} and \(\r_*\) from Table \ref{tab:decay}.
Therefore, the decay rate is larger than \(\r_*\), which proves \eqref{inc-decay}.
\qed

\vspace{2mm}
The idea of the proof is similar to the case of the increasing intervals, while it is technically a bit more complicated---in particular, it requires a bootstrapping argument.
We leave the proof of \eqref{dec-decay} in Theorem \ref{thm:shock} in the end of the paper in Appendix \ref{appendix:dec}.

\section{Proof of Theorem \ref{thm:main}: Contraction Property} \label{sec:main}
\setcounter{equation}{0}
This section is devoted to the proof of Theorem \ref{thm:main}, which establishes the \(L^2\)-contraction property for viscous-dispersive shocks.
As discussed in Section \ref{sec:idea}, the proof relies on a Poincar\'e-type inequality (Lemma \ref{KV_lemma}) and to make this applicable, we develop an inductive argument.
We also introduce the change of variables in \eqref{ztu}; accordingly, the inequalities \eqref{ineq_00}, \eqref{ineq_03} and \eqref{ineq_04} are required in order to exploit the diffusion term.
In addition, the inductive argument necessitates suitable \(L^2\) estimates for the shock profile.
Each of these ingredients is established in the subsequent subsections: the relevant inequalities are derived in Section \ref{sec:key}, while the \(L^2\) estimates are obtained in Section \ref{section_l2}.

\vspace{2mm}
In what follows, we consider only the case \(A=\frac{1}{2}\), for which the following bounds are available: 
\begin{equation} \label{A1/2}
u_0 \le (1+r)s \coloneqq 1.0601s, \qquad
\frac{u_{i-1}-s}{s-u_i} \ge \r_* \coloneqq 4.64, \qquad
\frac{s-u_i}{u_{i+1}-s} \ge \r^* \coloneqq 4.77,
\end{equation}
for each odd \(i\in\NN\).
These bounds correspond to the structural properties of the shock identified in Theorem \ref{thm:shock} and will be used crucially in the subsequent analysis.

\subsection{Key Inequality for Each Monotonic Interval} \label{sec:key}
In this subsection, we show the inequalities \eqref{ineq_00}, \eqref{ineq_03} and \eqref{ineq_04} on each maximal interval over which the shock profile is monotone.
We refer to these inequalities as the key inequalities, since they serve as Jacobian estimates associated with the change of variables in \eqref{ztu}.
They play a crucial role in establishing \(L^2\) estimates and are stated in the following proposition.

\begin{proposition} \label{prop:key}
The following holds: for each \(i\in\NN\cup\{0\}\) even,
\begin{equation} \label{key:dec}
-\util'(\x) \ge \lbar_i(u_i-\util(\x))(\util(\x)-u_{i-1}), \qquad \forall \x\in(\x_i,\x_{i-1}),
\end{equation}
and for each \(i\in\NN\) odd, 
\begin{equation} \label{key:inc}
\util'(\x) \ge \lbar_i(u_{i-1}-\util(\x))(\util(\x)-u_i), \qquad \forall \x\in(\x_i,\x_{i-1}),
\end{equation}
where \(\x_{-1}=+\infty\) and \(u_{-1}=\util(\x_{-1})=-s\), and \(\lbar_i\) are given by 
\begin{equation} \label{key:l}
\lbar_0 = 0.355, \qquad
\lbar_1 = 9.60, \qquad
\lbar_{2n} = (1.94) \r_*^{2n}, \qquad
\lbar_{2n+1} = (0.51)^{-1} \r_*^{2n+1}.
\end{equation}
\end{proposition}

\begin{pf}
The proof is based on the contradiction argument introduced in the present paper, and we split the proof into three cases: \(i=0\), odd \(i\), and even \(i\ge2\).

\vspace{2mm} \case{1} \(i=0\).
First of all, we define two functions in \(a=\util\):
\[
h(a)\coloneqq\util'=\util_\x, \qquad
p(a) \coloneqq \lbar_0(a-u_0)(a+s).
\]
Then, it is equivalent to show that 
\begin{equation} \label{ineq2}
h(a) \le p(a), \quad \forall a \in (-s,u_0).
\end{equation}
We proceed to analyze the behavior of the two functions at the endpoints.
It is trivial that 
\begin{equation} \label{h-p=0}
h(u_0)=p(u_0)=0, \qquad
h(-s)=p(-s)=0.
\end{equation}
Then, since \(p'(a)=\lbar_0(2a+s-u_0)\), we note that
\begin{equation} \label{hppp1-u0}
\lim_{a\to u_0-} h'(a) = +\infty > \lbar_0(u_0+s) = \lim_{a\to u_0-} p'(a).
\end{equation}
Moreover, since \(u_0\le 1.0601s\), \(\lbar_0\) satisfies
\begin{equation} \label{RM-keykey}
0.355 = \lbar_0
\le \frac{2}{1+\sqrt{3}} \cdot \frac{s}{u_0+s}
= \frac{2}{1+\sqrt{1+4A}} \cdot \frac{s}{u_0+s},
\end{equation}
and hence it follows that (recall \eqref{hpns})
\begin{equation} \label{hppp1-ns}
-h'(-s)
=\frac{2s}{1+\sqrt{1+4\d s}}
>\frac{2s}{1+\sqrt{1+4A}}
>\lbar_0(s+u_0)
=-p'(-s).
\end{equation}

We now establish \eqref{ineq2} by contradiction and assume that \(h(a)>p(a)\) for some \(a\in(-s,u_0)\).
Then, based on the above observations \eqref{h-p=0}, \eqref{hppp1-u0} and \eqref{hppp1-ns}, we choose two points \(b\) and \(c\) with \(-s<b<c<u_0\) such that
\begin{align*}
&p(b)=h(b), &&h'(b)-p'(b) \ge 0\\
&p(c)=h(c), &&h'(c)-p'(c) \le 0.
\end{align*}
To derive a contradiction, we define a function \(g\colon[-s,u_0]\to\RR\) as follows: 
\[
g(a) \coloneqq
2\lbar_0(u_0-a)
+(a-s)
-2\lbar_0^2\d(u_0-a)(2a+s-u_0).
\]
Then, since the following holds for \(a=b\) and \(a=c\), 
\[
h'(a)-p'(a)
=\frac{1}{\d}\Big[1-\frac{1}{2p(a)}(a-s)(a+s)\Big]-\lbar_0(2a+s-u_0)
=\frac{1}{\d}\Big[1-\frac{1}{2\lbar_0}\frac{a-s}{a-u_0}\Big]-\lbar_0(2a+s-u_0),
\]
it follows that 
\[
g(b)=2\lbar_0\d(u_0-b)\big(h'(b)-p'(b)\big)\ge0, \qquad
g(c)=2\lbar_0\d(u_0-c)\big(h'(c)-p'(c)\big)\le0.
\]
We also note that
\[
g(u_0) = u_0 -s >0.
\]
Moreover, using \eqref{RM-keykey}, we find that
\[
g(-s)
=2\lbar_0(u_0+s) -2s +2\lbar_0^2\d (u_0+s)^2 < 0.
\]
However, since \(g''(a)=8\lbar_0^2 \d>0\), the function \(g\) is strictly convex, which yields a contradiction.
This completes the proof of \eqref{key:dec} for the case \(i=0\).

\vspace{2mm}
\case{2} \(i\) odd.
To begin with, we consider the following two functions of \(a=\util\): 
\[
h(a)\coloneqq\util'=\util_\x, \qquad
p(a) \coloneqq \lbar_i(u_{i-1}-a)(a-u_i).
\]
It suffices to show that
\begin{equation} \label{inc:key2}
h(a) \ge p(a), \quad \forall a \in (u_i,u_{i-1}).
\end{equation}
To this end, we compare the two functions at the endpoints.
It is obvious that 
\[
h(u_i)=p(u_i)=0, \qquad h(u_{i-1})=p(u_{i-1})=0.
\]
Since \(p'(a)=-\lbar_i(2a-u_{i-1}-u_i)\), it follows that 
\[
\lim_{a\to u_i+} h'(a) = +\infty > \lbar_i(u_{i-1}-u_i) = p'(u_i), \qquad
\lim_{a\to u_{i-1}-} h'(a) = -\infty < -\lbar_i(u_{i-1}-u_i) = p'(u_{i-1}).
\]

We now turn to the proof of \eqref{inc:key2} by contradiction, i.e., we assume that there exists \(a\in(u_i,u_{i-1})\) such that \(h(a)<p(a)\).
Then, the above observations imply that there exist two points \(b\) and \(c\) with \(u_i<b<c<u_{i-1}\) such that 
\begin{align*}
&p(b)=h(b), &&h'(b)-p'(b) \le 0\\
&p(c)=h(c), &&h'(c)-p'(c) \ge 0.
\end{align*}
We consider a function \(g:[u_i,u_{i-1}]\to\RR\) which is defined by
\[
g(a) \coloneqq
2\lbar_i(u_{i-1}-a)(a-u_i)
-(a-s)(a+s)
+2\lbar_i^2\d(u_{i-1}-a)(a-u_i)(2a-u_{i-1}-u_i).
\]
Next, we determine the sign of \(g\) at the four points \(u_i\), \(b\), \(c\) and \(u_{i-1}\), and establish that \(g\) is concave.
As in the previous case, since the following holds for \(a=b\) and \(a=c\), 
\[
h'(b)-p'(b)
=\frac{1}{\d}\Big[1-\frac{1}{2p(b)}(b-s)(b+s)\Big]+\lbar_i(2a-u_{i-1}-u_i),
\]
it follows that 
\begin{align*}
&g(b)=2\lbar_i\d(u_{i-1}-b)(b-u_i)\big(h'(b)-p'(b)\big)\le0,\\
&g(c)=2\lbar_i\d(u_{i-1}-c)(c-u_i)\big(h'(c)-p'(c)\big)\ge0.
\end{align*}
In addition, it simply follows that
\[
g(u_i) = -(u_i^2-s^2)>0, \qquad
g(u_{i-1}) = -(u_{i-1}^2-s^2)<0.
\]
On the other hand, since \(g'''(a)=-24\lbar_i^2 \d<0\), we claim that 
\[
g''(u_i) = -4\lbar_i-2+12\lbar_i^2\d(u_{i-1}-u_i) <0,
\]
so that \(g''<0\) on \((u_i,u_{i-1})\), i.e., the function \(g\) is concave.
This can be verified as follows: since 
\begin{align*}
u_{i-1}-u_i=(u_{i-1}-s)+(s-u_i)
&\le (\r^*)^{-1} (s-u_{i-2}) + (\r_*)^{-1} (u_{i-1}-s)\\
&\le (\r_*)^{-1} (u_{i-1}-u_{i-2})
\le (\r_*)^{-(i-1)} (u_0-u_1)
\le (\r_*)^{-i} (\r_*+1) rs,
\end{align*}
the constants \(\lbar_i\) given in \eqref{key:l} satisfy
\[
12\lbar_i^2 \d (u_{i-1}-u_i)
\le 12\lbar_i^2 \d s (\r_*)^{-i} (\r_*+1) r
< 6 \lbar_i^2 (\r_*)^{-i} (\r_*+1) r
\le 4\lbar_i+2.
\]
In conclusion, the concavity of \(g\) contradicts 
\[
g(u_i) > 0, \qquad
g(b) \le 0, \qquad
g(c) \ge 0, \qquad
g(u_{i-1}) < 0, \qquad
(u_i<b<c<u_{i-1}),
\]
which completes the proof of \eqref{key:dec} for the case \(i\) odd.

\vspace{2mm}
\case{3} \(i\) even.
The proof is analogous to that for odd \(i\) and is deferred to Appendix \ref{appendix:dec}.
\end{pf}

\subsection{\(L^2\) Estimates} \label{section_l2}
The remaining ingredient in the proof of the contraction property is to obtain suitable \(L^2\) estimates for the inductive argument.
First, we introduce the following notation (see Figure \ref{fig_l2}).
For each \(i\in\NN\), we set \(J_i \coloneqq (\x_i,\x^i)\) where \(\x^i\in(\x_{i-1},\x_{i-2})\) and \(\util(\x^i)=u_i\). (We also recall \(\x_{-1}=+\infty\).)
In addition, we consider the following decomposition of \(J_i\) as follows: 
\begin{equation} \label{Ji-decomp}
J_i = J_{i,1} \cup J_{i,2} \cup J_{i,3} \coloneqq (\x_i,\x_*) \cup (\x_*,\x^*) \cup (\x^*,\x^i),
\end{equation}
where the two points \(\x_*\) and \(\x^*\) are given by \(\x_i<\x_*<\x^*<\x^i\) and \(\util(\x_*)=\util(\x^*)= \frac{u_i+u_{i-1}}{2}\).
In fact, \(\x_*\) and \(\x^*\) have \(i\)-dependence, but for simplicity, we omit it without confusion.


\begin{figure}[htp] \centering
	\includegraphics[width=.7\textwidth]{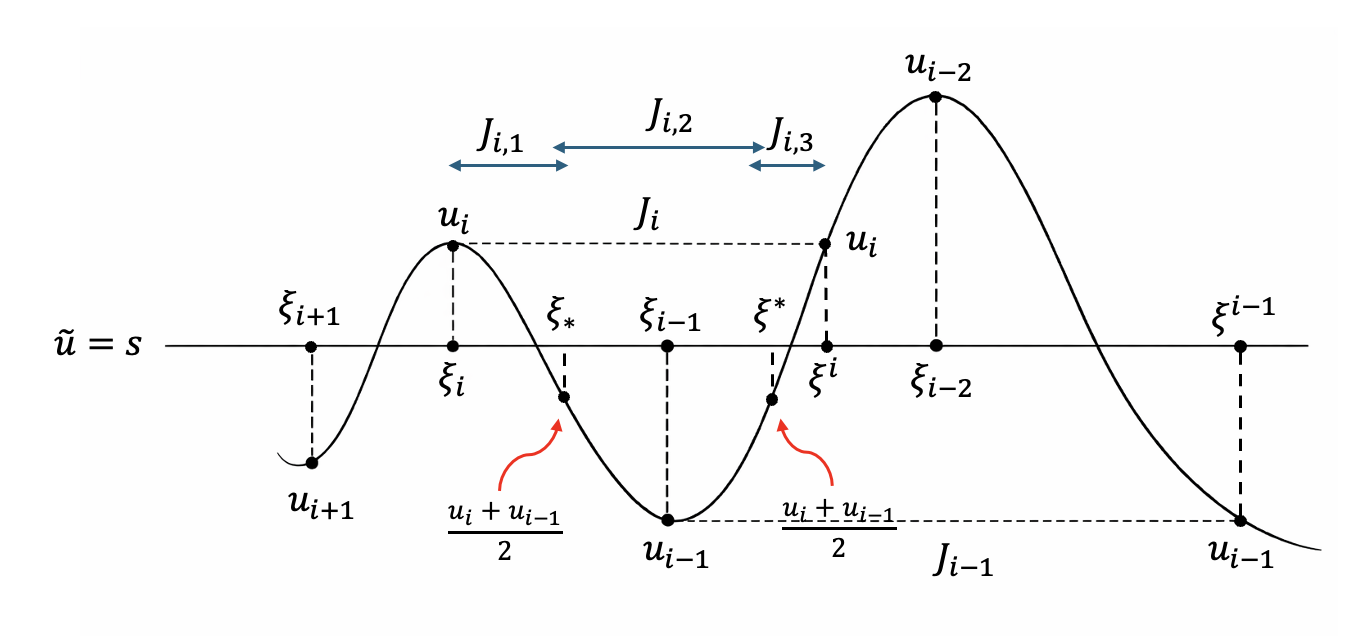}
	\caption{Illustration of \(J_i\), when \(i\) is even. The corresponding picture for odd \(i\) is similar.} \label{fig_l2}
\end{figure}

\begin{proposition} \label{prop:L2}
The viscous-dispersive shock profiles satisfies the following \(L^2\) estimate:
\begin{equation} \label{L2B}
\frac{1}{\abs{u_i-u_{i-1}}}\int_{J_i} (\util-u_i)^2 d\x
\le
\begin{cases}
0.178, & i=1,\\
0.81(\r_*)^{-i}, & i\ge2 \text{ even},\\
0.82(\r_*)^{-i}, & i\ge3 \text{ odd}.
\end{cases}
\end{equation}
Moreover, let \(\x_s\in(\x_1,\x_0)\) be the unique point satisfying \(\util(\x_s)=s\).
Then the following holds:
\begin{equation} \label{L2B-S}
\int_{-\infty}^{\x_s} (\util-s)^2 d\x \le 0.001s.
\end{equation}
\end{proposition}

\begin{pf}
To show the desired \(L^2\) estimates, we first introduce the decomposition \eqref{Ji-decomp}: 
\[
\int_{J_i} (\util-u_i)^2 d\x
=\int_{J_{i,1}} (\util-u_i)^2 d\x
+\int_{J_{i,2}} (\util-u_i)^2 d\x
+\int_{J_{i,3}} (\util-u_i)^2 d\x
\eqqcolon \Jcal_{i,1}+\Jcal_{i,2}+\Jcal_{i,3}.
\]
In what follows, we derive an upper bound for each term, separately for odd and even \(i\).

\vspace{2mm}
\bci{\Jcal_{i,1}} 
If \(i\) is odd, we apply Proposition \ref{prop:key} to obtain
\begin{equation} \label{odd:J1}
\begin{aligned}
\Jcal_{i,1}
&\le (\lbar_i)^{-1} \int_{J_{i,1}}  \frac{(\util-u_i)\util'}{(u_{i-1}-\util)}d\x
=(\lbar_i)^{-1} \int_{u_i}^{\frac{u_i+u_{i-1}}{2}}  \frac{(\util-u_i)}{(u_{i-1}-\util)}d\util\\
&=(\lbar_i)^{-1} \int_{u_i}^{\frac{u_i+u_{i-1}}{2}} \Big(-1 + \frac{(u_{i-1}-u_i)}{(u_{i-1}-\util)}\Big) d\util
=(\lbar_i)^{-1} \Big(\log 2 -\frac{1}{2}\Big) (u_{i-1}-u_i).
\end{aligned}
\end{equation}
Likewise, when \(i\) is even, we again use Proposition \ref{prop:key} to find that 
\begin{equation} \label{even:J1}
\Jcal_{i,1} \le (\lbar_i)^{-1} \Big(\log 2 -\frac{1}{2}\Big) (u_i-u_{i-1}).
\end{equation}

\vspace{2mm}
\bci{\Jcal_{i,2}}
Let \(J_s\) be an (unique) interval contained in \(J_i\) and containing \(J_{i,2}\) such that \(\util=s\) at each endpoint of \(J_s\).
Notice that \(J_s\) has \(i\)-dependence, but for simplicity, we omit it without confusion.
To obtain the \(L^2\) estimate over \(J_{i,2}\), we begin by studying \(J_s\) which is larger than \(J_{i,2}\).

\vspace{2mm}
Thanks to \eqref{t0}, we observe
\begin{equation} \label{L2-stos}
\begin{aligned}
\int_{J_s} (\util')^2 d\x
&=-\int_{J_s} (\util-s)\util'' d\x
=-\frac{1}{\d}\int_{J_s} (\util-s)\Big(\util' -\frac{1}{2}\util^2+\frac{1}{2}s^2\Big)d\x\\
&=-\frac{1}{\d}\int_{J_s} (\util-s)(\util-s)_\x d\x
+\frac{1}{2\d}\int_{J_s} (\util-s)(\util^2-s^2)d\x\\
&=\frac{1}{2\d}\int_{J_s} (\util-s)^2(\util+s)d\x.
\end{aligned}
\end{equation}
It also follows from \eqref{t0} that 
\[
\max_{J_s} \abs{\util'} \le \frac{1}{2}\abs{u_{i-1}^2-s^2}.
\]
This implies that 
\[
2\d \int_{J_s} (\util')^2 d\x
\le \d \abs{u_{i-1}^2-s^2} \int_{J_s} \abs{\util'} d\x
= 2\d (u_{i-1}+s) (u_{i-1}-s)^2.
\]
Thus, summing up, we have
\begin{equation} \label{Js}
\int_{J_s} (\util-s)^2(\util+s)d\x
\le 2\d (u_{i-1}+s) (u_{i-1}-s)^2.
\end{equation}

We now show that the decay rates \(\r_*\) and \(\r^*\) in \eqref{A1/2} yield an upper bound on \(\Jcal_{i,2}\).
If \(i\) is odd, since \(\util+s\ge2s\) on \(J_s\), the above observation \eqref{Js} implies that 
\[
\int_{J_s} (\util-s)^2 d\x
\le \frac{1}{2s} \int_{J_s} (\util-s)^2 (\util+s) d\x
\le \frac{\d}{s} (u_{i-1}+s) (u_{i-1}-s)^2
<\frac{1}{2}(u_{i-1}+s) \frac{(u_{i-1}-s)^2}{s^2}
\]
and thus, the decay rate \(\r_*\) for the increasing intervals shows the following:
\begin{equation} \label{odd:J2}
\begin{aligned}
\int_{J_{i,2}} (\util-u_i)^2 d\x 
&\le \bigg(\frac{\frac{u_i+u_{i-1}}{2}-u_i}{\frac{u_i+u_{i-1}}{2}-s}\bigg)^2
\int_{J_{i,2}} (\util-s)^2 d\x\\
&\le \left(\frac{u_{i-1}-u_i}{(u_{i-1}-s)-(s-u_i)}\right)^2
\int_{J_s} (\util-s)^2 d\x\\
&< \frac{1}{2} \left(\frac{\r_*+1}{\r_*-1}\right)^2 (u_{i-1}+s) \frac{(u_{i-1}-s)^2}{s^2}.
\end{aligned}
\end{equation}
When \(i\) is even, we note that \(\util+s \ge 2s-\r_*^{-1}rs>1.987s\) and that 
\[
\int_{J_s} (\util-s)^2 d\x
< \frac{2}{1.987}\frac{\d}{s} (u_{i-1}+s)(s-u_{i-1})^2
< \frac{1}{1.987} (u_{i-1}+s) \frac{(s-u_{i-1})^2}{s^2}.
\]
Then, similar to the case of odd \(i\), the decay rate \(\r^*\) for the decreasing intervals yields the following:
\begin{equation} \label{even:J2}
\begin{aligned}
\int_{J_{i,2}} (\util-u_i)^2 d\x 
&\le \bigg(\frac{u_i-\frac{u_i+u_{i-1}}{2}}{s-\frac{u_i+u_{i-1}}{2}}\bigg)^2
\int_{J_{i,2}} (\util-s)^2 d\x\\
&\le \left(\frac{u_i-u_{i-1}}{(s-u_{i-1})-(u_i-s)}\right)^2
\int_{J_s} (\util-s)^2 d\x\\
&\le \frac{1}{1.987} \left(\frac{\r^*+1}{\r^*-1}\right)^2 (u_{i-1}+s) \frac{(s-u_{i-1})^2}{s^2}.
\end{aligned}
\end{equation}

\vspace{2mm}
\bci{\Jcal_{i,3}} 
To estimate \(\Jcal_{i,3}\), we use the key inequalities established in Proposition \ref{prop:key}.\\
When \(i\) is odd, using \eqref{key:dec}, we obtain 
\[
\Jcal_{i,3}
\le (\lbar_{i-1})^{-1} \int_{J_{i,3}} \frac{(\util-u_i)^2(-\util')}{(\util-u_{i-2})(u_{i-1}-\util)}d\x
= (\lbar_{i-1})^{-1} \int_{u_i}^{\frac{u_i+u_{i-1}}{2}} \frac{(\util-u_i)^2}{(\util-u_{i-2})(u_{i-1}-\util)}d\util.
\]
Notice that since the decay rate \(\r_*\) shows that \(u_i>\frac{u_{i-1}+u_{i-2}}{2}\), the denominator of the integrand is bounded from below on \([u_i,\frac{u_i+u_{i-1}}{2}]\), and its minimum is attained at \(\util=\frac{u_i+u_{i-1}}{2}\): 
\[
(\util-u_{i-2})(u_{i-1}-\util)
\ge \Big(\frac{1}{2}(u_i+u_{i-1})-u_{i-2}\Big) \Big(\frac{u_{i-1}-u_i}{2}\Big).
\]
Thus, it follows that 
\begin{equation} \label{odd:J3}
\begin{aligned}
\Jcal_{i,3}
&\le \frac{4 (\lbar_{i-1})^{-1}}{(u_i+u_{i-1}-2u_{i-2})(u_{i-1}-u_i)}
\int_{u_i}^{\frac{u_i+u_{i-1}}{2}} (\util-u_i)^2 d\util\\
&\le \frac{(\lbar_{i-1})^{-1}}{6\big((u_i-u_{i-2})+(u_{i-1}-u_{i-2})\big)} (u_{i-1}-u_i)^2.
\end{aligned}
\end{equation}
In the case of even \(i\), the similar argument yields 
\begin{equation} \label{even:J3}
\Jcal_{i,3}
\le \frac{(\lbar_{i-1})^{-1}}{6\big((u_{i-2}-u_i)+(u_{i-2}-u_{i-1})\big)} (u_i-u_{i-1})^2.
\end{equation}

Finally, we sum the above estimates so as to obtain the desired \(L^2\) bounds.
Due to the choice of \(\lbar_i\) in \eqref{key:l}, we need to consider four cases: \(i=1\), \(i=2\), odd \(i\ge3\), and even \(i\ge4\).

\case{1} If \(i=1\), using \eqref{odd:J1}, \eqref{odd:J2} and \eqref{odd:J3}, we find that
\begin{align*}
\frac{1}{u_0-u_1}\int_{J_1} (\util-u_1)^2 d\x
&\le (\lbar_1)^{-1} \Big(\log 2 -\frac{1}{2}\Big)
+\left(\frac{\r_*+1}{\r_*-1}\right)^2 \frac{(u_0+s)(u_0-s)}{2s^2}\\
&\qquad
+\frac{(\lbar_0)^{-1}(u_0-u_1)}{6\big((u_1+s)+(u_0+s)\big)}.
\end{align*}
Recalling \(\lbar_0=0.355\), \(\lbar_1=9.60\), \(\r_*=4.64\), \(u_0\le (1+r)s=1.0601s\) and \(s-u_1\le(\r_*)^{-1}rs<0.013s\), we obtain
\begin{align*}
\frac{1}{u_0-u_1}\int_{J_1} (\util-u_1)^2 d\x
&\le \frac{1}{9.60} \Big(\log 2 -\frac{1}{2}\Big)
+\left(\frac{5.64}{3.64}\right)^2 \frac{(2.0601)(0.0601)}{2}\\
&\qquad
+\frac{0.0731}{6(0.355)(3.987)}
<0.178.
\end{align*}

\case{2} If \(i=2\), using \eqref{even:J1}, \eqref{even:J2} and \eqref{even:J3}, we have
\begin{align*}
\frac{1}{u_2-u_1}\int_{J_2} (\util-u_2)^2 d\x
&\le (\lbar_2)^{-1} \Big(\log 2 -\frac{1}{2}\Big)
+\left(\frac{\r^*+1}{\r^*-1}\right)^2 \frac{(u_1+s)(s-u_1)}{(1.987)s^2}\\
&\qquad
+\frac{(\lbar_1)^{-1}(u_2-u_1)}{6\big((u_0-u_2)+(u_0-u_1)\big)}.
\end{align*}
Recall \(\lbar_1=9.60\), \(\lbar_2=(1.94)(4.64)^2\), \(\r^*=4.77\), \(s-u_1<0.013s\) and \(u_1+s<2s\).
We also note that 
\begin{equation} \label{J3-dnmt}
\frac{(u_0-u_2)+(u_0-u_1)}{u_2-u_1}
=\frac{2(u_0-u_1)}{u_2-u_1}-1 \ge 2\r_*-1 =8.28.
\end{equation}
Thus, it follows that
\begin{align*}
\frac{1}{u_2-u_1}\int_{J_2} (\util-u_2)^2 d\x
&\le \frac{(\r_*)^{-2}}{(1.94)} \Big(\log 2 -\frac{1}{2}\Big)
+\left(\frac{5.77}{3.77}\right)^2 \frac{2(0.013)}{1.987}
+\frac{1}{6(9.60)(8.28)}
<0.81(\r_*)^{-2}.
\end{align*}

\case{3} If \(i\ge3\) is odd, using \eqref{odd:J1}, \eqref{odd:J2} and \eqref{odd:J3}, we obtain
\begin{align*}
\frac{1}{u_{i-1}-u_i}\int_{J_1} (\util-u_i)^2 d\x
&\le (\lbar_i)^{-1} \Big(\log 2 -\frac{1}{2}\Big)
+\left(\frac{\r_*+1}{\r_*-1}\right)^2 \frac{(u_{i-1}+s)(u_{i-1}-s)}{2s^2}\\
&\qquad
+\frac{(\lbar_{i-1})^{-1}(u_{i-1}-u_i)}{6\big((u_i-u_{i-2})+(u_{i-1}-u_{i-2})\big)}.
\end{align*}
Note that \(\lbar_i=(0.51)^{-1}(\r_*)^i\), \(\lbar_{i-1}=(1.94)(\r_*)^{i-1}\), \(u_{i-1}+s\le u_2+s<2.01s\),
\[
u_{i-1}-s \le (\r_*)^{-(i-1)} (u_0-s),
\]
and 
\[
\frac{(u_i-u_{i-2})+(u_{i-1}-u_{i-2})}{(u_{i-1}-u_i)}
=-1 + \frac{2(u_{i-1}-u_{i-2})}{(u_{i-1}-u_i)}
\ge -1+2\r_*=8.28.
\]
Thus, we have 
\begin{align*}
&\frac{1}{u_{i-1}-u_i}\int_{J_1} (\util-u_i)^2 d\x\\
&\le 
\Big[(0.51)\Big(\log 2 -\frac{1}{2}\Big)
+\left(\frac{5.64}{3.64}\right)^2 \frac{(2.01)(0.0601)(4.64)}{2}
+\frac{4.64}{6(8.28)(1.94)}\Big] (\r_*)^{-i}
<0.82(\r_*)^{-i}.
\end{align*}

\case{4} If \(i\ge4\) is even, using \eqref{even:J1}, \eqref{even:J2} and \eqref{even:J3}, we have
\begin{align*}
\frac{1}{u_i-u_{i-1}}\int_{J_i} (\util-u_i)^2 d\x
&\le (\lbar_i)^{-1} \Big(\log 2 -\frac{1}{2}\Big)
+\left(\frac{\r^*+1}{\r^*-1}\right)^2 \frac{(u_{i-1}+s)(s-u_{i-1})}{(1.987)s^2}\\
&\qquad
+\frac{(\lbar_{i-1})^{-1}(u_{i-1}-u_i)}{6\big((u_i-u_{i-2})+(u_{i-1}-u_{i-2})\big)}.
\end{align*}
To get an upper bound of the right-hand side, we recall that 
\[
\lbar_i=(1.94)(\r_*)^i, \qquad
\lbar_{i-1}=(0.51)^{-1}(\r_*)^{i-1}, \qquad
u_{i-1}+s < 2s, \qquad
s-u_{i-1} \le (\r_*)^{-(i-1)} (u_0-s).
\]
In addition, similar to \eqref{J3-dnmt}, we have 
\[
\frac{(u_i-u_{i-2})+(u_{i-1}-u_{i-2})}{(u_{i-1}-u_i)} \ge 8.28.
\]
Therefore, we obtain
\begin{align*}
&\frac{1}{u_i-u_{i-1}}\int_{J_i} (\util-u_i)^2 d\x\\
&\le \Big[(1.94)^{-1}\Big(\log 2 -\frac{1}{2}\Big)
+\left(\frac{5.77}{3.77}\right)^2 \frac{2(0.0601)(4.64)}{(1.987)}
+\frac{(0.51)(4.64)}{6(8.28)}\Big] (\r_*)^{-i}
<0.81(\r_*)^{-i}.
\end{align*}
This completes the proof of \eqref{L2B}.

\vspace{2mm}
It remains to show \eqref{L2B-S}.
We first recall \eqref{t1} and \eqref{t2}.
As in \eqref{L2-stos}, we find that 
\[
\int_{-\infty}^{\x_s} (\util-s)^2 (\util+s)d\x
=\int_{-\infty}^{\x_s} (\util')^2 d\x
\le 2\d\big(E(\x_0)-E(-\infty)\big)
=2\d \Big[\frac{2}{3}s^3-\frac{1}{6}(u_0+s)^2(2s-u_0)\Big].
\]
Note that \(l(x)=(x+s)^2(2s-x)\) is decreasing on \([s,1.0601s]\) and that  \(\util+s >1.987s\) on \((-\infty,\x_s)\).
Then, we obtain 
\begin{align*}
\int_{-\infty}^{\x_s} (\util-s)^2 d\x
&\le \frac{1}{1.987s} \int_{-\infty}^{\x_s} (\util-s)^2 (\util+s)d\x\\
&\le \frac{2\d}{1.987s} \Big[\frac{2}{3}s^3-\frac{1}{6}(u_0+s)^2(2s-u_0)\Big]
<0.001s
\end{align*}
This completes the proof of \eqref{L2B-S} and Proposition \ref{prop:L2}.
\end{pf}

\subsection{Proof of Theorem \ref{thm:main}} \label{sec:pf-main}
We are ready to prove the contraction property via an inductive argument.
First, we outline the main idea of the proof.
In view of Lemma \ref{lem:RHS} (with \eqref{L2-stable}-\eqref{X-def}), we aim to show the following:
\begin{equation} \label{goal}
-\frac{M}{2s} \Big(\int_\RR w \util'd\x\Big)^2
-\frac{1}{2}\int_\RR w^2 \util' d\x
- \frac{9}{10}\int_\RR (w_\x)^2 d\x
\le 0,
\end{equation}
where \(w\) denotes the perturbation, i.e., \(w \coloneqq u^X-\util\).
We then introduce the following notation:
\begin{equation} \label{badgood}
\Hcal_i \coloneqq -\frac{1}{2}\int_{\x_i}^{\x_{i-1}} w^2 \util' d\x, \qquad
\Dcal \coloneqq - \int_\RR (w_\x)^2 d\x, \qquad
\Dcal[a,b] \coloneqq - \int_a^b (w_\x)^2 d\x.
\end{equation}
Notice that \(\Hcal_i\) is non-positive for odd \(i\) and non-negative for even \(i\) and that \(\Dcal=\Dcal[-\infty,\infty]\).
On the left-hand side of \eqref{goal}, the terms \(\Hcal_i\) with even \(i\) are the only non-negative contributions; thus our goal is to cancel them with the other terms.

The basic strategy is to use the Poincar\'e-type inequality (Lemma \ref{KV_lemma}).
To this end, we introduce the change of variable \eqref{ztu} on each monotonic interval \((\x_i,\x_{i-1})\) as follows:
\begin{equation} \label{ztu0}
z(\x) \coloneqq \util(\x), \qquad
dz = \util' d\x.
\end{equation}
This requires Proposition \ref{prop:key} for the Jacobian estimates and allows us to rewrite the diffusion term \(\Dcal[\x_i,\x_{i-1}]\) in a form suitable for applying Lemma \ref{KV_lemma}.
However, a significant challenge remains in making use of Lemma \ref{KV_lemma}: we need a squared average term of the following form: 
\[
-\frac{1}{\abs{u_i-u_{i-1}}}\Big(\int_{\x_i}^{\x_{i-1}}w\util'd\x\Big)^2
=-\frac{1}{\abs{u_i-u_{i-1}}}\Big(\int_{u_i}^{u_{i-1}}w dz\Big)^2.
\]
This is the main difficulty in the proof, which is overcome by the inductive argument developed in the present paper.
In the course of the argument, we rely on \eqref{transfer}:
\begin{equation} \label{transfer0}
\Big(\int_{J_i} w\util' d\x\Big)^2
=\Big(\int_{J_i} w_\x (\util-u_i) d\x\Big)^2
\le \Big(\int_{J_i} (\util-u_i)^2 d\x\Big)
\Dcal[\x_i,\x^i],
\end{equation}
where \(J_i=(\x_i,\x^i)\).
This, together with the \(L^2\) estimates in Proposition \ref{prop:L2}, quantifies how much of the diffusion term \(\Dcal[\x_i,\x^i]\) is required to produce the squared average term on \(J_i\).

We also recall the structure of the inductive argument described in the final part of Section \ref{sec:main3}.
The base case of the induction is addressed in \textit{Steps 0} and \textit{1}, where the first term on the left-hand side of \eqref{goal} plays a crucial role.
In \textit{Step 2}, we carry out the inductive step.

\step{0}
First of all, using the same argument as in \eqref{transfer0}, we observe 
\[
-\Big(\int_{-\infty}^{\x_s}(\util-s)^2d\x\Big) \Big(\int_{-\infty}^{\x_s} (w_\x)^2 d\x\Big)
\le -\Big(\int_{-\infty}^{\x_s} w\util' d\x\Big)^2.
\]
Then, using \eqref{L2B-S} in Proposition \ref{prop:L2}, we find that
\begin{equation} \label{step0-1}
-20 C_0 (0.001) \Dcal[-\infty,\x_s]
\le -\frac{40C_0}{2s} \Big(\int_{-\infty}^{\x_s} w\util' d\x\Big)^2,
\end{equation}
where \(C_0\) is a positive constant to be determined at the end of \textit{Step 0}.
This together with the shift good term (the first term on the left-hand side of \eqref{goal}) yields that 
\[
-\frac{M}{2s} \Big(\int_\RR w\util' d\x\Big)^2
-\frac{40C_0}{2s} \Big(\int_{-\infty}^{\x_s} w\util' d\x\Big)^2
\le -\frac{1}{2s}\frac{40 C_0 M}{40C_0 + M} \Big(\int_{\x_s}^\infty w\util' d\x\Big)^2.
\]
Thus, for large enough \(M\)---in particular, for any \(M\ge\frac{40}{39}C_0\)---we have
\begin{equation} \label{step0-2}
-\frac{M}{2s} \Big(\int_\RR w\util' d\x\Big)^2
-\frac{40C_0}{2s} \Big(\int_{-\infty}^{\x_s} w\util' d\x\Big)^2
\le -\frac{C_0}{2s}\Big(\int_{\x_s}^\infty w\util' d\x\Big)^2.
\end{equation}
We also exploit a portion of \(\Hcal_1\) as follows:
\begin{equation} \label{step0-3}
\frac{1}{15}\Hcal_1
=-\frac{1}{30}\int_{\x_1}^{\x_0} w^2\util' d\x
\le -\frac{1}{30} \int_{\x_s}^{\x_0} w^2 \util' d\x
\le -\frac{1}{30(u_0-s)} \Big(\int_{\x_s}^{\x_0} w\util' d\x\Big)^2.
\end{equation}
Using \eqref{step0-2} and \eqref{step0-3}, we obtain the squared average term on the rightmost decreasing interval:
\begin{equation} \label{step0-4}
-\frac{C_0}{2s}\Big(\int_{\x_s}^\infty w\util' d\x\Big)^2
-\frac{1}{30(u_0-s)} \Big(\int_{\x_s}^{\x_0} w\util' d\x\Big)^2
\le -\frac{C_0}{30C_0(u_0-s)+2s} \Big(\int_{\x_0}^\infty w\util' d\x\Big)^2.
\end{equation}
We choose the constant \(C_0\) so that it satisfies
\[
\frac{C_0}{30C_0(u_0-s)+2s} \ge \frac{7}{12(u_0+s)}, \quad \text{or equivalently,} \quad
C_0 \ge \frac{14s}{222s-198u_0}.
\]
On the other hand, since \(u_0\le1.0601s\) by \eqref{u0-exp}, it suffices to choose \(C_0\ge\frac{13}{10}\).
We then fix \(C_0=\frac{13}{10}\).
Based on the change of variables \eqref{ztu0}, we use Lemma \ref{KV_lemma}, with Proposition \ref{prop:key} and \eqref{step0-4}, to get
\begin{equation} \label{step0-5}
\begin{aligned}
&-\frac{7}{12(u_0+s)}\Big(\int_{\x_0}^\infty w\util' d\x\Big)^2
-\frac{7}{24}(\lbar_0)^{-1} \Dcal[\x_0,\infty]\\
&\qquad
\le -\frac{7}{12(u_0+s)}\Big(\int_{-s}^{u_0} w dz\Big)^2
-\frac{7}{24}\int_{-s}^{u_0} (u_0-z)(z+s) (w_z)^2 dz\\
&\qquad
\le -\frac{7}{12} \int_{-s}^{u_0} w^2 dz
=\frac{7}{12} \int_{\x_0}^\infty w^2 \util' d\x.
\end{aligned}
\end{equation}
Note that \(\lbar_0=0.355\) and \(\frac{7}{24}(\lbar_0)^{-1} < 0.83\).
This cancels with \(\Hcal_0\) as follows:
\begin{equation} \label{step0-6}
\frac{7}{12} \int_{\x_0}^\infty w^2 \util' d\x
+\Hcal_0
=\frac{7}{12} \int_{\x_0}^\infty w^2 \util' d\x
-\frac{1}{2}\int_{\x_0}^\infty w^2 \util' d\x
=\frac{1}{12} \int_{\x_0}^\infty w^2 \util' d\x.
\end{equation}
Thus, summing up \eqref{step0-1}--\eqref{step0-6}, we summarize \textit{Step 0} as follows:
\begin{equation} \label{step0}
\begin{aligned}
-(0.026)\Dcal[-\infty,\x_s]
-\frac{M}{2s}\Big(\int_\RR w\util' d\x\Big)^2
+\frac{1}{15}\Hcal_1
-(0.83)\Dcal[\x_0,\infty]
+\Hcal_0
\le \frac{1}{12} \int_{\x_0}^\infty w^2 \util' d\x.
\end{aligned}
\end{equation}
We note that \(\Hcal_1\) is negative, whereas \(\Hcal_0\) is positive.
We also recall that \(C_0=\frac{13}{10}\) and \(M\ge\frac{4}{3}\).

\step{1} In this step, we verify the base case of the induction---we obtain the following good term: 
\begin{equation} \label{base}
-\Big(\frac{1}{2}+a_1\Big)\int_{\x_1}^{\x_0} w^2 \util' d\x,
\end{equation}
for some constant \(a_1>0\).
To this end, using the remainder in \eqref{step0}, we first observe
\begin{equation} \label{step1-1}
\frac{1}{12} \int_{\x_0}^\infty w^2 \util' d\x
\le \frac{1}{12} \int_{\x_0}^{\x^1} w^2 \util' d\x
\le -\frac{1}{12(u_0-u_1)} \Big(\int_{\x_0}^{\x^1} w \util' d\x\Big)^2.
\end{equation}
On the other hand, using \eqref{transfer0} with \eqref{L2B} in Proposition \ref{prop:L2}, we find that
\begin{equation} \label{step1-2}
-0.178C_1 \int_{J_1} (w_\x)^2 d\x
=-0.178C_1 \Dcal[\x_1,\x^1]
\le -\frac{C_1}{u_0-u_1} \Big(\int_{J_1} w\util'd\x\Big)^2,
\end{equation}
where \(C_1>0\) is a constant to be determined later.
The two estimates above show that we have
\[
-\frac{1}{12(u_0-u_1)} \Big(\int_{\x_0}^{\x^1} w \util' d\x\Big)^2
-\frac{C_1}{u_0-u_1} \Big(\int_{J_1} w\util' d\x\Big)^2
\le - \frac{C_1}{(12C_1+1)(u_0-u_1)} \Big(\int_{\x_1}^{\x_0} w \util' d\x\Big)^2.
\]
We choose \(C_1\) to satisfy 
\[
\frac{C_1}{12C_1+1} = \frac{1}{15}, \quad \text{or equivalently,} \quad
C_1=\frac{1}{3}.
\]
Then, we rewrite the above inequality into the following:
\begin{equation} \label{step1-3}
-\frac{1}{12(u_0-u_1)} \Big(\int_{\x_0}^{\x^1} w \util' d\x\Big)^2
-\frac{1}{3(u_0-u_1)} \Big(\int_{J_1} w\util' d\x\Big)^2
\le - \frac{1}{15(u_0-u_1)} \Big(\int_{\x_1}^{\x_0} w \util' d\x\Big)^2.
\end{equation}
It follows from Lemma \ref{KV_lemma} and Proposition \ref{prop:key} with \eqref{ztu0} that
\begin{equation} \label{step1-4}
\begin{aligned}
&- \frac{1}{15(u_0-u_1)} \Big(\int_{\x_1}^{\x_0} w \util' d\x\Big)^2
-\frac{1}{30}(\lbar_1)^{-1} \Dcal[\x_1,\x_0]\\
&\qquad
\le -\frac{1}{15(u_0-u_1)}\Big(\int_{u_1}^{u_0} w dz\Big)^2
-\frac{1}{30}\int_{u_1}^{u_0} (u_0-z)(z-u_1) (w_z)^2 dz\\
&\qquad
\le -\frac{1}{15} \int_{u_1}^{u_0} w^2 dz
=-\frac{1}{15} \int_{\x_1}^{\x_0} w^2 \util' d\x.
\end{aligned}
\end{equation}
We recall \(\lbar_1=9.60\).
Thus, combining \eqref{step1-1}-\eqref{step1-4}, we obtain the following summary of \textit{Step 1}:
\[
\frac{1}{12} \int_{\x_0}^\infty w^2 \util' d\x
-(0.06)\Dcal[\x_1,\x^1]
-(0.01)\Dcal[\x_1,\x_0]
\le -\frac{1}{15} \int_{\x_1}^{\x_0} w^2 \util' d\x.
\]
This together with \eqref{step0} and the \(\Hcal_1\) term implies that
\begin{equation} \label{step0and1}
\begin{aligned}
&-\frac{M}{2s}\Big(\int_\RR w\util' d\x\Big)^2
+\Hcal_0
-(0.83)\Dcal[\x_0,\infty]
-(0.026)\Dcal[-\infty,\x_s]\\
&\qquad
-(0.06)\Dcal[\x_1,\x^1]
-(0.01)\Dcal[\x_1,\x_0]
+\Hcal_1
\le -\Big(\frac{1}{2}+\frac{1}{30}\Big)\int_{\x_1}^{\x_0} w^2\util'd\x.
\end{aligned}
\end{equation}
This verifies the base case \eqref{base} with \(a_1=\frac{1}{30}\).

\step{2}
To prepare for the inductive step, we introduce two sequences as follows: for each \(n\in\NN\),
\begin{equation} \label{seq}
a_n\coloneqq\Big(\frac{1}{15}\Big)2^{-n}, \qquad
C_{2n}\coloneqq a_{2n-1}+\frac{3}{2}+\frac{1}{2a_{2n-1}}, \qquad
C_{2n+1}\coloneqq a_{2n}.
\end{equation}
We now assume the induction hypothesis: for some odd \(i\ge1\), the following good term is available:
\begin{equation} \label{ind-hypo0}
-\Big(\frac{1}{2}+a_i\Big) \int_{\x_i}^{\x_{i-1}} w^2 \util' d\x.
\end{equation}
We show that the induction hypothesis holds for \(i+2\).
To this end, we first observe
\begin{equation} \label{step2-1}
-\Big(\frac{1}{2}+a_i\Big) \int_{\x_i}^{\x_{i-1}} w^2 \util' d\x
\le -\Big(\frac{1}{2}+a_i\Big) \int_{\x_i}^{\x^{i+1}} w^2 \util' d\x
\le -\frac{\big(\frac{1}{2}+a_i\big)}{u_{i+1}-u_i} \Big(\int_{\x_i}^{\x^{i+1}} w \util' d\x\Big)^2.
\end{equation}
Meanwhile, it follows from \eqref{transfer0} with \eqref{L2B} in Proposition \ref{prop:L2} that
\begin{equation} \label{step2-2}
\begin{aligned}
-(0.81)(\r_*)^{-(i+1)}C_{i+1}\int_{J_{i+1}} (w_\x)^2 d\x
&=-(0.81)(\r_*)^{-(i+1)}C_{i+1}\Dcal[\x_{i+1},\x^{i+1}]\\
&\le -\frac{C_{i+1}}{u_{i+1}-u_i}\Big(\int_{J_{i+1}} w\util'd\x\Big)^2.
\end{aligned}
\end{equation}
Then, using \eqref{step2-1} and \eqref{step2-2}, we obtain
\begin{align*}
&-\frac{\big(\frac{1}{2}+a_i\big)}{u_{i+1}-u_i} \int_{\x_i}^{\x^{i+1}} w^2 \util' d\x
-\frac{C_{i+1}}{u_{i+1}-u_i}\Big(\int_{J_{i+1}} w\util'd\x\Big)^2\\
&\qquad
\le -\frac{1}{u_{i+1}-u_i}\frac{\big(\frac{1}{2}+a_i\big)C_{i+1}}{\big(\frac{1}{2}+a_i\big)+C_{i+1}}
\Big(\int_{\x_{i+1}}^{\x_i} w \util' d\x\Big)^2.
\end{align*}
Thanks to \eqref{seq}, this is equivalent to the following:
\begin{equation} \label{step2-3}
-\frac{\big(\frac{1}{2}+a_i\big)}{u_{i+1}-u_i} \int_{\x_i}^{\x^{i+1}} w^2 \util' d\x
-\frac{C_{i+1}}{u_{i+1}-u_i}\Big(\int_{J_{i+1}} w\util'd\x\Big)^2
\le -\frac{\big(\frac{1}{2}+a_{i+1}\big)}{u_{i+1}-u_i} \Big(\int_{\x_{i+1}}^{\x_i} w \util' d\x\Big)^2.
\end{equation}
We apply Lemma \ref{KV_lemma} together with Proposition \ref{prop:key} and \eqref{ztu0} to obtain
\begin{equation} \label{step2-4}
\begin{aligned}
&-\frac{\big(\frac{1}{2}+a_{i+1}\big)}{u_{i+1}-u_i} \Big(\int_{\x_{i+1}}^{\x_i} w \util' d\x\Big)^2
-\frac{1}{2}\Big(\frac{1}{2}+a_{i+1}\Big)(\lbar_{i+1})^{-1} \Dcal[\x_{i+1},\x_i]\\
&\qquad
\le -\frac{\big(\frac{1}{2}+a_{i+1}\big)}{u_{i+1}-u_i} \Big(\int_{u_i}^{u_{i+1}} w dz\Big)^2
-\frac{1}{2}\Big(\frac{1}{2}+a_{i+1}\Big) \int_{u_i}^{u_{i+1}} (u_{i+1}-z)(z-u_i) (w_z)^2 dz\\
&\qquad
\le -\Big(\frac{1}{2}+a_{i+1}\Big) \int_{u_i}^{u_{i+1}} w^2 dz
=\Big(\frac{1}{2}+a_{i+1}\Big) \int_{\x_{i+1}}^{\x_i} w^2 \util' d\x,
\end{aligned}
\end{equation}
where \(\lbar_{i+1}=(1.94)(\r_*)^{i+1}\).
This cancels with \(\Hcal_{i+1}\) as follows:
\begin{equation} \label{step2-5}
\Big(\frac{1}{2}+a_{i+1}\Big) \int_{\x_{i+1}}^{\x_i} w^2 \util' d\x
+\Hcal_{i+1}
=a_{i+1}\int_{\x_{i+1}}^{\x_i} w^2 \util' d\x (\le 0).
\end{equation}
This yields the following good term: 
\begin{equation} \label{step2-6}
a_{i+1}\int_{\x_{i+1}}^{\x_i} w^2 \util' d\x
\le a_{i+1}\int_{\x_{i+1}}^{\x^{i+2}} w^2 \util' d\x
\le -\frac{a_{i+1}}{u_{i+1}-u_{i+2}} \Big(\int_{\x_{i+1}}^{\x^{i+2}} w\util' d\x\Big)^2.
\end{equation}
Moreover, using \eqref{transfer0} with \eqref{L2B} in Proposition \ref{prop:L2}, we obtain
\begin{equation} \label{step2-7}
\begin{aligned}
-(0.82)(\r_*)^{-(i+2)}C_{i+2}\int_{J_{i+2}} (w_\x)^2 d\x
&=-(0.82)(\r_*)^{-(i+2)}C_{i+2}\Dcal[\x_{i+2},\x^{i+2}]\\
&\le -\frac{C_{i+2}}{u_{i+1}-u_{i+2}}\Big(\int_{J_{i+2}} w\util'd\x\Big)^2.
\end{aligned}
\end{equation}
Combining \eqref{step2-6} and \eqref{step2-7}, we find that
\begin{align*}
&-\frac{a_{i+1}}{u_{i+1}-u_{i+2}} \Big(\int_{\x_{i+1}}^{\x^{i+2}} w\util' d\x\Big)^2
-\frac{C_{i+2}}{u_{i+1}-u_{i+2}}\Big(\int_{J_{i+2}} w\util'd\x\Big)^2\\
&\qquad
\le -\frac{1}{u_{i+1}-u_{i+2}}\frac{a_{i+1}C_{i+2}}{a_{i+1}+C_{i+2}}
\Big(\int_{\x_{i+2}}^{\x_{i+1}}w\util' d\x\Big)^2.
\end{align*}
Under the choices of \(a_i\) and \(C_i\) in \eqref{seq}, we rewrite this into the following:
\begin{equation} \label{step2-8}
\begin{aligned}
&-\frac{a_{i+1}}{u_{i+1}-u_{i+2}} \Big(\int_{\x_{i+1}}^{\x^{i+2}} w\util' d\x\Big)^2
-\frac{C_{i+2}}{u_{i+1}-u_{i+2}}\Big(\int_{J_{i+2}} w\util'd\x\Big)^2\\
&\qquad
\le -\frac{a_{i+2}}{u_{i+1}-u_{i+2}} \Big(\int_{\x_{i+2}}^{\x_{i+1}}w\util' d\x\Big)^2.
\end{aligned}
\end{equation}
Then, using Lemma \ref{KV_lemma} with Proposition \ref{prop:key} and \eqref{ztu0}, we obtain
\begin{equation} \label{step2-9}
\begin{aligned}
&-\frac{a_{i+2}}{u_{i+1}-u_{i+2}} \Big(\int_{\x_{i+2}}^{\x_{i+1}}w\util' d\x\Big)^2
-\frac{a_{i+2}}{2}(\lbar_{i+2})^{-1} \Dcal[\x_{i+2},\x_{i+1}]\\
&\qquad
\le -\frac{a_{i+2}}{u_{i+1}-u_{i+2}} \Big(\int_{u_{i+2}}^{u_{i+1}}w dz\Big)^2
-\frac{a_{i+2}}{2} \int_{u_{i+2}}^{u_{i+1}} (u_{i+1}-z)(z-u_{i+2}) (w_z)^2 dz\\
&\qquad
\le -a_{i+2} \int_{u_{i+2}}^{u_{i+1}} w^2 dz
=-a_{i+2} \int_{\x_{i+2}}^{\x_{i+1}} w^2 \util' d\x.
\end{aligned}
\end{equation}
This with \(\Hcal_{i+2}\) recovers the induction hypothesis \eqref{ind-hypo0} for \(i+2\)---gathering \eqref{step2-1}-\eqref{step2-9}, we have
\begin{equation} \label{step2}
\begin{aligned}
&-\Big(\frac{1}{2}+a_i\Big) \int_{\x_i}^{\x_{i-1}} w^2 \util' d\x
+\Hcal_{i+1}+\Hcal_{i+2}\\
&-(0.81)(\r_*)^{-(i+1)}C_{i+1}\Dcal[\x_{i+1},\x^{i+1}]
-(0.82)(\r_*)^{-(i+2)}C_{i+2}\Dcal[\x_{i+2},\x^{i+2}]\\
&-\frac{1}{2}\Big(\frac{1}{2}+a_{i+1}\Big)(1.94)^{-1}(\r_*)^{-(i+1)} \Dcal[\x_{i+1},\x_i]
-\frac{a_{i+2}}{2}(0.51)(\r_*)^{-(i+2)} \Dcal[\x_{i+2},\x_{i+1}]\\
&\qquad\le -\Big(\frac{1}{2}+a_{i+2}\Big) \int_{\x_{i+2}}^{\x_{i+1}} w^2 \util' d\x.
\end{aligned}
\end{equation}
We now examine how much of the diffusion term is used.
We recall \eqref{seq}: 
\[
a_i=\frac{1}{15}2^{-i}, \quad
a_{i+1}=\frac{1}{15}2^{-(i+1)}, \quad
a_{i+2}=\frac{1}{15}2^{-(i+2)}, \quad
C_{i+1}=a_i+\frac{3}{2}+\frac{1}{2a_i}, \quad
C_{i+2}=a_{i+1}.
\]
Note that for any odd \(i\in\NN\), \(C_{i+3}\le4C_{i+1}\), and thus we have
\[
(0.81)(\r_*)^{-(i+1)}C_{i+1}
\le (0.81)(\r_*)^{-2}C_2
< 0.63.
\]
Moreover, for any odd \(i\in\NN\), the other coefficients of the diffusion term are all less than \(1\):
\[
(0.82)(\r_*)^{-(i+2)}C_{i+2}, \quad
\frac{1}{2}\Big(\frac{1}{2}+a_{i+1}\Big)(1.94)^{-1}(\r_*)^{-(i+1)}, \quad
\frac{a_{i+2}}{2}(0.51)(\r_*)^{-(i+2)} < 0.01.
\]
Thus, from \eqref{step2}, we find that for each odd \(i\in\NN\),
\begin{equation} \label{step2-0}
\begin{aligned}
&-\Big(\frac{1}{2}+a_i\Big) \int_{\x_i}^{\x_{i-1}} w^2 \util' d\x
+\Hcal_{i+1}+\Hcal_{i+2}
-(0.63)\Dcal[\x_{i+1},\x^{i+1}]\\
&-(0.01)\Dcal[\x_{i+2},\x^{i+2}]
-(0.01)\Dcal[\x_{i+1},\x_i]
-(0.01)\Dcal[\x_{i+2},\x_{i+1}]\\
&\qquad\le -\Big(\frac{1}{2}+a_{i+2}\Big) \int_{\x_{i+2}}^{\x_{i+1}} w^2 \util' d\x.
\end{aligned}
\end{equation}

Combining the base case \eqref{step0and1} and the inductive step \eqref{step2-0}, we observe that for any interval on \((a,b)\), less than \(\frac{9}{10}\) of the diffusion term, i.e., \(\frac{9}{10}\Dcal[a,b]\) is used.
Thus, we conclude that
\begin{align*}
\frac{d}{dt}\frac{1}{2}\int_\RR w^2 d\x
&=-\frac{M}{s} \Big(\int_\RR w \util'd\x\Big)^2
-\frac{1}{2}\int_\RR w^2 \util' d\x
- \int_\RR (w_\x)^2 d\x\\
&\le -\frac{M}{2s} \Big(\int_\RR w \util'd\x\Big)^2
-\frac{1}{10}\int_\RR (w_\x)^2 d\x \le 0,
\end{align*}
which yields \eqref{L2-stable}.
In addition, using \eqref{L2-stable}, we proceed along the lines of the proof of \cite[Theorem 2.3]{Hur} to derive the time-asymptotic stability result \eqref{timeasy}.
More precisely, the implication from \eqref{L2-stable} to \eqref{timeasy} is established in \cite{Hur}.
Since this implication does not require any spectral assumptions and the same estimates used in \cite{Hur} are available in the present setting, we omit the proof for brevity.
Lastly, using \eqref{timeasy}, we observe that 
\[
|\dot{X}(t)| \le \norm{u(t,\cdot)-\util(\cdot-X(t))}_{L^\infty(\RR)} \int_\RR |\util'(\x)|d\x
\to 0 \quad \text{as } t \to \infty,
\]
since the total variation of the oscillatory viscous-dispersive shock is bounded because of the decay rates \eqref{inc-decay} and \eqref{dec-decay}.
This completes the proof of Theorem \ref{thm:main}. \qed

\section{Proof of Theorem \ref{thm:limit}: Zero Viscosity-Dispersion Limits} \label{sec:limit}
\setcounter{equation}{0}
Finally, we prove Theorem \ref{thm:limit}.
The key idea is to exploit the contraction property in Theorem \ref{thm:main} in order to derive a uniform estimate.
To this end, we first show that the contraction property is invariant under the scaling \((\n t,\n x)\), from which the first part of Theorem \ref{thm:limit} follows.
Then, we establish the second part of Theorem \ref{thm:limit}.

\subsection{Proof of Theorem \ref{thm:limit}: Part 1} \label{sec:scaling}
Throughout this section, we fix two parameters \(\e\) and \(\d\) satisfying \eqref{cond1}, i.e.,
\[
\frac{1}{4} < \frac{\d(u_--u_+)}{2\e^2} < \frac{1}{2}.
\]
Let \(\util\) be the viscous-dispersive shock profile associated with the fixed pair of parameters \((\e,\d)\), and let \(u_0\) be an initial datum such that \(u_0-\util \in H^1(\RR)\).
Then, we recall the original initial value problem \eqref{burgers}:
\begin{equation} \label{burgers0}
u_t + \Big(\frac{u^2}{2}\Big)_x = \e u_{xx}-\d u_{xxx}, \qquad u(0,x)=u_0(x) \in H^1(\RR).
\end{equation}
Then, for the solution \(u\in\Xcal_T\), Theorem \ref{thm:main} shows that
\begin{equation} \label{re-cont}
\begin{aligned}
&\int_\RR \big(u(t,x)-\util(x-\s t-X(t))\big)^2 dx
+\frac{(u_--u_+)}{M}\int_0^T |\dot{X}(t)|^2 dt\\
&\qquad +\frac{\e}{10}\int_0^T \int_\RR \big((u(t,x)-\util(x-\s t-X(t)))_x\big)^2 dx dt
\le \int_\RR \big(u_0(x)-\util(x)\big)^2 dx.
\end{aligned}
\end{equation}
We recall the scaled equation: 
\begin{equation} \label{scaled-eq0}
(u^\n)_t + \Big(\frac{(u^\n)^2}{2}\Big)_x = \n \e (u^\n)_{xx}- \n^2 \d (u^\n)_{xxx}, \qquad
u^\n(0,x)=u_0\Big(\frac{x}{\n}\Big).
\end{equation}
In this subsection, we first prove the scaling invariance of the contraction property---namely, we aim to show that
\begin{equation} \label{nu-cont}
\begin{aligned}
&\int_\RR \big((u^\n)(t,x)-(\util^\n)(x-\s t-X_\n(t))\big)^2 dx
+\frac{(u_--u_+)}{M}\int_0^T |\dot{X}_\n(t)|^2 dt\\
&\qquad +\frac{\e}{10}\n \int_0^T \int_\RR \big(((u^\n)(t,x)-(\util^\n)(x-\s t-X_\n(t)))_x\big)^2 dx dt
\le \int_\RR \Big(u_0\Big(\frac{x}{\n}\Big)-\util^\n(x)\Big)^2 dx,
\end{aligned}
\end{equation}
where \(\util^\n(x)\coloneqq\util(x/\n)\) and the shift function \(X_\n\) is defined by \(\dot{X}_\n(0)=0\) and
\begin{equation} \label{nu-shift}
\dot{X}_\n(t)
=-\frac{2M}{(u_--u_+)}\int_\RR \big(u^\n(t,x)-\util^\n(x-\s t-X_\n(t))\big) (\util^\n)'(x-\s t-X_\n(t))dx.
\end{equation}
This can be shown as follows.
We first set \(u(t,x)\coloneqq u^\n(\n t,\n x)\), which satisfies the equation \eqref{burgers0} subject to the initial datum \(u_0(x)\).
Then, we apply Theorem \ref{thm:main} to find that \eqref{re-cont} holds with the shift function \(X(t)\coloneqq X_{\n=1}(t)\) which is given by \(X(0)=0\) and 
\[
\dot{X}(t)
=-\frac{2M}{(u_--u_+)}\int_\RR \big(u(t,x)-\util(x-\s t-X(t))\big) \util'(x-\s t-X(t))dx.
\]
To proceed, we examine the relationship between the shift functions \(X_\n\) and \(X\).
Indeed, \(X_\n\) can be defined as \(X_\n(t)=\n X\big(\frac{t}{\n}\big)\), which can be justified in the following way:
\begin{equation} \label{sfts}
\begin{aligned}
\dot{X}\Big(\frac{t}{\n}\Big)
&=-\frac{2M}{(u_--u_+)}\int_\RR
\Big(u\Big(\frac{t}{\n},\frac{x}{\n}\Big)-\util\Big(\frac{x}{\n}-\frac{\s t}{\n}-X\big(\frac{t}{\n}\big)\Big)\Big)
\util'\Big(\frac{x}{\n}-\frac{\s t}{\n}-X\big(\frac{t}{\n}\big)\Big) \frac{dx}{\n}\\
&=-\frac{2M}{(u_--u_+)}\int_\RR
\Big(u^\n(t,x)-\util^\n\Big(x-\s t-\n X\big(\frac{t}{\n}\big)\Big)\Big)
(\util^\n)'\Big(x-\s t-\n X\big(\frac{t}{\n}\big)\Big) dx\\
&=-\frac{2M}{(u_--u_+)}\int_\RR \big(u^\n(t,x)-\util^\n(x-\s t-X_\n(t))\big) (\util^\n)'(x-\s t-X_\n(t)) dx
=\dot{X}_\n(t).
\end{aligned}
\end{equation}
Then, we obtain that 
\begin{align*}
\int_\RR \big(u^\n(t,x)-\util^\n(x-\s t-X_\n(t))\big)^2 dx
&=\n \int_\RR \Big(u^\n(\n t,\n x)-\util^\n(\n(x-\s t)-X_\n(\n t))\Big)^2 dx\\
&=\n \int_\RR \Big(u(t,x)-\util\big(x-\s t-\frac{1}{\n}X_\n(\n t)\big)\Big)^2 dx\\
&=\n \int_\RR \Big(u(t,x)-\util\big(x-\s t-X(t)\big)\Big)^2 dx.
\end{align*}
Moreover, using \eqref{sfts}, we find that
\begin{align*}
\int_0^T |\dot{X}_\n(t)|^2 dt
=\int_0^T \Big|\dot{X}\big(\frac{t}{\n}\big)\Big|^2 dt
=\n \int_0^\frac{T}{\n} |\dot{X}(t)|^2 dt
\end{align*}
and similarly, it also holds that 
\begin{multline*}
\int_0^T \int_\RR \Big(\big((u^\n)(t,x)-(\util^\n)(x-\s t-X_\n(t))\big)_x\Big)^2 dx dt\\
=\int_0^\frac{T}{\n} \int_\RR \Big(\big(u(t,x)-\util(x-\s t-X(t))\big)_x\Big)^2 dx dt.
\end{multline*}
Thus, since \eqref{re-cont} holds for any \(T>0\), we obtain
\begin{align*}
&\int_\RR \big((u^\n)(t,x)-(\util^\n)(x-\s t-X_\n(t))\big)^2 dx
+\frac{(u_--u_+)}{M}\int_0^T |\dot{X}_\n(t)|^2 dt\\
&\qquad +\frac{\e}{10}\n \int_0^T \int_\RR \big(((u^\n)(t,x)-(\util^\n)(x-\s t-X_\n(t)))_x\big)^2 dx dt\\
&=\n \int_\RR \Big(u(t,x)-\util\big(x-\s t-X(t)\big)\Big)^2 dx
+\frac{(u_--u_+)}{M} \n \int_0^\frac{T}{\n} |\dot{X}(t)|^2 dt\\
&\qquad
+\frac{\e}{10}\n \int_0^\frac{T}{\n} \int_\RR \Big(\big(u(t,x)-\util(x-\s t-X(t))\big)_x\Big)^2 dx dt\\
&\le \n \int_\RR \big(u_0(x)-\util(x)\big)^2 dx
= \int_\RR \Big(u_0\Big(\frac{x}{\n}\Big)-\util^\n(x)\Big)^2 dx.
\end{align*}
This establishes \eqref{nu-cont} as we desired.

\vspace{2mm}
The first part of Theorem \ref{thm:limit}, i.e., \eqref{pstu}, follows immediately.
Since we have 
\[
\int_\RR \big(\util^\n(x)-\ubar(x)\big)^2 dx
=\int_\RR \Big(\util\big(\frac{x}{\n}\big)-\ubar(x)\Big)^2 dx
=\int_\RR \Big(\util\big(\frac{x}{\n}\big)-\ubar\big(\frac{x}{\n}\big)\Big)^2 dx
=\n\int_\RR (\util(x)-\ubar(x))dx \le C\n,
\]
it follows from \eqref{nu-cont} that
\begin{align*}
&\int_\RR \big(u^\n(t,x)-\ubar(x-\s t-X_\n(t))\big)^2dx\\
&\le \int_\RR \big(u^\n(t,x)-\util^\n(x-\s t-X_\n(t))\big)^2dx
+ \int_\RR \big(\util^\n(x-\s t-X_\n(t))-\ubar(x-\s t-X_\n(t))\big)^2dx\\
&\le \int_\RR \Big(u_0\big(\frac{x}{\n}\big)-\util^\n(x)\Big)^2 dx + C\n\\
&\le \int_\RR \Big(u_0\big(\frac{x}{\n}\big)-\ubar(x)\Big)^2 dx
+\int_\RR \big(\ubar(x)-\util^\n(x)\big)^2 dx
+C\n
\le \int_\RR \Big(u_0\big(\frac{x}{\n}\big)-\ubar(x)\Big)^2 dx +C\n.
\end{align*}
This yields the desired conclusion \eqref{pstu}.

\subsection{Proof of Theorem \ref{thm:limit}: Part 2}
Finally, we prove Theorem \ref{thm:limit}.
The proof follows the argument developed in \cite{KV}.
We begin by proving the existence of well-prepared initial data, then justify the zero viscosity-dispersion limit together with its stability estimate.
Lastly, we establish control of the (limit of) shift, which yields uniqueness.

\subsubsection{Proof of \eqref{wpid}: Well-Prepared Initial Data}
Let \(u^0\) be a given initial datum.
We introduce a truncation of \(u^0\) as follows: for any \(r>1\),
\[
u_0^r \coloneqq
\begin{cases}
u^0 \one{-r^{-1}\le u^0\le r^{-1}}
+\ubar \one{u^0<-r^{-1} \text{ or }u^0>r^{-1}}, &\text{if }  -r^{-1}\le x\le r^{-1},\\
u_-, &\text{if } x\le -r^{-1},\\
u_+, &\text{if } x\ge r^{-1}.
\end{cases}
\]
We also define a smooth mollifier \(\phi_1\) which satisfies \(\supp \phi_1 \subset [-1,1]\), and set \(\phi_\n(x)\coloneqq \frac{1}{\sqrt{\n}}\phi_1\big(\frac{x}{\sqrt{\n}}\big)\).
Then, we introduce a double sequence \(\{u_0^{r,\n}\}_{r,\n>0}\) by 
\[
u_0^{r,\n} \coloneqq u_0^r \ast \phi_\n.
\]
Note that
\[
\lim_{r\to 0} \lim_{\n\to 0} u_0^{r,\n} = u^0 \quad \text{a.e.}
\]
Following \cite{KV}, a standard argument based on the dominated convergence theorem yields that
\[
\lim_{r\to0} \lim_{\n\to0} \int_\RR \big(u_0^{r,\n}-\util\big(\frac{x}{\n}\big)\big)^2 dx
= \int_{-\infty}^0 (u_0^r-u_-)^2 dx
+\int_0^\infty (u_0^r-u_+)^2 dx.
\]
Then, by a diagonal extraction, we select a sequence of smooth functions \(\{u_0^\n\}_{\n>0}\) such that 
\[
\lim_{\n\to0} \int_\RR \big(u_0^\n-\util\big(\frac{x}{\n}\big)\big)^2 dx
= \int_{-\infty}^0 (u_0^r-u_-)^2 dx
+\int_0^\infty (u_0^r-u_+)^2 dx,
\]
which establishes \eqref{wpid}.

\subsubsection{Proof of \eqref{zvdl}: Zero Viscosity-Dispersion Limits}
We now justify the zero viscosity-dispersion limits.
To this end, we derive a uniform estimate with respect to the vanishing parameter \(\n>0\), where the contraction estimate plays a crucial role.
We proceed as follows.
First, for any \(\d\in(0,1)\), we choose \(\n_*>0\) such that for all \(\n\in(0,\n_*)\),
\[
\abs{\int_\RR \big(u_0^\n(x)-\util^\n(x)\big)^2 dx
- \int_\RR \big(u^0(x)-\ubar(x)\big)^2 dx} < \d.
\]
Let \(u^\n\) be the solution to \eqref{scaled-eq0} subject to the initial data \(u_0^\n\) and let \(X_\n\) denote the shift function associated with the initial data \(u_0^\n\), as defined in \eqref{X-def}.
Then, by the analysis in Section \ref{sec:scaling}, in particular \eqref{nu-cont}, we obtain that for any \(\n\in(0,\n_*)\),
\begin{equation} \label{uniform}
\begin{aligned}
&\int_\RR \big((u^\n)(t,x)-(\util^\n)^{-X_\n}(x-\s t)\big)^2 dx
+\frac{(u_--u_+)}{M}\int_0^T |\dot{X}_\n(t)|^2 dt\\
&\qquad +\frac{\e}{10}\n \int_0^T \int_\RR \big(((u^\n)(t,x)-(\util^\n)^{-X_\n}(x-\s t))_x\big)^2 dx dt
\le \Ecal_0 + \d,
\end{aligned}
\end{equation}
where \(\Ecal_0 \coloneqq \int_\RR \big(u^0(x)-\ubar(x)\big)^2 dx\).

\vspace{2mm}
We now show the existence of zero viscosity-dispersion limits, i.e., the convergence of \(\{u_\n\}_{\n>0}\).
Given the end states \(u_\pm\), we first choose a constant \(L>1\) such that 
\[
(u_+,u_-) \subset (-L/2,L/2).
\]
We define a continuous function \(\bar{\varphi}\) by 
\[
\bar{\varphi} \coloneqq
\begin{cases}
x & \text{if } \abs{x}\le L,\\
-L & \text{if } x< -L,\\
L & \text{if } x> L.
\end{cases}
\]
Then, we set 
\[
\uun \coloneqq \bar{\varphi}(u^\n), \qquad
u_e^\n = u^\n - \uun.
\]
Since \(|u_e^\n| \le \max\{(-u^\n-L)_+, (u^\n-L)_+\}\) and \(-L<\util^\n<L\), we find that
\[
|u_e^\n| \le |u^\n-\util^\n(x-\s t-X_\n)|,
\]
which together with \eqref{uniform} implies that for any \(\n\in(0,\n_*)\),
\[
\int_\RR |u_e^\n|^2 dx \le \Ecal_0 +1.
\]
This establishes that \(\{u_e^\n\}_{\n>0}\) is bounded in \(L^\infty(0,T;L^2(\RR))\).
Moreover, since \(\{\uun\}_{\n>0}\) is bounded in \(L^\infty(0,T;L^\infty(\RR)) \subset L^\infty(0,T;L_{loc}^2(\RR))\), we obtain that \(\{u^\n\}_{\n>0}\) is bounded in \(L^\infty(0,T;L_{loc}^2(\RR))\).
Thus, there exists \(u_\infty \in L^\infty(0,T;L_{loc}^2(\RR))\) such that 
\[
u^\n \rightharpoonup u_\infty \quad \text{in } L^\infty(0,T;L_{loc}^2(\RR)).
\]
This justifies the zero viscosity-dispersion limits \eqref{zvdl}.

\subsubsection{Proof of \eqref{zvds}: Stability Estimates}
In the remainder of this section, \(C\) denotes a positive constant that may change from line to line and may depend on the prescribed states \(u_\pm\) and \(T>0\), but is independent of \(\Ecal_0\) and \(\n\).

\vspace{2mm}
To establish the stability estimate \eqref{zvds}, we first show the convergence of the shift functions \(X_\n\).
Thanks to \eqref{uniform}, for any \(\d\in(0,1)\), we choose \(\n_*>0\) such that for any \(\n\in(0,\n_*)\), 
\[
\int_0^T |\dot{X}_\n(t)|^2 dt \le C\Ecal_0+\d.
\]
Then, since \(|\dot{X}_\n(t)|\le 1 +|\dot{X}_\n(t)|^2\), \(\dot{X}_\n\) is uniformly bounded in \(L^1(0,T)\)---namely, we have 
\[
\int_0^T |\dot{X}_\n(t)| dt 
\le \int_0^T \big(1+|\dot{X}_\n(t)|^2\big) dt
\le C(\Ecal_0+1).
\]
Moreover, since \(|X_\n(t)| \le \int_0^t |\dot{X}_\n(s)|ds \le C(\Ecal_0+1)\), we find that 
\[
\int_0^T |X_\n(t)| dt \le C(\Ecal_0+1).
\]
Thus, the compactness of BV (e.g., \cite[Theorem 3.23]{AFP}) implies that there exists \(X_\infty\in BV(0,T)\) such that up to subsequence,
\begin{equation} \label{sfts-conv}
X_\n \to X_\infty \quad \text{in } L^1(0,T), \quad \text{as } \n\to0.
\end{equation}
This establishes the convergence of \(X_\n\).

\vspace{2mm}
We now proceed to prove the stability estimate \eqref{zvds}.
To this end, we first introduce a smooth mollifier \(\phi_1\) which satisfies \(\phi\ge0\), \(\int_\RR \phi(x)dx=1\) and \(\supp \phi_1 \subset [-1,1]\).
Then, for any \(t\in(0,T)\) and any \(\et>0\), we set 
\[
\phi_{t,\et}(s) \coloneqq \frac{1}{\et}\phi_1\big(\frac{s-t}{\et}\big).
\]
We consider 
\begin{align*}
&\frac{1}{2} \int_{(0,T)\times\RR} \phi_{t,\et}(s) \big(u^\n(s,x)-\ubar(x-\s s-X_\infty(s))\big)^2 dx ds\\
&\qquad
=\frac{1}{2} \int_{u^\n\in[-L,L]} \phi_{t,\et}(s) \big(\uun(s,x)-\ubar(x-\s s-X_\infty(s))\big)^2 dx ds\\
&\qquad\qquad
+\frac{1}{2} \int_{u^\n\notin [-L,L]} \phi_{t,\et}(s) \big(u^\n(s,x)-\ubar(x-\s s-X_\infty(s))\big)^2 dx ds
\eqqcolon J_1+J_2.
\end{align*}
When \(u^\n\notin [-L,L]\), since \(|u^\n(s,x)-\ubar(x-\s t-X_\infty(s))| \le 3|u^\n(s,x)-\util(x-\s s-X_\n(s))|\), we get 
\[
J_2
\le \frac{9}{2} \int_{(0,T)\times\RR} \phi_{t,\et}(s) \big(u^\n(s,x)-\util(x-\s s-X_\n(s))\big)^2 dx ds
\le C(\Ecal_0+\d)
\]
for each \(\n\in(0,\n_*)\).
To control \(J_1\), we decompose \(J_1\) it in the following way:
\begin{align*}
J_1
&\le \frac{1}{2} \int_{(0,T)\times\RR} \phi_{t,\et}(s) \big(\uun(s,x)-\ubar(x-\s s-X_\infty(s))\big)^2 dx ds \\
&=\frac{1}{2} \int_{(0,T)\times\RR} \phi_{t,\et}(s) \big(\uun(s,x)-\util^\n(x-\s s-X_\n(s))\big)^2 dx ds\\
&
+\frac{1}{2} \int_{(0,T)\times\RR} \phi_{t,\et}(s)
\Big[\big(\uun(s,x)-\ubar(x-\s s-X_\infty(s))\big)^2
-\big(\uun(s,x)-\util^\n(x-\s s-X_\n(s))\big)^2\Big]dx ds.
\end{align*}
Then, using \eqref{uniform}, we find that for any \(\n\in(0,\n_*)\), 
\[
\frac{1}{2} \int_{(0,T)\times\RR} \phi_{t,\et}(s) \big(\uun(s,x)-\util^\n(x-\s s-X_\n(s))\big)^2 dx ds
\le C(\Ecal_0+\d).
\]
Moreover, for the second term, we note that
\begin{align*}
&\abs{\big(\uun(s,x)-\ubar(x-\s s-X_\infty(s))\big)^2
-\big(\uun(s,x)-\util^\n(x-\s s-X_\n(s))\big)^2}\\
&\le C \Big|\util^\n(x-\s s-X_\n(s))-\ubar(x-\s s-X_\infty(s))\Big|\\
&\le C \Big|\util^\n(x-\s s-X_\n(s))-\ubar(x-\s s-X_\n(s))\Big|
+ C \Big|\ubar(x-\s s-X_\n(s))-\ubar(x-\s s-X_\infty(s))\Big|.
\end{align*}
Then, since we have \(\norm{\util^\n-\ubar}_{L^1(\RR)}\le C\n\) and 
\[
\norm{\ubar(x-\s s-X_\n(s))-\ubar(x-\s s-X_\infty(s))}_{L^1(\RR)}=(u_--u_+)|X_\n-X_\infty|,
\]
we find that
\[
V_\n\coloneqq
\frac{1}{2} \int_{(0,T)\times\RR} \phi_{t,\et}(s)
\Big[\big(\uun(s,x)-\ubar(x-\s s-X_\infty(s))\big)^2
-\big(\uun(s,x)-\util^\n(x-\s s-X_\n(s))\big)^2\Big]dx ds
\]
converges to \(0\) as \(\n \to 0\).
In summary, we show that for any \(\n\in(0,\n_*)\),
\[
\frac{1}{2} \int_{(0,T)\times\RR} \phi_{t,\et}(s) \big(u^\n(s,x)-\ubar(x-\s s-X_\infty(s))\big)^2 dx ds
\le C(\Ecal_0+\d) + V_\n
\]
where \(V_\n\to0\) when \(\n\to0\).
This together with the weak lower semi-continuity of the \(L^2\)-norm (for instance, see \cite{Evans-CBMS}) shows that
\[
\frac{1}{2} \int_{(0,T)\times\RR} \phi_{t,\et}(s) \big(u_\infty(s,x)-\ubar(x-\s s-X_\infty(s))\big)^2 dx ds
\le C(\Ecal_0+\d).
\]
Finally, since \(\d\in(0,1)\) is arbitrary, taking \(\et\to0\), we conclude that
\[
\frac{1}{2} \int_\RR \big(u_\infty(t,x)-\ubar(x-\s t-X_\infty(t))\big)^2 dx
\le C\Ecal_0.
\]
This completes the proof of \eqref{zvds}.

\subsubsection{Proof of \eqref{X-con}: Control of the Shift Function \(X_\infty\)}
It remains to establish \eqref{X-con}.
Since only weak convergence of \(u^\n\) is available, the limit \(u_\infty\) does not necessarily satisfy the inviscid Burgers equation.
By contrast, in the case of the barotropic Navier--Stokes equations in the Lagrangian mass coordinates, the continuity equation is linear.
As a result, inviscid limits of Navier--Stokes solutions satisfy the mass conservation law of the Euler equations.
Therefore, a different approach is required in the present setting, and we adopt the idea developed in \cite{EEK-BNSF}.

\vspace{2mm}
First, we choose \(r=r(\Ecal_0,T)>1\) such that \(\norm{X_\n}_{L^\infty}\le r-\abs{\s}T\) for any \(\n\in(0,\n_*)\).
It suffices to take \(r\) so that \(r \le C(\Ecal_0+T+1)\).
We also consider a nonnegative smooth function \(\psi\colon\RR\to\RR\) such that \(\psi(x)=\psi(-x)\) for \(x\in\RR\), \(\psi'(x)\le 0\) for \(x\ge0\), and \(\abs{\psi'(x)}\le \frac{2}{r}\), \(\abs{\psi''(x)}\le \frac{4}{r^2}\), \(\abs{\psi'''(x)} \le \frac{8}{r^3}\) and
\[
\psi(x)=
\begin{cases}
1 &\text{if } \abs{x}\le r,\\
0 &\text{if } \abs{x}\ge 2r.
\end{cases}
\]
We then define a nonnegative smooth function \(\th\colon\RR\to\RR\) such that \(\th(s)=\th(-s)\), \(\int_\RR \th(s)ds=1\) and \(\supp \th \subset [-1,1]\).
For any \(\et>0\), we set 
\[
\th_\et(s)=\frac{1}{\et}\th\Big(\frac{s-\et}{\et}\Big).
\]
Moreover, for any \(t\in(0,T)\) and any \(\et\in(0,t/2)\), we consider 
\[
\phi_{t,\et}(s) \coloneqq \int_0^s \big(\th_\et(\t)-\th_\et(\t-t)\big)d\t.
\]
We recall that \(u^\n\in\Xcal_T\) satisfies 
\[
(u^\n)_t + \Big(\frac{(u^\n)^2}{2}\Big)_x = \n \e (u^\n)_{xx}- \n^2 \d (u^\n)_{xxx},
\]
and thus, we obtain
\begin{equation} \label{unueq}
\begin{aligned}
&\int_{[0,T]\times\RR} \Big(\phi_{t,\et}'(s) \psi(x) u^\n(s,x)
+\phi_{t,\et}(s) \psi'(x) \Big(\frac{(u^\n)^2}{2}\Big)(s,x) \Big) dx ds\\
&\hspace{10mm}+ \n \e \int_{[0,T]\times\RR} \phi_{t,\et}(s) \psi''(x) u^\n dx ds
+ \n^2 \d \int_{[0,T]\times\RR} \phi_{t,\et}(s) \psi'''(x) u^\n dx ds
=0. 
\end{aligned}
\end{equation}
Note that \(\int_\RR \psi(x)u^\n(s,x) dx\) is continuous in \(s\):
\begin{align*}
&\abs{\int_\RR \psi(x)u^\n(s,x)dx-\int_\RR \psi(x) u^\n(s',x)dx}
=\abs{\int_\RR \int_{s'}^s \psi(x) (u^\n)_t(\t,x) d\t dx}\\
&\qquad
=\abs{\int_{s'}^s \int_\RR \psi'(x) \Big(-\frac{1}{2}(u^\n)^2 + \n \e (u^\n)_x\Big)
+\psi''(x) \n^2 \d (u^\n)_x dx d\t}\\
&\qquad
\le C(\e,\d,\n,r,T) \norm{u^\n}_{L^\infty(0,T;H^1(\RR))} \abs{s-s'}^\frac{1}{2}.
\end{align*}
Since \(u^\n\) and \((u^\n)^2\) are locally integrable and \(\psi,\psi',\psi''\) and \(\psi'''\) are bounded, we apply the dominated convergence theorem to pass to the limit \(\et\to0\) and find that
\begin{multline*}
\int_\RR \psi(x) u_0^\n(x)dx
-\int_\RR \psi(x) u^\n(t,x)dx
+\int_0^t \int_\RR \psi'(x) \frac{(u^\n)^2}{2}(s,x) dxds\\
=-\n \e \int_0^t \int_\RR \psi''(x) u^\n(s,x) dxds
-\n^2 \d \int_0^t \int_\RR \psi'''(x) u^\n(s,x) dxds
\eqqcolon J_1+J_2.
\end{multline*}
We decompose the left-hand side as follows:
\begin{align*}
0&=J_1+J_2
\underbrace{+\int_\RR \psi(x) \big[u^\n(t,x)-\util^\n(x-\s t-X_\n(t))\big]dx}_{\eqqcolon J_3}
\underbrace{+\int_\RR \psi(x) \big[\util^\n(x)-u_0^\n(x)\big] dx}_{\eqqcolon J_4}\\
&\qquad\underbrace{+\int_\RR \psi(x) \big[\util^\n(x-\s t-X_\n(t))-\util^\n(x)\big]dx}_{\eqqcolon J_5}
\underbrace{-\int_0^t\int_\RR \psi'(x)\frac{(u^\n-\util^\n)^2}{2}dxds}_{\eqqcolon J_6}\\
&\qquad\underbrace{-\int_0^t\int_\RR \psi'(x)\util^\n(u^\n-\util^\n)dxds}_{\eqqcolon J_7}
\underbrace{-\int_0^t\int_\RR \psi'(x)\frac{(\util^\n)^2}{2}dxds}_{\eqqcolon J_8},
\end{align*}
where \(u^\n=u^\n(s,x)\) and \(\util^\n=\util^\n(x-\s s-X_\n(s))\) in \(J_6\), \(J_7\) and \(J_8\).

To proceed further, we first note that \eqref{sfts-conv} ensures the \(L^1\)-convergence \(X_\n\to X_\infty\), and hence, up to a subsequence, we also obtain almost everywhere pointwise convergence.
In what follows, let \(\d\in(0,1)\) be arbitrarily, and let \(\n_*\) denote the corresponding constant given by \eqref{uniform}.

We now analyze the right-hand side term by term.
For \(J_1\) and \(J_2\), since \(\supp \psi \subset [-2r,2r]\),
\begin{equation} \label{J12}
\abs{J_1} + \abs{J_2} 
\le C\n \int_0^t \int_{-2r}^{2r} 1+ \big(u^\n(s,x)-\util^\n(x-\s s-X_\n(s))\big)^2 dxds
\le C(\Ecal_0+1) \n.
\end{equation}
For \(J_3\) and \(J_4\), since we have 
\[
\int_\RR (\psi(x))^2 dx
\le \int_\RR \psi(x) dx
\le 4r \le C(\Ecal_0+1),
\]
we apply \eqref{uniform} to find that for any \(\n\in(0,\n_*)\),
\begin{equation} \label{J34}
\begin{aligned}
\abs{J_3} &\le \sqrt{\int_\RR (\psi(x))^2 dx}
\sqrt{\int_\RR \big(u^\n(t,x)-\util^\n(x-\s t-X_\n(t))\big)^2 dx}
\le C \sqrt{(\Ecal_0+1)(\Ecal_0+\d)},\\
\abs{J_4} &\le \sqrt{\int_\RR (\psi(x))^2 dx}
\sqrt{\int_\RR \big(\util^\n(x)-u_0^\n(x)\big)^2 dx}
\le C \sqrt{(\Ecal_0+1)(\Ecal_0+\d)}.
\end{aligned}
\end{equation}
We postpone the analysis of \(J_5\) and \(J_8\) until the final step.
For \(J_6\) and \(J_7\), we observe that
\begin{equation} \label{J67}
\begin{aligned}
\abs{J_6}
&\le C\int_0^t 
\int_\RR \big(u^\n(s,x)-\util^\n(x-\s s-X_\n(s))\big)^2 dx ds
\le C (\Ecal_0+\d),\\
\abs{J_7}
&\le C\int_0^t \sqrt{\int_\RR (\psi'(x))^2 dx}
\sqrt{\int_\RR \big(u^\n(s,x)-\util^\n(x-\s s-X_\n(s))\big)^2 dx} ds
\le C \sqrt{\Ecal_0+\d}.
\end{aligned}
\end{equation}
We now deal with \(J_5\) and \(J_8\).
We split \(J_5\) into three pieces:
\begin{align*}
J_5
&=\int_\RR \psi(x) \big[\util^\n(x-\s t-X_\n(t))-\ubar(x-\s t-X_\n(t))\big]dx\\
&+\int_\RR \psi(x) \big[\ubar(x-\s t-X_\n(t))-\ubar(x)\big]dx
+\int_\RR \psi(x) \big[\ubar(x)-\util^\n(x)\big]dx.
\end{align*}
Note that \(\norm{\util^\n-\ubar}_{L^1(\RR)}=\n\norm{\util-\ubar}_{L^1(\RR)} \le C\n\).
In addition, using the choice of \(r\), we also have
\[
\int_\RR \psi(x) \big[\ubar(x-\s t-X_\n(t))-\ubar(x)\big]dx
=(X_\n(t)+\s t)(u_--u_+).
\]
Thus, we have
\begin{equation} \label{J5}
\abs{J_5-(X_\n(t)+\s t)(u_--u_+)} \le C\n.
\end{equation}
Meanwhile, we split \(J_8\) into two pieces:
\begin{align*}
J_8
&=-\int_0^t\int_\RR \psi'(x) \frac{1}{2}\Big(\big(\util^\n(x-\s s-X_\n(s))\big)^2
-\big(\ubar(x-\s s-X_\n(s))\big)^2\Big) dx ds\\
&\qquad
-\int_0^t\int_\RR \psi'(x) \frac{1}{2}\big(\ubar(x-\s s-X_\n(s))\big)^2 dx ds.
\end{align*}
Thanks to the fact that \(\norm{\util^\n-\ubar}_{L^1(\RR)}\le C\n\), the first term on the r.h.s. could be bounded by \(C\n\).
It also holds that 
\[
-\int_0^t\int_\RR \psi'(x) \frac{1}{2}\big(\ubar(x-\s s-X_\n(s))\big)^2 dx ds
=-\frac{t}{2}(u_-^2-u_+^2) = -\s t (u_--u_+).
\]
Hence, we obtain
\begin{equation} \label{J8}
\abs{J_8+\s t (u_--u_+)} \le C\n.
\end{equation}
Thus, gathering all from \eqref{J12} to \eqref{J8}, we conclude that for any \(\n\in(0,\n_*)\),
\[
\abs{X_\n(t)(u_--u_+)}
\le C(\Ecal_0+1)\n + C \sqrt{(\Ecal_0+1)(\Ecal_0+\d)}.
\]
Therefore, for almost every \(t\in(0,T)\), it follows that 
\[
\abs{X_\infty(t)(u_--u_+)}
\le C(\Ecal_0 + \sqrt{\Ecal_0}).
\]
This completes \eqref{X-con} and Theorem \ref{thm:limit}. \qed

\appendix
\section{Global Existence of $H^1$ Solutions} \label{sec:exist}
\setcounter{equation}{0}
This section concerns the global existence of solution to \eqref{burgers} in the class \(\Xcal_T\).
We are interested in solutions with different asymptotic states as \(\abs{x}\to\infty\).
Throughout this section, the coefficients \(\e\) and \(\d\) do not affect our analysis; hence, we may assume \(\e=\d=1\).

Assuming that the initial data \(u_0\) satisfying \(u_0-\uubar\in H^1(\RR)\), we aim to show the global existence of solutions to \eqref{burgers} such that for any \(T>0\), \(u-\uubar \in \Scal_T\), where \(\Scal_T\) is given by
\[
\Scal_T \coloneqq \{f\colon\RR^+\times\RR\to\RR\mid f\in C([0,T];H^1(\RR)) \cap L^2(0,T;H^2(\RR))\}.
\]
We remark that \(u-\uubar\in\Scal_T\) is equivalent to \(u\in\Xcal_T\).
The global existence of solutions follows from local well-posedness and an a priori estimate through a standard continuation argument.
We first present the local well-posedness lemma.

\begin{lemma}\label{lem:LWP}
Let \(u_-\) and \(u_+\) be two given constant states.
Let \(\uubar\) be a smooth monotone function satisfying \eqref{def:ubar}.
Then, for any \(M_0>0\), there exists \(\hat{T}>0\) such that the following holds:
For any initial datum \(u_0\) which satisfies \(u_0-\uubar \in H^1(\RR)\) and
\[
\norm{u_0-\uubar}_{H^1(\RR)} \le M_0,
\]
there exists a unique solution \(u\) to \eqref{burgers} on \([0,\hat{T}]\) subject to the initial value \(u_0\) such that 
\[
u-\uubar \in C([0,\hat{R}];H^1(\RR)) \cap L^2(0,\hat{T};H^2(\RR)), \qquad
\sup_{t\in[0,\hat{T}]} \norm{u(t)-\uubar}_{H^1(\RR)} \le 2M_0.
\]
\end{lemma}

\begin{pf}
The proof of the lemma relies on a classical method, but to make the paper self-contained, we provide it in the below.
The proof consists of five steps.
\step{1} We first define the perturbation variable 
\[
w(t,x) \coloneqq u(t,x)-\ubar(x).
\]
The initial condition becomes \(w(0)=w_0\coloneqq u_0-\ubar\in H^1(\RR)\), and the \(w\) variable satisfies
\begin{equation} \label{w-eq}
w_t + w_{xxx} - w_{xx} + ww_x - \ubar w_x - \ubar_x w
= F,
\end{equation}
where \(F\coloneqq -\ubar \ubar_x + \ubar_{xxx} - \ubar_{xx}\).
It is obvious that \(F\in C_c^\infty(\RR) \subset H^1(\RR) \subset H^{-1}(\RR)\).

We now introduce a linear differential operator \(G\): 
\[
G\coloneqq -\rd_x^3 +\rd_x^2, \qquad 
D(G) = H^3(\RR).
\]
Notice that \(G\) generates a contraction \(C_0\)-semigroup \(\{S(t)\}_{t\ge0}\) on \(L^2(\RR)\) (see \cite[Proposition 2.1]{CCKR-AIHP}).

\step{2}
We consider the linear problem 
\[
v_t-v_{xxx}+v_{xx} = f, \qquad
v(0,x)=v_0(x).
\]
In this step, we show that for any \(v_0\in H^1(\RR)\) and any \(f\in L^2(0,T;H^{-1}(\RR))\), the following holds:
the mild solution satisfies \(v\in C([0,T];H^1(\RR))\cap L^2(0,T;H^2(\RR))\), and
\begin{equation} \label{lin-energy}
\norm{v}_{L^\infty(0,T;H^1(\RR))}^2
+\norm{v}_{L^2(0,T;H^2(\RR))}^2
\le C\big(\norm{v_0}_{H^1(\RR)}^2+\norm{f}_{L^2(0,T;H^{-1}(\RR))}^2\big).
\end{equation}
To this end, we first construct the mild solution via the semigroup and the Duhamel formula: 
\begin{equation} \label{mild-sol}
v(t)=S(t)v_0
+\int_0^t S(t-s)f(s) ds.
\end{equation}
The strong continuity of the semigroup ensures that the initial condition is satisfied.
We note that 
\[
\frac{1}{2}\frac{d}{dt} \norm{v}_{L^2(\RR)}^2
+ \norm{v_x}_{L^2(\RR)}^2
= <f,v>_{H^{-1},H^1}.
\]
Moreover, it holds that 
\[
\frac{1}{2}\frac{d}{dt} \norm{v_x}_{L^2(\RR)}^2
+ \norm{v_{xx}}_{L^2(\RR)}^2
=<f,-v_{xx}>_{H^{-1},H^1}.
\]
Then, it follows that for some small constant \(\et\),
\[
\frac{1}{2}\frac{d}{dt} \norm{v}_{H^1(\RR)}^2
+ \norm{v_x}_{L^2(\RR)}^2
+ \norm{v_{xx}}_{L^2(\RR)}^2
\le C\norm{f}_{H^{-1}(\RR)} \norm{v}_{H^2(\RR)}
\le \et\norm{v}_{H^2(\RR)}^2
+C_\et \norm{f}_{H^{-1}(\RR)}^2.
\]
We then apply Gronwall's lemma to establish \eqref{lin-energy}.
As the mild solution \eqref{mild-sol} is a priori only in \(C([0,T];H^{-1}(\RR))\), the preceding formal calculations could be rigorously justified by approximating the data with smooth functions and passing to the limit via a standard density argument.

\step{3}
In this step, prior to introducing the iteration scheme, we derive estimates for functions in the class \(\Scal_T\).
These estimates allow us to apply the result of \textit{Step 2} to the approximate solutions arising from the iteration scheme, thereby obtaining uniform bounds.

Let \(v\) be any function in \(\Scal_T\).
Then, using the following basic functional inequalities 
\begin{equation} \label{ftnal-ineq}
\norm{fg}_{H^{-1}(\RR)}
\le\norm{f}_{L^\infty(\RR)}\norm{g}_{L^2(\RR)}, \qquad
\norm{f}_{L^\infty(\RR)} \le C\norm{f}_{H^1(\RR)},
\end{equation}
we find that
\begin{equation} \label{non-1}
\norm{v v_x}_{L^2(0,T;H^{-1}(\RR))}
\le C\sqrt{T} \norm{v}_{L^\infty(0,T;H^1(\RR))}^2,
\end{equation}
and
\begin{equation} \label{non-2}
\norm{\ubar v_x}_{L^2(0,T;H^{-1}(\RR))}
+\norm{\ubar_x v}_{L^2(0,T;H^{-1}(\RR))}
\le C\sqrt{T} \norm{v}_{L^\infty(0,T;H^1(\RR))}.
\end{equation}
On the other hand, since \(F\in C_c^\infty(\RR) \subset H^{-1}(\RR)\) and \(F\) is stationary, we have 
\begin{equation} \label{non-3}
\norm{F}_{L^2(0,T;H^{-1}(\RR))}
=\sqrt{T} \norm{F}_{H^{-1}(\RR)}.
\end{equation}
These estimates \eqref{non-1}-\eqref{non-3} will be crucially used later.

\step{4}
We now introduce the iteration scheme as follows.
Firstly, we set \(w^{(0)}(t,x)=w_0(x)\), and then, for each \(n\in\NN\cup\{0\}\), we define \(w^{(n+1)}\) to be the solution of
\[
\begin{cases}
(w^{(n+1)})_t + (w^{(n+1)})_{xxx} - (w^{(n+1)})_{xx}
= -w^{(n)}(w^{(n)})_x
-\ubar(w^{(n)})_x
-\ubar_x w^{(n)}
+F,\\
(w^{(n+1)})(0,x)=w_0(x).
\end{cases}
\]
We define \(\Phi(w^{(n)})=w^{(n+1)}\).
Notice that the map \(\Phi\) is well-defined on \(\Scal_T\) by \textit{Step 2} and \textit{Step 3}.
Then, using \eqref{lin-energy} and \eqref{non-1}-\eqref{non-3}, we find that for each \(n\in\NN\cup\{0\}\),
\[
\|w^{(n+1)}\|_{\Scal_T}
\le C\big( \norm{w_0}_{H^1(\RR)}
+\sqrt{T} (\|w^{(n)}\|_{L^\infty(0,T;H^1(\RR))}^2
+\|w^{(n)}\|_{L^\infty(0,T;H^1(\RR))}+1)\big).
\]
This shows that there exist \(R>0\) and \(T>0\) such that \(R\le2M_0\) and
\[
\|w^{(n)}\|_{\Scal_T} \le R \quad
\Rightarrow \quad
\|w^{(n+1)}\|_{\Scal_T} \le R.
\]
This provides an uniform estimate.

\step{5} In this step, we show that \(\{w^{(n)}\}\) is a Cauchy sequence in \(\Scal_T\).
It suffices to prove that 
\begin{equation} \label{shrink}
\|w^{(n+1)}-w^{(n)}\|_{\Scal_T}
\le C\sqrt{T} (R+1) \|w^{(n)}-w^{(n-1)}\|_{\Scal_T}.
\end{equation}
If then, choosing small \(T>0\) so that \(C\sqrt{T} (R+1)<\frac{1}{2}\), we get a contraction and a Cauchy sequence.

Let \(z^{(n)}\coloneqq w^{(n+1)}-w^{(n)}\).
Then, it satisfies \(z^{(n)}(0,x)=0\) and 
\begin{align*}
&(z^{(n+1)})_t + (z^{(n+1)})_{xxx} - (z^{(n+1)})_{xx}\\
&\qquad= -z^{(n)}(w^{(n+1)})_x
-w^{(n)}(z^{(n)})_x
-\ubar(z^{(n)})_x
-\ubar_x z^{(n)}.
\end{align*}
To apply \eqref{lin-energy}, we need to establish an upper bound on the \(L_t^2 H_x^{-1}\) norm of the right-hand side.
Thanks to \eqref{ftnal-ineq}, we obtain
\begin{align*}
\|z^{(n)}(w^{(n+1)})_x\|_{L^2(0,T;H^{-1}(\RR))}
&\le C \sqrt{T} \|z^{(n)}\|_{L^\infty(0,T;H^1(\RR))}
\|w^{(n+1)}\|_{L^\infty(0,T;H^1(\RR))},\\
\|w^{(n)}(z^{(n)})_x\|_{L^2(0,T;H^{-1}(\RR))}
&\le C \sqrt{T} \|w^{(n)}\|_{L^\infty(0,T;H^1(\RR))}
\|z^{(n)}\|_{L^\infty(0,T;H^1(\RR))},\\
\|\ubar(z^{(n)})_x+\ubar_x z^{(n)}\|_{L^2(0,T;H^{-1}(\RR))}
&\le C \sqrt{T} \|z^{(n)}\|_{L^\infty(0,T;H^1(\RR))},
\end{align*}
from which, it follows that 
\[
\norm{\text{r.h.s.}}_{L^2(0,T;H^{-1}(\RR))} 
\le C\sqrt{T}(R+1) \|z^{(n)}\|_{\Scal_T}.
\]
We now apply \eqref{lin-energy} to find that 
\[
\|z^{(n+1)}\|_{\Scal_T}
\le C\sqrt{T}(R+1) \|z^{(n)}\|_{\Scal_T}.
\]
This is equivalent to \eqref{shrink}.
Thus, the sequence \(\{w^{(n)}\}\) is Cauchy in the Banach space \(\Scal_T\), and it converges to some \(w\in\Scal_T\).
Then, the contraction from \eqref{shrink} shows that \(\Phi\) is Lipschitz continuous, and so \(w\) is a fixed point, i.e., \(\Phi(w)=w\).
In view of \eqref{mild-sol}, \(w\) is now the mild solution, and \textit{Step 2} and \textit{Step 4} yield the desired conclusion.
Finally, the uniqueness follows from the contraction.
\end{pf}

\vspace{2mm}
We next state a lemma on a priori estimates.
\begin{lemma} \label{lem:apriori}
Let \(u_-\) and \(u_+\) be two given constant states.
Let \(\uubar\) be a smooth monotone function satisfying \eqref{def:ubar}.
Let \(u_0\) be an initial datum with \(u_0-\uubar \in H^1(\RR)\).
Then, if \(u\) is a solution of \eqref{burgers} on \([0,T_0)\) for some \(T_0>0\) such that \(u\in\Xcal_T\) (or equivalently \(u-\uubar\in\Scal_T\)) for all \(T\in(0,T_0)\), then there exists a constant \(C(T_0)\) such that 
\[
\sup_{t\in[0,T_0)} \norm{u-\uubar}_{H^1(\RR)} \le C(T_0).
\]
\end{lemma}

\begin{pf}
Throughout the proof, \(C\) denotes a positive constant that may vary from line to line and depend on the initial datum \(\norm{u_0-\uubar}_{H^1(\RR)}\), the given states \(u_-\) and \(u_+\), and \(T_0\), but is independent of \(T\in(0,T_0)\).

\vspace{2mm}
We employ the standard energy method.
We first observe from \eqref{burgers} that
\begin{equation} \label{L2-tot}
\frac{1}{2}\frac{d}{dt}\int_\RR (u-\uubar)^2 dx
=-\int_\RR (u-\uubar) u u_x dx
+\int_\RR (u-\uubar) u_{xx} dx
-\int_\RR (u-\uubar) u_{xxx} dx.
\end{equation}
We analyze the right-hand side on a term-by-term basis.
The first term reduces to
\begin{align*}
-\int_\RR (u-\uubar) u u_x dx
&=-\int_\RR (u-\uubar)^2(u-\uubar)_x dx
-\int_\RR \uubar(u-\uubar)(u-\uubar)_x dx\\
&\qquad
-\int_\RR (u-\uubar)^2 \uubar_x dx
-\int_\RR (u-\uubar) \uubar \uubar_x dx\\
&=-\frac{1}{2}\int_\RR (u-\uubar)^2 \uubar_x dx
-\int_\RR (u-\uubar) \uubar \uubar_x dx.
\end{align*}
Note that \(\uubar\) and \(\uubar_x\) are bounded and \(\supp \uubar_x \subset [-1,1]\).
Then, using Young's inequality, we obtain 
\begin{equation} \label{L2-first}
-\int_\RR (u-\uubar) u u_x dx
\le C \int_\RR (u-\uubar)^2 dx + C.
\end{equation}
Then, we apply integration by parts to the second term: 
\begin{equation} \label{L2-sec}
\int_\RR (u-\uubar) u_{xx} dx
=-\int_\RR (u_x)^2 dx
+\int_\RR \uubar_x u_x dx
\le -\frac{1}{2}\int_\RR (u_x)^2 dx
+C.
\end{equation}
We also find that 
\begin{equation} \label{L2-third}
\begin{aligned}
-\int_\RR (u-\uubar) u_{xxx} dx
&=-\int_\RR (u-\uubar) (u-\uubar)_{xxx} dx
-\int_\RR (u-\uubar) \uubar_{xxx} dx\\
&=-\int_\RR (u-\uubar) \uubar_{xxx} dx
\le C\int_\RR (u-\uubar)^2 dx + C.
\end{aligned}
\end{equation}
Gathering \eqref{L2-tot}, \eqref{L2-first}, \eqref{L2-sec} and \eqref{L2-third}, we obtain 
\[
\frac{1}{2}\frac{d}{dt}\int_\RR (u-\uubar)^2 dx
+\frac{1}{2}\int_\RR (u_x)^2 dx
\le C\int_\RR (u-\uubar)^2 dx + C.
\]
Then, Gronwall's lemma implies that
\begin{equation} \label{L2-E}
\sup_{t\in[0,T]}\int_\RR (u-\uubar)^2 dx
+\int_0^T \int_\RR (u_x)^2 dx dt
\le C.
\end{equation}

Now we perform the energy method in order to obtain a bound on the derivative.
To this end, using integration by parts, we observe from \eqref{burgers} that
\begin{align*}
\frac{1}{2}\frac{d}{dt}\int_\RR (u_x)^2 dx
&=-\int_\RR u u_x u_{xx} dx
-\int_\RR (u_x)^3 dx
+\int_\RR u_x u_{xxx} dx
-\int_\RR u_x u_{xxxx} dx\\
&=\int_\RR u u_x u_{xx} dx
-\int_\RR (u_{xx})^2 dx.
\end{align*}
Since it holds that
\[
\int_\RR u u_x u_{xx} dx
\le \frac{1}{2}\int_\RR (u_{xx})^2 dx
+ C\int_\RR u^2 (u_x)^2 dx
\le \frac{1}{2}\int_\RR (u_{xx})^2 dx
+ C\norm{u}_{L^\infty}^2 \int_\RR (u_x)^2 dx,
\] 
we obtain 
\[
\frac{1}{2}\frac{d}{dt}\int_\RR (u_x)^2 dx
+\frac{1}{2}\int_\RR (u_{xx})^2 dx
\le C\norm{u}_{L^\infty(\RR)}^2 \int_\RR (u_x)^2 dx.
\]
On the other hand, it follows from \eqref{L2-E} that 
\[
\norm{u}_{L^2(0,T;L^\infty(\RR))}\le C.
\]
Then, Gronwall's lemma shows that 
\[
\sup_{t\in[0,T]}\int_\RR (u_x)^2 dx
+\int_0^T \int_\RR (u_{xx})^2 dx dt
\le C.
\]
This together with \eqref{L2-E} yields the desired conclusion.
\end{pf}

\vspace{2mm}
The combination of these two lemmas allows us to demonstrate the global-in-time existence via a standard continuation argument.
Therefore, for any \(u_0\) with \(u_0-\uubar \in H^1(\RR)\), there exists a unique global-in-time solution \(u\in\Xcal_T\).

\section{Proof of Lemma \ref{lem:RHS}} \label{section_proof}
\setcounter{equation}{0}
In this section, we show Lemma \ref{lem:RHS}.
First, using \eqref{burgers} and \eqref{visS}, we observe that
\[
\big(u-\util^{-X}\big)_t = \dot{X}(t) (\util')^{-X} - \bigg(\frac{u^2}{2}-\frac{(\util^{-X})^2}{2}\bigg)_\x + \big(u-\util^{-X}\big)_{\x\x} - \d\big(u-\util^{-X}\big)_{\x\x\x}.
\]
Then, it follows that
\begin{align*}
\frac{d}{dt}\frac{1}{2}\int_\RR \big(u-\util^{-X}\big)^2 d\x
&=\int_\RR \big(u-\util^{-X}\big)\big(u-\util^{-X}\big)_t d\x\\
&=\dot{X}(t)\int_\RR (u-\util^{-X})(\util')^{-X} d\x
-\int_\RR \big(u-\util^{-X}\big) \bigg(\frac{u^2}{2}-\frac{(\util^{-X})^2}{2}\bigg)_\x d\x\\
&\qquad
+\int_\RR \big(u-\util^{-X}\big)(u-\util^{-X})_{\x\x} d\x
-\d\int_\RR \big(u-\util^{-X}\big)\big(u-\util^{-X}\big)_{\x\x\x} d\x.
\end{align*}
We analyze the right-hand side on a term-by-term basis as follows: first, we have
\begin{align*}
&-\int_\RR (u-\util^{-X})\bigg(\frac{u^2}{2}-\frac{(\util^{-X})^2}{2}\bigg)_\x d\x
=\int_\RR (u-\util^{-X})_\x\bigg(\frac{u^2}{2}-\frac{(\util^{-X})^2}{2}\bigg) d\x\\
&\qquad\qquad=\frac{1}{2}\int_\RR \big(u-\util^{-X}\big)_\x
\Big(\big(u-\util^{-X}\big)^2 + 2\util^{-X}\big(u-\util^{-X}\big)\Big) d\x\\
&\qquad\qquad=\frac{1}{2}\int_\RR \Big(\big(u-\util^{-X}\big)^2\Big)_\x \util^{-X} d\x
=-\frac{1}{2}\int_\RR \big(u-\util^{-X}\big)^2 (\util')^{-X} d\x,
\end{align*}
the dissipation yields a non-positive contribution, formulated as 
\[
\int_\RR \big(u-\util^{-X}\big)\big(u-\util^{-X}\big)_{\x\x} d\x
=-\int_\RR \Big(\big(u-\util^{-X}\big)_\x\Big)^2 d\x,
\]
and the dispersion effect vanishes as 
\[
-\d\int_\RR \big(u-\util^{-X}\big)\big(u-\util^{-X}\big)_{\x\x\x} d\x
=\d\int_\RR \big(u-\util^{-X}\big)_x\big(u-\util^{-X}\big)_{\x\x} d\x =0.
\]
Thus, gathering all, \eqref{key1} is established as we desired. \qed

\section{Proof Details for the Decreasing Intervals} \label{appendix:dec}
\setcounter{equation}{0}
This section is devoted to the proof details for the decreasing intervals of shock profile, i.e., the intervals \(I_i=(\x_i,\x_{i-1})\) with even \(i\ge2\).
Note that for even \(i\ge2\), \(\util'<0\) on \(I_i\) and
\[
\util(\x_{i-1})=u_{i-1} < s < \util(\x_i) = u_i, \qquad
\util_\x(\x_{i-1})=\util_\x(\x_i)=0.
\]
Indeed, from \(E(\x_{i-1})>E(\x_i)\), it simply follows that \(s-u_{i-1} > u_i-s\).

\subsection{Proof of \eqref{dec-decay} in Theorem \ref{thm:shock}}
The idea is similar to that for the increasing intervals, and the key ingredient of the proof is an approximation function for the derivative of shock, as formalized in the following proposition.

\begin{proposition} \label{prop:dec}
Let \(A\) be a constant which satisfies \(\frac{1}{4}<A\le1\).
Let \(k>0\) be a constant.\\
For each even \(i \in \NN\), define \(\l_i \coloneqq k \sqrt{\frac{1}{A}} \sqrt{\frac{s(u_i^2-s^2)}{u_i-u_{i-1}}}\).
Assume that 
\begin{equation} \label{dec-ass}
\frac{1}{2} \le k \le \min \bigg(1, \sqrt{\frac{s^2-u_{i-1}^2}{u_i^2-s^2}}\bigg)
\end{equation}
and that the local extrema do not approach \(s\) too rapidly, in the sense that
\begin{equation} \label{ass10}
\r\coloneqq\frac{s-u_{i-1}}{u_i-s}<10.
\end{equation}
Then, the following holds:
\begin{equation} \label{dec-ineq}
-\util'(\x) \ge \l_i (u_i-\util(\x))^\frac{1}{2}(\util(\x)-u_{i-1})^\frac{1}{2}, \quad \forall \x \in I_i,
\end{equation}
\begin{equation} \label{dec-ui}
3(\r-1) \ge \frac{3k}{4}\pi \sqrt{\frac{2}{A}} \sqrt{\r+1}.
\end{equation}
\end{proposition}

\paragraph{Proof of \eqref{dec-ineq}:}
To begin with, we define two functions of \(a=\util\) as follows: 
\[
h(a) \coloneqq \util' = \util_\x, \qquad
p(a) \coloneqq -\l_i (u_i-a)^\frac{1}{2} (a-u_{i-1})^\frac{1}{2}.
\]
We now rewrite \eqref{dec-ineq} into the following form:
\begin{equation} \label{dec-ineq2}
h(a) \le p(a), \quad \forall a \in (u_{i-1},u_i).
\end{equation}
Then we observe the two functions at the endpoints.
First, since the shock has local extrema at \(\x_{i-1}\) and \(\x_i\), the functions vanish at the endpoints:
\[
h(u_{i-1})=p(u_{i-1})=0, \qquad
h(u_i)=p(u_i)=0.
\]
To observe the derivatives \(h'\) and \(p'\), we note by \eqref{hp} that
\[
h'(a) = \frac{d\util_\x}{d\util} = \frac{1}{\d} \Big[1 - \frac{1}{2h(a)}(a-s)(a+s)\Big]
\]
and that
\[
p'(a)=
\frac{\l_i}{2} (u_i-a)^{-\frac{1}{2}} (a-u_{i-1})^\frac{1}{2}
-\frac{\l_i}{2} (u_i-a)^\frac{1}{2} (a-u_{i-1})^{-\frac{1}{2}}.
\]
Both \(h'(a)\) and \(p'(a)\) diverge as \(a\) approaches to the endpoints, i.e., when \(a \to u_{i-1}\) and \(a \to u_i\).
The divergence rate of \(p'(a)\) is given as follows: 
\[
\lim_{a\to u_{i-1}+} p'(a)(a-u_{i-1})^\frac{1}{2} = -\frac{\l_i}{2} (u_i-u_{i-1})^\frac{1}{2}, \qquad
\lim_{a\to u_i-} p'(a)(u_i-a)^\frac{1}{2} = \frac{\l_i}{2} (u_i-u_{i-1})^\frac{1}{2}.
\]
Moreover, similarly to the computations for \eqref{hp-ui} and \eqref{hp-ui-1}, we obtain 
\[
\lim_{a\to u_{i-1}+} h'(a)(a-u_{i-1})^\frac{1}{2}
=-\frac{1}{2}\sqrt{\frac{s^2-u_{i-1}^2}{\d}}, \qquad
\lim_{a\to u_i-} h'(a)(u_i-a)^\frac{1}{2}
=\frac{1}{2}\sqrt{\frac{u_i^2-s^2}{\d}}.
\]
Then, from the assumption \eqref{dec-ass}, it follows that
\[
-\lim_{a\to u_{i-1}+} h'(a)(a-u_{i-1})^\frac{1}{2}
=\frac{1}{2}\sqrt{\frac{s^2-u_{i-1}^2}{\d}}
> \frac{\l_i}{2} (u_i-u_{i-1})^\frac{1}{2}
=-\lim_{a\to u_{i-1}+} p'(a)(a-u_{i-1})^\frac{1}{2}
\]
and
\[
\lim_{a\to u_i-} h'(a)(u_i-a)^\frac{1}{2}
=\frac{1}{2}\sqrt{\frac{u_i^2-s^2}{\d}}
>
\frac{\l_i}{2} (u_i-u_{i-1})^\frac{1}{2}
=\lim_{a\to u_i-} p'(a)(u_i-a)^\frac{1}{2}.
\]

We now argue by contradiction and assume that \eqref{dec-ineq2} fails.
The above comparison of \(h\) and \(p\) at the endpoints implies that there exist two points \(b\) and \(c\) such that \(u_{i-1}<b<c<u_i\) and
\begin{align*}
&h(b)=p(b), &&h'(b)-p'(b) \ge 0\\
&h(c)=p(c), &&h'(c)-p'(c) \le 0.
\end{align*}
We define a function \(g:[u_{i-1},u_i]\to\RR\) by
\[
g(a) \coloneqq
\l_i(u_i-a)^\frac{1}{2}(a-u_{i-1})^\frac{1}{2}
+\frac{1}{2}(a-s)(a+s)
+\frac{\l_i^2\d}{2} (u_i+u_{i-1}-2a).
\]
Then, for \(a=b\) and \(a=c\), we have
\[
h'(a)-p'(a)
=\frac{1}{\d}\Big[1-\frac{1}{2p(a)}(a-s)(a+s)\Big]
-\frac{\l_i}{2}(u_i-a)^{-\frac{1}{2}} (a-u_{i-1})^\frac{1}{2}
+\frac{\l_i}{2}(u_i-a)^\frac{1}{2} (a-u_{i-1})^{-\frac{1}{2}},
\]
and thus it follows that
\begin{align*}
g(b) = \l_i \d (h'(b)-p'(b)) (u_i-b)^\frac{1}{2} (b-u_{i-1})^\frac{1}{2} \ge0,\\
g(c) = \l_i \d (h'(c)-p'(c)) (u_i-c)^\frac{1}{2} (c-u_{i-1})^\frac{1}{2} \le0.
\end{align*}
Thanks to the assumption \eqref{dec-ass}, the sign of the function \(g\) at the endpoints is determined: 
\[
g(u_i)
= \frac{1}{2} (u_i^2-s^2) - \frac{\l_i^2\d}{2} (u_i-u_{i-1})
= \frac{u_i^2-s^2}{2} \Big(1-k^2\frac{\d s}{A}\Big) >0,
\]
\[
g(u_{i-1})
= -\frac{1}{2} (s^2-u_{i-1}^2) + \frac{\l_i^2\d}{2} (u_i-u_{i-1})
= -\frac{1}{2} \Big( (s^2-u_{i-1}^2) -k^2\frac{\d s}{A} (u_i^2-s^2) \Big) <0.
\]
On the other hand, we claim that the function \(g\) is strictly concave.
This can be verified as follows.
First, using \eqref{u0-exp} and \eqref{inc-decay}, we obtain \(u_i-u_{i-1}<u_0-u_1<\frac{1}{5}s\).
This with \eqref{ass10} yields that
\[
(u_i-u_{i-1})^3
< \frac{s^2 (u_i-u_{i-1})}{25}
< \frac{s^2 (u_i-u_{i-1})}{20A}
< \frac{s^2 (s-u_{i-1})}{10A}
< \frac{s^2 (u_i-s)}{A}
< 4k^2\frac{s (u_i^2-s^2)}{A},
\]
from which the following holds: 
\[
g''(a)
=1 -\frac{\l_i}{4} \frac{(u_i-u_{i-1})^2}{(u_i-a)^\frac{3}{2} (a-u_{i-1})^\frac{3}{2}}
\le 1 -\frac{\l_i}{4} \frac{(u_i-u_{i-1})^2}{(\frac{u_i-u_{i-1}}{2})^3}
= 1-\frac{2\l_i}{u_i-u_{i-1}} <0.
\]
To summarize, the function \(g\) is concave and 
\[
g(u_{i-1}) < 0, \qquad
g(b) \ge 0, \qquad
g(c) \le 0, \qquad
g(u_i) > 0, \qquad
(u_{i-1}<b<c<u_i).
\]
This is a contradiction.
This completes the proof of \eqref{dec-ineq2} and \eqref{dec-ineq}.

\paragraph{Proof of \eqref{dec-ui}:}
We use \eqref{dec-ineq} together with \eqref{t1}-\eqref{t2} to find that
\begin{align*}
E(\x_{i-1})-E(\x_i)
&=\frac{1}{6}(u_{i-1}-u_i)(u_{i-1}^2+u_{i-1}u_i+u_i^2-3s^2)
=\int_{\x_i}^{\x_{i-1}} (\util')^2 d\x\\
&\ge \l_i \int_{\x_i}^{\x_{i-1}} (u_i-\util)^\frac{1}{2} (\util-u_{i-1})^\frac{1}{2} (-\util') d\x\\
&= \l_i \int_{u_{i-1}}^{u_i} (u_i-\util)^\frac{1}{2} (\util-u_{i-1})^\frac{1}{2} d\util
= \frac{\pi}{8}\l_i (u_i-u_{i-1})^2.
\end{align*}
Then, it follows that
\[
3s^2-u_{i-1}^2-u_{i-1}u_i-u_i^2
\ge \frac{3k}{4}\pi \sqrt{\frac{1}{A}} \sqrt{s(u_i^2-s^2)} \sqrt{u_i-u_{i-1}}.
\]
Letting \(u_{i-1}=(1-\b)s\) and \(u_i=(1+\a)s\), we obtain
\[
3(\b-\a)-\a^2-\b^2+\a\b \ge \frac{3k}{4}\pi \sqrt{\frac{1}{A}} \sqrt{2\a+\a^2} \sqrt{\a+\b}.
\]
Then, \(\r=\frac{\b}{\a}=\frac{s-u_{i-1}}{u_i-s}>1\), and we rewrite the above inequality into the following form:
\[
3(\r-1)-(\r^2-\r+1)\a
\ge \frac{3k}{4}\pi \sqrt{\frac{1}{A}} \sqrt{2+\a}\sqrt{\r+1}.
\]
Since \(\a>0\), it follows that
\[
3(\r-1) \ge \frac{3k}{4}\pi \sqrt{\frac{2}{A}} \sqrt{\r+1}.
\]
This completes the proof of \eqref{dec-ui} and Proposition \ref{prop:dec}. \qed

\vspace{2mm}
Proposition \ref{prop:dec} leads us to the following decay rate \eqref{dec-decay} in Theorem \ref{thm:shock}.

\paragraph{Proof of \eqref{dec-decay} in Theorem \ref{thm:shock}.}
It suffices to consider the case where \eqref{ass10} holds.\\
We first recall from \eqref{inc-decay} and \eqref{u0-exp} that, for each selected \(A\), \(u_1\ge u_{i-1}>0.9s\), which implies \(u_i\le u_2<1.1s\).
Thus, it follows that
\[
\sqrt{\frac{s^2-u_{i-1}^2}{u_i^2-s^2}}
> \sqrt{\frac{u_{i-1}+s}{u_i+s}}
> \sqrt{\frac{2}{2.1}} > \frac{11}{12}.
\]
This shows that \(k=\frac{11}{12}\) satisfies the condition \eqref{dec-ass}.
Then, using \eqref{dec-ui}, we obtain 
\[
3(\r-1)
\ge \frac{3k}{4}\pi \sqrt{\frac{2}{A}} \sqrt{\r+1}
=\frac{11}{16}\pi \sqrt{\frac{2}{A}} \sqrt{\r+1}
\]
where \(\r=\frac{s-u_{i-1}}{u_i-s}>1\).
Notice that for any \(\frac{1}{4}<A\le1\), the left-hand side grows faster than the right-hand side, and the above inequality fails at \(\r=3\).
Thus, we have \(\r\ge3\), which in turn implies
\[
\sqrt{\frac{s^2-u_{i-1}^2}{u_i^2-s^2}} > 1.
\]
Hence, \(k=1\) now satisfies the condition \eqref{dec-ass}, and so we use \eqref{dec-ui} to find that
\[
3(\r-1)
\ge \frac{3}{4}\pi \sqrt{\frac{1}{A}} \sqrt{2} \sqrt{\r+1}.
\]
Once this inequality is resolved, the desired decay rate \eqref{dec-decay} for each selected value \(A\) follows by an argument similar to that for the increasing case.
This completes the proof of \eqref{dec-decay}.
\qed

\subsection{Proof of Proposition \ref{prop:key} for the Decreasing Intervals}
This subsection is devoted to the proof of \textit{Case 3} in Proposition \ref{prop:key}, concerning the decreasing intervals of the shock profile other than the rightmost one.
The proof is based on the contradiction argument developed in this paper.
Note that we consider only the case \(A=\frac{1}{2}\).

\vspace{2mm}
\begin{pf}
Firstly, we define two functions in terms of \(a=\util\) as follows: 
\[
h(a)\coloneqq\util'=\util_\x, \qquad
p(a) \coloneqq \lbar_i(a-u_i)(a-u_{i-1}).
\]
Our goal is now to prove
\begin{equation} \label{dec:key2}
h(a) \le p(a), \quad \forall a \in (u_{i-1},u_i).
\end{equation}
As before, we examine the two functions at the endpoints.
It simply holds that 
\[
h(u_{i-1})=p(u_{i-1})=0, \qquad
h(u_i)=p(u_i)=0.
\]
Moreover, since \(p'(a)=\lbar_i(2a-u_{i-1}-u_i)\), it follows that
\[
p'(u_i)=\lbar_i(u_i-u_{i-1}), \qquad
p'(u_{i-1})=-\lbar_i(u_i-u_{i-1}),
\]
and thus, we obtain 
\[
\lim_{a\to u_i-} h'(a) = +\infty > p'(u_i), \qquad
\lim_{a\to u_{i-1}+} h'(a) = -\infty < p'(u_{i-1}).
\]

We argue by contradiction to prove \eqref{dec:key2} and assume that \(h(a)>p(a)\) for some \(a\in(u_{i-1},u_i)\).
As a consequence of the observations above, we choose \(b\) and \(c\) such that \(u_{i-1}<b<c<u_i\) and
\begin{align*}
&p(b)=h(b), &&h'(b)-p'(b) \ge 0\\
&p(c)=h(c), &&h'(c)-p'(c) \le 0.
\end{align*}
To obtain a contradiction, we define a function \(g:[u_{i-1},u_i]\to\RR\) as follows:
\[
g(a) \coloneqq
2\lbar_i(u_i-a)(a-u_{i-1})
+(a-s)(a+s)
-2\lbar_i^2\d(u_i-a)(a-u_{i-1})(2a-u_{i-1}-u_i),
\]
which satisfies 
\[
g(u_{i-1}) = u_{i-1}^2-s^2 <0, \qquad
g(u_i) = u_i^2-s^2 >0,
\]
and 
\begin{align*}
&g(b)=2\lbar_i\d(u_i-b)(b-u_{i-1})\big(h'(b)-p'(b)\big)\ge0,\\
&g(c)=2\lbar_i\d(u_i-c)(c-u_{i-1})\big(h'(c)-p'(c)\big)\le0.
\end{align*}
We now claim that the function \(g\) is concave.
Since \(g'''(a)=24\lbar_i^2 \d>0\), it is enough to show 
\[
g''(u_i) = -4\lbar_i+2+12\lbar_i^2\d(u_i-u_{i-1}) <0.
\]
Then, we observe that \((u_i-u_{i-1}) \le (\r_*)^{-(i-1)} (u_0-u_1) \le (\r_*)^{-i} (\r_*+1) r s\) and that \(\lbar_i\) satisfies
\[
\frac{2-\sqrt{4-12 t^{-i}(t+1) r}}{6 t^{-i}(t+1) r}
< \lbar_i
< \frac{2+\sqrt{4-12 t^{-i}(t+1) r}}{6 t^{-i}(t+1) r},
\]
from which the following holds:
\[
12\lbar_i^2\d(u_i-u_{i-1})
\le12\lbar_i^2 \d s (\r_*)^{-i} (\r_*+1) r 
<6 \lbar_i^2 t^{-i}(t+1) r
<4\lbar_i-2.
\]
Therefore, the concavity of the function \(g\) contradicts to
\[
g(u_{i-1}) < 0, \qquad
g(b) \ge 0, \qquad
g(c) \le 0, \qquad
g(u_i) > 0, \qquad
(u_{i-1}<b<c<u_i).
\]
This completes the proof of \textit{Case 3} in Proposition \ref{prop:key}.
\end{pf}

\vspace{2mm}
\noindent\textbf{Acknowledgement.}
The first author is partially supported by National Science Foundation (DMS-2306258).
The second and third authors were supported by Samsung Science and Technology Foundation under Project Number SSTF-BA2102-01.
The fourth author is partially supported by National Science Foundation (DMS-2206218).
The first and fourth authors are partially supported by a SQuaRE at the American Institute of Mathematics.

\vspace{2mm}
\noindent\textbf{Declaration of competing interest.}
The authors declared that they have no conflict of interest to this work.

\vspace{2mm}
\noindent\textbf{Data availability statement.}
We do not analyze or generate any datasets, because our work proceeds within a theoretical and mathematical approach.

\bibliographystyle{plain}
\bibliography{reference}

\end{document}